\documentclass[a4paper]{article}%
\usepackage{amsmath}
\usepackage{amsfonts}
\usepackage{amssymb}
\usepackage{graphicx}%
\usepackage{float}
\usepackage{listings}
\usepackage{listings}
\usepackage{color}

\definecolor{miverde}{rgb}{0,0.6,0}
\definecolor{migris}{rgb}{0.5,0.5,0.5}
\definecolor{mimalva}{rgb}{0.58,0,0.82}
\definecolor{miazul}{rgb}{0.2,0.1,1}
\lstnewenvironment{listing}[1][]
{\lstset{#1}\pagebreak[0]}{\pagebreak[0}
\lstdefinestyle{consola}{basicstyle=\scriptsize\bf\ttfamily,background = \color{gray75},}
\providecommand{\U}[1]{\protect\rule{.1in}{.1in}}
\providecommand{\U}[1]{\protect\rule{.1in}{.1in}}
\newtheorem{theorem}{Theorem}

\newtheorem{definition}[theorem]{Definition}

\newtheorem{proposition}[theorem]{Proposition}
\newtheorem{remark}[theorem]{Remark}

\newenvironment{proof}[1][Proof]{\noindent\textbf{#1.} }{\ \rule{0.5em}{0.5em}}
\begin{document}

\title{\textbf{On regularization methods based on R\'enyi's pseudodistances for
sparse high-dimensional linear regression models}}
\author{E. Castilla$^{1}$, A. Ghosh$^{2},$ M. Jaenada$^{1}$ and L. Pardo$^{1}$\\$^{1}${\small Department of Statistics and O.R., Complutense University of
Madrid, 28040 Madrid, Spain}\\$^{2}${\small Indian Statistical Institute, Kolkata 700108, India}}
\date{}
\maketitle

\begin{abstract}
Several regularization methods have been considered over the last decade for sparse high-dimensional linear regression models, but the most common ones use the least square (quadratic) or likelihood loss and hence are not robust against data contamination. Some authors have overcome the problem of non-robustness by considering suitable loss function based on divergence measures (e.g., density power divergence, $\gamma-$divergence, etc.) instead of the quadratic loss. In this paper we shall consider a loss function based on the R\'enyi's pseudodistance jointly with non-concave penalties in order to simultaneously perform variable selection and get robust estimators of the parameters in a high-dimensional linear regression model of nonpolynomial dimensionality. The desired oracle properties of our proposed method are derived theoretically and its usefulness is illustustrated numerically through simulations and real data examples.

\end{abstract}

\bigskip\bigskip



\section{Introduction\label{sec1}}

We consider the high-dimensional linear regression model (LRM) given by
\begin{equation}
Y_{i}=\boldsymbol{X}_{i}^{T}\boldsymbol{\beta}+U_{i},\quad i=1,\dots,n,
\label{0.1}%
\end{equation}
where $\boldsymbol{X}_{i}=(X_{i1},..,X_{ip})^{T}$ are the explanatory
variables, $\boldsymbol{\beta=}(\beta_{1},..,\beta_{p})^{T}$ $\in
\mathbb{R}^{p}$ is the vector of unknown regression coefficients and $U_{i}$s
are random noise with $\boldsymbol{U}=\left(  U_{1},...,U_{n}\right)
\in\mathbb{R}^{n}$ being normally distributed with null mean vector and variance
covariance matrix $\sigma^{2}\boldsymbol{I}_{n}$. Assume that the
explanatory variables are stochastic in nature; in other words, $\left(
\boldsymbol{X}_{i}^{T},Y_{i}\right)  ,$ $i=1,...,n$ are independent and
identically distributed. Without loss of generality, we may assume that the
model does not have any intercept terms by mean-centering all the response and
covariates. We denote by $\boldsymbol{\mathbb{X}}$ the $(n\times
p)$-dimensional matrix $\boldsymbol{\mathbb{X}}=(\boldsymbol{X}_{1}%
,..,\boldsymbol{X}_{n})^{T}$. Therefore, we can write (\ref{0.1}) in a
matricial form by
\begin{equation}
\boldsymbol{Y}=\boldsymbol{\mathbb{X}}\boldsymbol{\beta}+\boldsymbol{U},
\label{0.2}%
\end{equation}
being $\boldsymbol{Y=}\left(  Y_{1},...,Y_{n}\right)  ^{T}.$ We shall assume,
in the context of sparse high-dimensional LRM, that the number of explanatory
variables, $p,$ is greater than the number of observations. More concretely, in this paper, we consider nonpolynomial dimensionality, i.e., $\log
p=O(n^{\alpha})$ for some $\alpha\in\left(  0,1\right) $; see Fan and Lv
(2010). In many applications most explanatory variables do not provide
relevant information to predict the response, i.e., most of the true regression coefficients are zero. In this situation we say that the
regression parameter $\boldsymbol{\beta}$ is sparse, in the sense that many of its elements are zero and the corresponding LRM is called ``sparse high-dimensional LRM''.

Regularization methods for sparse high-dimensional data analysis are
characterized by loss functions measuring data fits and penalty terms
constraining model parameters. In LRM, regularization estimates of the  parameter vector
$\left(  \boldsymbol{\beta},\sigma\right)  \in\mathbb{R}^{p+1}$ is obtained by minimizing
a criterion function or objective function of the form
\begin{equation}
Q_{n,\lambda}\left(  \boldsymbol{\beta},\sigma\right)  =L_{n}\left(
\boldsymbol{\beta},\sigma\right)  +%
{\textstyle\sum_{j=1}^{p}}
p_{\lambda_{n}}\left(  \left\vert \beta_{j}\right\vert \right)  , \label{0.21}%
\end{equation}
which consists of a data fit functional $L_{n}\left(  \boldsymbol{\beta}%
,\sigma\right)  $, called loss function, and a penalty function $%
{\textstyle\sum_{j=1}^{p}}
p_{\lambda_{n}}\left(  \left\vert \beta_{j}\right\vert \right),$
assessing the physical plausibility of $\boldsymbol{\beta}$. The loss function measures how well $\boldsymbol{\beta}$ fits the observed set of data; on the other hand, the penalty is used to
control the complexity of the fitted model in order to avoid overfitting. A regularization parameter $\lambda_{n}$ $\left(  \lambda_{n}\geq0\right)$  regulates the penalty. From a practical point of view, the regularization parameter is chosen using some information criterion, e.g., AIC or BIC, or sorts of
cross-validation. The former emphasizes the model's fit to the data, while the
latter is more focused on its predictive performance. If $L_{n}\left(
\boldsymbol{\beta},\sigma\right)  $ corresponds to the loss function
associated to an M-estimator,  the minimizers of an objective
function like (\ref{0.21}) are called ``penalized regression M-estimators''.
Such an estimator verifies the oracle properties, see Fan and Li (2001), if it
estimates zero components of the true parameter vector exactly as zero with
probability approaching one as sample size increases.

Let us consider the $l_{q}$ norm $\Vert\boldsymbol{\beta}\Vert_{q}=\left(
\sum_{j=1}^{p} |\beta_{j}|^{q} \right) ^{1/q}$. The most common data fit
functional is the quadratic loss function%
\begin{equation}
L_{n}\left(  \boldsymbol{\beta},\sigma\right)  =\frac{1}{n}\left\Vert
\boldsymbol{Y}-\boldsymbol{\mathbb{X}}\boldsymbol{\beta}\right\Vert _{2}^{2}.
\label{0.3}%
\end{equation}
If we consider jointly with (\ref{0.3}) the penalty function $%
{\textstyle\sum_{j=1}^{p}}
p_{\lambda_{n}}\left(  \left\vert \beta_{j}\right\vert \right)  =\lambda
\left\Vert \boldsymbol{\beta}\right\Vert _{q}^{q}$ for a given $\lambda$,
where $q>0$, its minimization leads us to Bridge estimators (Frank and
Friedman, 1993). For $q=2$, we get the Ridge estimator considered in Hoerl and
Kennard (1970), while for $q=1$, we get the well-known LASSO estimator
introduced by Tibshirani (1996). However, Zou (2006) provided some
examples where the LASSO is inconsistent for variable selection.
Estimators obtained using $l_{2}$ penalty function or smooth penalty
functions, in general, are unable to detect the null regression coefficients, see Fan and Li (2001). On the other hand, $l_{1}$ penalty function produces sparse estimators for the regressions parameters. Knight and Fu (2000) showed that, the estimators corresponding to a penalty function with $q<1$ have the oracle properties, but for $q=1$, the asymptotic distribution of the LASSO estimator corresponding to zero coefficients of the true parameter vector can put positive probability at zero. More details about
the previous regularization procedures can be seen in B\"{u}hlmann and van de Geer (2011) as well as in the reviews by Fan and Lv (2010) and Tibshirani (2011).

To address the problem of high false positives in LASSO, there have been several generalizations of it yielding consistent estimator of the active set under much weaker conditions. Some of the most popular are: the adaptive LASSO (Zou, 2006); the relaxed LASSO (related to the adaptive LASSO discussed by Meinshausen, 2007); the group LASSO (Yuan and Lin, 2006);\ Multi-step adaptive LASSO, considered in B\"{u}hlmann and Meier (2008); Dantzig selector (Candes
and Tao, 2007); Fused LASSO (Tibshirani et al., 2005); Graphical LASSO, studied in Yuan and Lin (2007) and Friedman et al. (2007); etc.

A further limitation of the estimators based on minimizing the objective function, $Q_{n,\lambda}\left(  \boldsymbol{\beta},\sigma\right)  ,$ with quadratic loss function is their lack of robustness with regard to outliers. Alfons et al. (2013) established that the breakpoint of the LASSO estimator is
$1/n$, i.e., only one single outlier can make the estimate completely
unreliable. Subsequently different procedures are developed for obtaining sparse estimators that limit the impact of contamination in the data. In general, these procedures rely on the intuition that a loss function yielding robust estimators in simple (classical) statistical set-up (Hampel et al., 1986) should also define robust
estimators when it is penalized by a deterministic function. More concretely, the idea is to replace the quadratic loss function by a loss function based on an M-estimator, i.e., to consider ``penalized regression M-estimators''. Let us briefly summarize the penalized M-estimators previously studied in the literature: Wang et al. (2007) considered the least absolute deviation (LAD) loss function, namely $L_{n}\left(  \boldsymbol{\beta},\sigma\right)  =\frac
{1}{n}\left\Vert \boldsymbol{Y}-\boldsymbol{\mathbb{X}}\boldsymbol{\beta
}\right\Vert _{1},$ jointly with $l_{1}$-penalty function (LAD-LASSO
estimators). These estimators are only resistant to the outliers in the response variable but not to the outliers in predictors. Arslan (2012) presented a weighted version of LAD-LASSO estimator that combine robust parameter estimation and variable selection simultaneously. Alfons et al. (2013) considered the least trimmed square (LTS) loss function given by
$L_{n}\left(  \boldsymbol{\beta},\sigma\right)  =\frac{1}{n}
{\textstyle\sum_{j=1}^{h}}
r_{\left(  i\right)  }^{2}(\boldsymbol{\beta}),$ with $r_{i}^{2}%
(\boldsymbol{\beta})$ =$\left( y_{i}-\boldsymbol{x}_{i}^{T}\boldsymbol{\beta
}\right)  ^{2}$ denoting squared residuals errors, $r_{\left(  1\right)  }^{2}(\boldsymbol{\beta}),...,r_{\left(  n\right)  }^{2}(\boldsymbol{\beta})$ being their order statistics and $h\leq n$ being the size of the subsample that is considered to consist of non-outliying observations. Combining the LTS with
LASSO penalty function we get the LTS-LASSO estimator. Alfons et al. (2013) established that it has a high breakdown point. Other results in relation to the LTS-LASSO estimator can be seen in Alfons et al. (2016) and Olleres et al. (2015). Li et al. (2011) considered a general class of loss functions of the form $L_{n}\left(
\boldsymbol{\beta},\sigma\right)  =\frac{1}{n}
{\textstyle\sum_{j=1}^{n}} \rho\left(  y_{i}-\boldsymbol{x}_{i}^{T}\boldsymbol{\beta}\right)$, for some
$\rho:\mathbb{R\rightarrow R},$ and penalty function $2\lambda
{\textstyle\sum_{j=1}^{p}} J(\beta_{j})$, for suitable $J:\mathbb{R\rightarrow R.}$ While LASSO and Ridge have a
quadratic loss function $\rho(x)=x^{2},$ LAD-LASSO use $\rho(x)=\left\vert x\right\vert.$ The penalty function of Ridge is quadratic $J(z)=z^{2}$, whereas LASSO and LAD-LASSO uses the $l_{1}-$penalty. Wang et al (2013) proposed the exponential loss function (ESL) to get the ESL-LASSO estimator.
Smucler and Yohai (2017) considered the $l_{1}$-penalized MM-estimators. Fan and Li (2001) considered the SCAD (smoothly clipped absolute deviatin) penalty function jointly the quadratic loss function; here we will also pay special attention to the SCAD penalty.

As pointed out in Avella-Medina (2017), only the papers of Alfons et al. (2013) and Wang et al. (2013) established formal robustness properties for their proposed regularized estimators. In Avella-Medina (2017) local robustness properties of general penalized M-estimators are studied on the basis of their influence functions (IF). The IF are obtained not only in the cases where the penalty function is twice differentiable but also for non-differentiable penalty functions. Avella-Medina and Ronchetti (2018) have studied a class of robust penalized M-estimators for sparse high-dimensional LRM establishing that the estimators satisfy the oracle properties and are stable in a neighborhood of the model.

The regression M-estimators based on minimum distance approach have
played an important role because it has been observed that they produce highly efficient robust inference under classical low-dimensional set-up. Under the high-dimensional regime, departing from the likelihood-based methods, Lozano et al. (2016) have first developed a penalized minimum distance criterion for robust and consistent estimation of sparse high-dimensional regression using
the $L_{2}$-distance. Zang et al. (2017) have then sparsified the density power divergence (DPD) loss (Basu et al., 1998; Ghosh and Basu, 2013) based regression, and Kawashima and Fujisawa (2017) have done the same for the $\gamma$-divergence loss function; but both of them are restricted to the $l_{1}$-penalty and LRM. Zhang et al. (2010) used loss functions based on Bregman divergences. Ghosh and Mujandar (2017) have combined the strengths of non-concave penalties (e.g., SCAD) and the DPD loss function to simultaneously perform variable selection and obtain robust estimates of $\boldsymbol{\beta}$ under sparse high-dimensional LRM with general location-scale errors. They ensured robustness against contamination of infinitesimal magnitude using influence function analysis, and established theoretical consistency and oracle properties of their proposed estimator under nonpolynomial dimensionality.

The R\'{e}nyi's pseudodistance (RP) was introduced for the first time in Jones et al. (2001) and later additional properties were studied in Broniatowski et al. (2012).
In this paper, we shall consider a loss function based on RP, to which we call RP loss function, jointly with
non-concave penalties in order to simultaneously perform variable selection and to obtain robust estimators of $\boldsymbol{\beta}$ and $\sigma$ in high-dimensional LRM with nonpolynomial dimensionality. 
 
 This RP loss function was earlier considered for a low-dimensional LRM with $p<n$ in Castilla et al. (2020) establishing their nice robust properties. Here we present a nonconcave penalized version of the PR loss function for the LRM. This method achieves simultaneously robust parameter estimation and variable selection in an ultra-high dimensional setting. 
 It is worthwhile to note that Kawashima and Fujisawa (2017) considered the $\gamma-$divergence loss function, which has the same expression as the RP loss function for the LRM, but they only considered the LASSO penalty function (with no theory).
 Considering nonconcave penalties is the most important (empirical) difference with respect to Kawashima and Fujisama 's work, where only LASSO penalty was contemplate. Additionally, we also develop detailed theory of the proposed estimators, proving their oracle model selection property as well as consistency and asymptotic normality of the non-zero estimates. Performances of the proposed estimators are illustrated and compared with the state-of-the-art procedures via extensive simulation studies and interesting real data examples.  For brevity, all the proofs are presented in the Online Supplement along with additional numerical results. The R codes for the computation of the proposed estimator is also provided in the Online Supplement enabling any practitioner to apply this procedure in future researches.

\section{The proposed RP based regularization method in sparse high-dimensional LRM}

\subsection{The RP loss function}

Based on (\ref{0.1}), we define $U_{i}=Y_{i}-\boldsymbol{X}_{i}^{T}%
\boldsymbol{\beta},$ for $i=1,...,n$ . Let $G_{n}^{\boldsymbol{\beta}}%
(u)=\sum_{i=1}^{n}\frac{1}{n}\operatorname{I}(u_{i}\leq u)$ is the empirical distribution function corresponding to the random sample $u_{1},...,u_{n}$ from
$U_{1},...,U_{n}$; here $I(\cdot)$ denotes the indicator function. The probability mass function associated to $G_{n}%
^{\boldsymbol{\beta}}(u)$ is given by $p_{n}^{\boldsymbol{\beta}}%
(u)=G_{n}^{\boldsymbol{\beta}}(u)-G_{n}^{\boldsymbol{\beta}}(u^{-})=\sum_{i=1}^{n}\frac{1}{n}\operatorname{I}(U_{i}=u).$ On the other hand, $U_{i}$ is normally distributed with mean zero and variance $\sigma^{2}.$
Therefore, the density function for $U_{i}$ is given by
\[ f_{\boldsymbol{\beta},\sigma}(u)=f_{\boldsymbol{\beta},\sigma}%
(y-\boldsymbol{x}^{T}\boldsymbol{\beta})=\frac{1}{\sqrt{2\pi}\sigma}%
\exp\left(  -\frac{1}{2}\left(  \frac{y-\boldsymbol{x}^{T}\boldsymbol{\beta}%
}{\sigma}\right)  ^{2}\right)  .
\]
If we denote by $P_{\boldsymbol{\beta,}\sigma}$ the measure of probability associated to the density function $f_{\boldsymbol{\beta},\sigma}(u)$ and by
$P_{n}^{\boldsymbol{\beta}}$ the measure of probability associated to the empirical distribution function $G_{n}^{\boldsymbol{\beta}}(u),$ the RP between $P_{\boldsymbol{\beta,\sigma}}$ and $P_{n}^{\boldsymbol{\beta}}$, in accordance with Formula (7) in Broniatowski et al. (2012), can be written by
\[
\mathcal{R}_{\alpha}\left(  P_{\boldsymbol{\beta},\sigma},P_{n}%
^{\boldsymbol{\beta}}\right)  =\frac{1}{\alpha+1}\log\int f_{\boldsymbol{\beta
},\sigma}(u)^{\alpha}dP_{\boldsymbol{\beta},\sigma}\left(  u\right)  +\frac
{1}{\alpha(\alpha+1)}\log\int p_{n}^{\boldsymbol{\beta}}(u)^{\alpha}%
dP_{n}^{\boldsymbol{\beta}}(u)-\frac{1}{\alpha}\log\int f_{\boldsymbol{\beta
},\sigma}(u)^{\alpha}dP_{n}^{\boldsymbol{\beta}}(u),
\]
where $\alpha$ is a non-negative tuning parameter controlling  the compromises
between efficiency and robustness.

Taking into account that
\[
\int p_{n}^{\boldsymbol{\beta}}(u)^{\alpha}dG_{n}^{\boldsymbol{\beta}}%
(u)=\int\left(  \frac{1}{n}\sum_{i=1}^{n}\operatorname{I}(u_{i}=u)\right)
^{\alpha}dG_{n}^{\boldsymbol{\beta}}(u)= \frac{1}{n}\sum_{j=1}^{n}\left(\sum_{i=1}^{n}\left(
\frac{1}{n}\operatorname{I}(u_{i}=u_{j})\right)  ^{\alpha} \right)=\left(  \frac{1}%
{n}\right)  ^{\alpha},
\]
we have%
\[
\mathcal{R}_{\alpha}\left(  P_{\boldsymbol{\beta},\sigma},P_{n}%
^{\boldsymbol{\beta}}\right)  =\frac{1}{\alpha+1}\log\int f_{\boldsymbol{\beta
},\sigma}(u)^{\alpha+1}du+\frac{1}{(\alpha+1)}\log\left(\frac{1}{n}\right)^{\alpha}-\frac{1}{\alpha
}\log\frac{1}{n}\sum_{i=1}^{n}f_{\boldsymbol{\beta},\sigma}(u_{i})^{\alpha}%
\]
for $\alpha>0.$ For $\alpha=0$ it is given by the limit as
\[
R_{0}\left(  P_{\boldsymbol{\beta},\sigma},P_{n}^{\boldsymbol{\beta}}\right)
=\lim_{\alpha\downarrow0}R_{\alpha}\left(  P_{\boldsymbol{\beta},\sigma}%
,P_{n}^{\boldsymbol{\beta}}\right)  =\log\frac{1}{n}-\frac{1}{n}\sum_{i=1}%
^{n}\log f_{\boldsymbol{\beta},\sigma}(u_{i}).
\]
 We are going to simplify the expression of $\mathcal{R}_{\alpha}\left(  P_{\boldsymbol{\beta
},\sigma},P_{n}^{\boldsymbol{\beta}}\right)  .$ It is immediate to see that,%

\[
\frac{1}{\alpha+1}\log\int f_{\beta,\sigma}(u)^{\alpha+1}du=\frac{1}{\alpha
+1}\log\left\{  \left(  \frac{1}{\sqrt{2\pi}\sigma}\right)  ^{\alpha}\frac
{1}{\sqrt{\alpha+1}}\right\}
\]
and
\[
\frac{1}{\alpha}\log\frac{1}{n}\sum_{i=1}^{n}f_{\boldsymbol{\beta},\sigma
}(u_{i})^{\alpha}=\frac{1}{\alpha}\log\left\{  \frac{1}{n}\sum_{i=1}%
^{n}\left(  \frac{1}{\sqrt{2\pi}\sigma}\right)  ^{\alpha}\exp\left(
-\frac{\alpha}{2}\left(  \frac{y_{i}-\boldsymbol{x}_{i}^{T}\boldsymbol{\beta}%
}{\sigma}\right)  ^{2}\right)  \right\}  .
\]
Therefore we have,
\begin{equation} \label{RPL}
\begin{aligned}
\mathcal{R}_{\alpha}\left(  P_{\boldsymbol{\beta},\sigma},P_{n}%
^{\boldsymbol{\beta}}\right)   &  =\frac{1}{\alpha+1}\log\left\{  \left(
\frac{1}{\sqrt{2\pi}\sigma}\right)  ^{\alpha}\frac{1}{\sqrt{\alpha+1}%
}\right\}  +\frac{1}{\alpha(\alpha+1)}\log\left(  \frac{1}{n}\right)
^{\alpha}\\
&  -\frac{1}{\alpha}\log\left\{  \frac{1}{n}\sum_{i=1}^{n}\left(  \frac
{1}{\sqrt{2\pi}\sigma}\right)  ^{\alpha}\exp\left(  -\frac{\alpha}{2}\left(
\frac{y_{i}-\boldsymbol{x}_{i}^{T}\boldsymbol{\beta}}{\sigma}\right)
^{2}\right)  \right\}  .
\end{aligned}
\end{equation}
An estimator for $\ \boldsymbol{\beta}$ and $\sigma$ can be defined by
minimizing $\mathcal{R}_{\alpha}\left(
P_{\boldsymbol{\beta},\sigma},P_{n}^{\boldsymbol{\beta}}\right)$ with respect to $\boldsymbol{\beta}$ and $\sigma$, i.e. for
$\alpha>0$,%

\begin{equation}
(\widehat{\boldsymbol{\beta}}^{\alpha},\widehat{\sigma}^{\alpha}%
)=\operatorname{arg}\min_{\boldsymbol{\beta},\sigma}\left\{  -\frac{1}%
{n}\sigma^{\frac{-\alpha}{\alpha+1}}\sum_{i=1}^{n}\exp\left(  -\frac{\alpha
}{2}\left(  \frac{y_{i}-\boldsymbol{x}_{i}^{T}\boldsymbol{\beta}}{\sigma
}\right)  ^{2}\right)  \right\}  , \label{0.4}%
\end{equation}
and for $\alpha=0$ we get the maximum likelihood estimator (MLE).

\begin{remark}
\label{remark1} If we consider the loss function given by the RP between the
density function associated to our model, $f_{\beta,\sigma}(y|\boldsymbol{x}%
),$ and the true density function for the model $g(y|\boldsymbol{x}),$ the
loss function associated with the RP is
\[
L_{Y|\boldsymbol{X}}^{\alpha}(\boldsymbol{\beta},\sigma)=-\int f_{\beta
,\sigma}(y|\boldsymbol{x})^{\alpha}g(y|\boldsymbol{x})dy\left[  \int
f_{\beta,\sigma}(y|\boldsymbol{x})^{\alpha+1}dy\right]  ^{\frac{-\alpha
}{\alpha+1}}.
\]
If we assume that the distribution function of the random variable
$\boldsymbol{X}$ is given by $G(\boldsymbol{x}),$ under some regularity conditions, we have
\begin{equation}
\frac{1}{n}\sum_{i=1}^{n}h(x_{i})\underset{n\rightarrow\infty
}{\overset{P}{\rightarrow}}\int h(\boldsymbol{x})dG(\boldsymbol{x}%
)=\mathbb{E}_{\boldsymbol{X}}[h(\boldsymbol{X})]. \label{ec aprox esperanza}%
\end{equation}
Ghosh and Basu (2013) proposed, on the basis of the density power divergence (DPD), to
minimize the expectation of the DPD expression between
$g(y|\boldsymbol{x})$ and $f_{\boldsymbol{\beta},\sigma}(y|\boldsymbol{x})$.
In our situation we can consider the same but using the RP instead of DPD, i.e,
\begin{equation}
\mathbb{E}_{\boldsymbol{X}}\left[  L_{Y|\boldsymbol{X}}^{\alpha}%
(\boldsymbol{\beta},\sigma)\right]  =-\left[  \int\int f_{\beta,\sigma
}(y|\boldsymbol{x})^{\alpha+1}dyg(\boldsymbol{x})d\boldsymbol{x}\right]^{\frac{-\alpha}{\alpha+1}}\int\int f_{\beta,\sigma}(y|\boldsymbol{x}%
)^{\alpha}g(y,\boldsymbol{x})dyd\boldsymbol{x}.\label{eq:1}%
\end{equation}
where $g(y,\boldsymbol{x})$ denotes the  joint density function. Now, expression (\ref{eq:1}) can be approximated by
\[
-\frac{1}{n}\left(  \sum_{i=1}^{n}\int f_{\beta,\sigma}(y|\boldsymbol{x_{i}%
})^{\alpha+1}dy\right)  ^{\frac{-\alpha}{\alpha+1}}\left(  \frac{1}{n}%
\sum_{i=1}^{n}f_{\beta,\sigma}(y_{i}|\boldsymbol{x_{i}})^{\alpha}\right)
=-\left(  \int f_{\beta,\sigma}(y|\boldsymbol{x})^{\alpha+1}dy\right)
^{\frac{-\alpha}{\alpha+1}}\left(  \frac{1}{n}\sum_{i=1}^{n}f_{\beta,\sigma
}(y_{i}|\boldsymbol{x_{i}})^{\alpha}\right)  .
\]

\end{remark}

\bigskip

Based on (\ref{0.4}) and Remark \ref{remark1}, we can consider the loss
function for the LRM based on RP by%

\begin{equation}
L_{n}^{\alpha}(\boldsymbol{\beta},\sigma)=\left\{
\begin{array}
[c]{lcc}%
\frac{1}{n}\sum_{i=1}^{n}-\sigma^{\frac{-\alpha}{\alpha+1}}\exp\left(
-\frac{\alpha}{2}\left(  \frac{y_{i}-\boldsymbol{x}_{i}^{T}\boldsymbol{\beta}%
}{\sigma}\right)  ^{2}\right)  & \text{if} & \alpha>0;\\
\log(\sigma\sqrt{2\pi})+\frac{1}{n}\sum_{i=1}^{n}\frac{1}{2}\left(
\frac{y_{i}-\boldsymbol{x}_{i}^{T}\boldsymbol{\beta}}{\sigma}\right)  ^{2} &
\text{if} & \alpha=0.
\end{array}
\right.  \label{0.8}
\end{equation}
Again, we can observe that for $\alpha=0$, $L_{n}^{\alpha}(\boldsymbol{\beta
},\sigma)$ coincides with the negative loglikelihood function. Therefore, the MLE is a particular case of the minimum RP estimator.

Based on (\ref{0.8}) the estimating equations are given for $\alpha>0$ by%

\[
\left\{
\begin{array}
[c]{ll}%
\sum_{i=1}^{n}\exp\left(  \frac{-\alpha}{2\sigma^{2}}(y_{i}-\boldsymbol{x}%
_{i}^{T}\boldsymbol{\beta})^{2}\right)  \left(  \frac{y_{i}-\boldsymbol{x}%
_{i}^{T}\boldsymbol{\beta}}{\sigma}\right)  \boldsymbol{x}_{i} &
=\mathbf{0}_{p},\\
\sum_{i=1}^{n}\exp\left(  \frac{-\alpha}{2\sigma^{2}}\left(  y_{i}%
-\boldsymbol{x}_{i}^{T}\boldsymbol{\beta}\right)  ^{2}\right)  \left[  \left(
\frac{y_{i}-\boldsymbol{x}_{i}^{T}\boldsymbol{\beta}}{\sigma}\right)
^{2}-\frac{1}{\alpha+1}\right]  & =0,
\end{array}
\right.
\]
and for $\alpha=0$%

\[
\left\{
\begin{array}
[c]{ll}%
\sum_{i=1}^{n}\left(  \frac{y_{i}-\boldsymbol{x}_{i}^{T}\boldsymbol{\beta}%
}{\sigma}\right)  \boldsymbol{x}_{i} & =\boldsymbol{0}_{p},\\
\sum_{i=1}^{n}-\frac{1}{\sigma}+\frac{1}{\sigma}\left(  \frac{y_{i}%
-\boldsymbol{x}_{i}^{T}\boldsymbol{\beta}}{\sigma}\right)  ^{2} & =0.
\end{array}
\right.
\]
It is clear that the estimating equations of the minimum RP estimator, for
$\alpha>0,$ can be written as%

\[
\sum_{i=1}^{n}\boldsymbol{\psi}_{\alpha}(\boldsymbol{x}_{i},y_{i}%
,\boldsymbol{\beta},\sigma)=\boldsymbol{0}_{p+1},
\]
with
\begin{equation}
\boldsymbol{\psi}_{\alpha}(\boldsymbol{x},y,\boldsymbol{\beta},\sigma)=\left(
\psi_{\alpha,1}\left(  \boldsymbol{x},y,\boldsymbol{\beta},\sigma\right)
,\psi_{\alpha,2}\left(  \boldsymbol{x},y,\boldsymbol{\beta},\sigma\right)
\right)  =\left(  \phi_{\alpha,1}\left(  \frac{y-\boldsymbol{x}^{T}%
\boldsymbol{\beta}}{\sigma}\right)  \boldsymbol{x},\phi_{\alpha,2}\left(
\frac{y-\boldsymbol{x}^{T}\boldsymbol{\beta}}{\sigma}\right)  \right)  ,
\label{0.9}%
\end{equation}
where \begin{equation} \label{phi1}
	\phi_{\alpha,1}(u)=u\exp\left(\frac{-\alpha}{2}u^{2}\right)
\end{equation} and \begin{equation} \label{phi2}
\phi_{\alpha
	,2}(u)=\left(u^{2}-\frac{1}{\alpha+1}\right)\exp\left(\frac{-\alpha}{2}u^{2}\right).
\end{equation}
Thus, the minimum RP estimator is an M-estimator and its asymptotic distribution can be obtained on the basis of the asymptotic distribution of an M-estimator (see Maronna, et al., 2006). More details about the asymptotic distribution can be
found in Broniatosky et al. (2012).

\subsection{Non-concave penalty functions}

Several penalty functions have been considered in regularization methods for
high-dimensional LRM. In addition to the $l_{i}$-penalties $(i=1,2)$
associated to LASSO and Ridge methods, respectively, we can define the $l_{0}$-penalty as $\ p_{\lambda}\left(  \left\vert \beta_{j}\right\vert \right) =\lambda\operatorname{I}(\beta_{j}\neq0)$, or consider the $l_{q}$-penalty functions given by $p_{\lambda}\left(  \left\vert \beta_{j}\right\vert \right)
=\lambda\left\vert \beta_{j}\right\vert ^{q}$, which have been examined for this purpose over the choices $0<q<2$. Some combinations of such penalties are also used; for example, the combination of $l_{1}$ and $l_{2}$ penalties are referred to as the elastic net penalty. The $l_{1}$ penalty is increasing and therefore imposes larger penalty for larger $\left\vert \beta_{j}\right\vert
$; hence it induces biased estimator for $\boldsymbol{\beta}$ even when the true $\boldsymbol{\beta}$ is sufficiently large. To remedy this flaw, the nonconcave penalties, such as SCAD (smoothly clipped absolute deviation) considered by Fan (1997) and Fan and Li (2001) and MCP (minimax concave penalty) introduced by Zhang (2010), transmit from $l_{1}$ function to constant function as $\boldsymbol{\beta}$ increases, in the sense that
$p_{\lambda}\left(  \left\vert \beta_{j}\right\vert \right)  $ is an absolute linear function around the $0$ and it becomes a constant when $\left\vert \beta_{j}\right\vert $ is larger than some threshold.

Fan and Li (2001) advocated three characteristics properties of a
\textquotedblleft good" penalty function, namely
\textit{Unbiasedness, Sparsity} and \textit{Continuity}. It has been verified that the
$l_{q}$-penalty with $q>1$ does not satisfy the sparsity condition, whereas the $l_{1}$-penalty does not satisfy the unbiasedness condition; also the concave $l_{q}$-penalty having $0\leq q<1$ does not satisfy the continuity condition. In other words, none of the $l_{q}$-penalties satisfy the three conditions simultaneously. The SCAD penalty verifies the three properties and the MCP penalty verifies the unbiasedness and sparsity but not continuity.

In this paper we shall consider non-concave penalties $p_{\lambda}(.)$ that admits a decomposition of the form
\begin{equation}
p_{\lambda}(|s|)=\tilde{J}_{\lambda}(|s|)+\lambda|s|, \label{Des}%
\end{equation}
where $\tilde{J}_{\lambda}(|s|)$ is a differentiable concave function. 
It is immediate to see that the penalties SCAD and MCP verify the
decomposition (\ref{Des}) with the function $\tilde{J}_{\lambda}(|s|)$ being given, respectively, by

\[
\tilde{J}_{\lambda}(|\beta_{j}|)=\left\{
\begin{array}
[c]{lcc}%
-\frac{\beta_{j}^{2}-2\lambda|\beta_{j}|+\lambda^{2}}{2(a-1)} & \text{if } &
\lambda\leq\left\vert \beta_{j}\right\vert <a\lambda;\\
\frac{(a+1)\lambda^{2}}{2}-\lambda|\beta_{j}| & \text{if} & a\lambda
<\left\vert \beta_{j}\right\vert
\end{array}
\right.  \text{ and }\tilde{J}_{\lambda}(|\beta_{j}|)=\left\{
\begin{array}
[c]{lcc}%
\frac{\beta_{j}^{2}}{2a} & \text{if} & 0\leq\left\vert \beta_{j}\right\vert
<a\lambda;\\
\frac{a\lambda^{2}}{2}-\lambda|\beta_{j}| & \text{if} & a\lambda<\left\vert
\beta_{j}\right\vert
\end{array}
\right.  .
\]

\subsection{ The proposed estimation procedure}	
The criterion function for the nonconcave penalized RP estimator has the form
\begin{equation}
Q_{n,\lambda}^{\alpha}(\boldsymbol{\beta},\sigma)=L_{n}^{\alpha}%
(\boldsymbol{\beta},\sigma)+\sum_{j=1}^{p}p_{\lambda}(|\beta_{j}|), \label{CF}%
\end{equation}
with  $L_{n}^{\alpha}(\boldsymbol{\beta},\sigma)$ the loss function and $p_{\lambda}(.)$ any nonconcave penalty function. Using the expression of $L_{n}^{\alpha}
(\boldsymbol{\beta},\sigma)$ in (\ref{0.8}), $Q_{n,\lambda}^{\alpha}(\boldsymbol{\beta},\sigma)$ is given by
%

\begin{equation}
Q_{n,\lambda}^{\alpha}(\boldsymbol{\beta},\sigma)=\left\{
\begin{array}
[c]{lcc}%
\frac{1}{n}\sum_{i=1}^{n}-\sigma^{\frac{-\alpha}{\alpha+1}}\exp\left(
-\frac{\alpha}{2}\left(  \frac{y_{i}-\boldsymbol{x}_{i}^{T}\boldsymbol{\beta}%
}{\sigma}\right)  ^{2}\right)  +\sum_{j=1}^{p}p_{\lambda}(|\beta_{j}|) &
\text{if} & \alpha>0;\\
\log(\sigma\sqrt{2\pi})+\frac{1}{2n}\sum_{i=1}^{n}\left(  \frac{y_{i}%
-\boldsymbol{x}_{i}^{T}\boldsymbol{\beta}}{\sigma}\right)  ^{2}+\sum_{j=1}%
^{p}p_{\lambda}(|\beta_{j}|) & \text{if} & \alpha=0.
\end{array}
\right.  \label{0.Qn}%
\end{equation}

In the following, the estimator obtained by minimizing the objective function (\ref{0.Qn}) with respect to $\boldsymbol{\beta}$ and $\sigma$ will be called Minimum
Non-concave Penalized RP estimator (MNPRPE).

We could also define, in the same way, the statistical functional corresponding to  the MNPRPE.
For this purpose, let $G$ be the true distribution function of the random vector $\left(
Y,\boldsymbol{X}\right)  $ and $g(y,\boldsymbol{x})$ the corresponding density
function, which can be expressed by $g(y,\boldsymbol{x})=g(y/\boldsymbol{x}%
)g(\boldsymbol{x})$. Given a random sample $(y_{1},\boldsymbol{x}%
_{1}),...,(y_{n},\boldsymbol{x}_{n})$ from $\ \left(  Y,\boldsymbol{X}\right)
$ we shall denote by \ $G_{n}(y,\boldsymbol{x})=\sum_{i=1}^{n}\frac{1}%
{n}\operatorname{I}(y_{i}\leq y,\boldsymbol{x}_{i}\leq\boldsymbol{x})$ its
empirical distribution function. Here, the inequality $\boldsymbol{x}_{i}%
\leq\boldsymbol{x}$ refers to the vector ordering in $\mathbb{R}^{p}$. We
then define the MNPRPE  functional, $\boldsymbol{T}_{\alpha}(G)$, at the
true joint distribution function, $G$, as the minimizer of
\begin{equation}
Q_{\lambda}^{\alpha}(\boldsymbol{\beta},\sigma)=L^{\alpha}(\boldsymbol{\beta
},\sigma)+ \boldsymbol{1}^T\tilde{\boldsymbol{p}}_{\lambda}(\boldsymbol{\beta}), \label{Avella}%
\end{equation}
with
\[
L^{\alpha}(\boldsymbol{\beta},\sigma)=\int-\left(  \frac{1}{\sigma}\right)
^{\frac{\alpha}{\alpha+1}}\exp\left(  \frac{-\alpha}{2\sigma^{2}%
}(y-\boldsymbol{x}^{T}\boldsymbol{\beta})^{2}\right)  g(y,\boldsymbol{x}%
)dyd\boldsymbol{x} =\int L^{\ast\alpha}(\boldsymbol{\beta},\sigma
)dG(y,\boldsymbol{x})
\]
and $\tilde{\boldsymbol{p}}_{\lambda}(\boldsymbol{\beta})=\left(  p_{\lambda
}(\beta_{1}),..,p_{\lambda}(\beta_{p})\right) ^{T} $ the penalty function.
We denote the resulting penalized M-estimator as $T_{\alpha}(G)=(\boldsymbol{\beta}_{\ast},\sigma_{\ast})^{T}$, with
$\boldsymbol{\beta}_{\ast}\in\mathbb{R}^{p}$ and $\sigma_{\ast}\in\mathbb{R}$,
and the MNPRPE will be \ $\boldsymbol{T}_{\alpha}(G_{n})$ with $$\boldsymbol{T}%
_{\alpha}(G_{n})\underset{n\rightarrow\infty}{\overset{P}{\rightarrow}%
}\boldsymbol{T}_{\alpha}(G).$$

\section{Influence function of the MNPRPE}

We compute the IF of the MNPRPE, following the notation of
Avella-Medina (2017), depending on  whether the penalty function is twice
differentiable or not. For example, the $l_{2}-$penalty function is twice
differentiable but $l_{1}$, SCAD and MCP penalty functions are not twice
differentiable. We pay special attention to these last two non-concave
penalties.  We follow the same steps as in Section 3 in Ghosh and Majunder (2020). Note that equality (\ref{Avella}) is equivalent to equation (1) in Avella-Medina (2017) with $L^{\ast\alpha}(\boldsymbol{\beta},\sigma)=L(\boldsymbol{Z}%
,\boldsymbol{\theta}).$
Then, the IF of the functional $\boldsymbol{T}_{\alpha}(G)$, corresponding to the MNPRPE, is the
Gateaux derivative given by (Hampel, 1974)
\[
IF((y_{t},\boldsymbol{x}_{t}),G,\boldsymbol{T}_{\alpha})=\lim_{\varepsilon
\rightarrow0}\frac{\boldsymbol{T}_{\alpha}(G_{\varepsilon})-\boldsymbol{T}%
_{\alpha}(G)}{\varepsilon},
\]
where $G_{\varepsilon}=(1-\varepsilon)G+\varepsilon\Delta_{(y_{t}%
,\boldsymbol{x}_{t})}$ being $\varepsilon$ the contamination proportion and
$\Delta_{(y_{t},\boldsymbol{x}_{t})}$ the distribution that
assigns mass $1$ at point $(y_{t},\boldsymbol{x}_{t})$ and $0$ elsewhere. Clearly, the IF describes the effect of an infinitesimal contamination, at the point $(y_{t},\boldsymbol{x}_{t}),$ on the estimate, standardized by the mass of contamination.

\subsection{Twice differentiable functions}

In case we assume that the penalty function $\tilde{\boldsymbol{p}}_{\lambda
}(\boldsymbol{\beta})=\left( p_{\lambda}(\beta_{1}),..,p_{\lambda}(\beta
_{p})\right) ^{T} $ is twice differentiable, we shall use Lemma 1 in Avella
Medina (2017) in order to get the IF of the MNPRPE functional.

First note that, denoting $\boldsymbol{\Psi}_{\alpha}(\boldsymbol{\beta},\sigma)=\nabla
L_{\alpha}^{\ast}(\boldsymbol{\beta},\sigma)$, with $\nabla$ being the gradient with respect to $(\boldsymbol{\beta}, \sigma)$, we have%

\begin{equation}
\boldsymbol{\Psi}_{\alpha}(\boldsymbol{\beta},\sigma) 
  =-\alpha\sigma^{-\frac{2\alpha+1}{\alpha+1}}\left(
\begin{array}
[c]{c}%
\phi_{1,\alpha}\left(  \frac{y-\boldsymbol{x}^{T}\boldsymbol{\beta}}{\sigma
}\right)  \boldsymbol{x}\\
\phi_{2,\alpha}\left(  \frac{y-\boldsymbol{x}^{T}\boldsymbol{\beta}}{\sigma
}\right)
\end{array}
\right)  , \label{psi}%
\end{equation}
where $\phi_{1,\alpha}(u)$ and
$\phi_{2,\alpha}(u)$ are as defined in Equations (\ref{phi1}) and (\ref{phi2}), respectively.
On the other hand, let us denote $\tilde{\boldsymbol{p}}_{\lambda}^{\ast
}(\boldsymbol{\beta})=\left(  p_{\lambda}^{\prime}(\beta_{1}),..,p_{\lambda
}^{\prime}(\beta_{p})\right)  ^{T}$. The Jacobian matrix associated to the
penalty term is $\nabla\tilde{\boldsymbol{p}}_{\lambda}(\beta_{j}%
)=diag(\tilde{\boldsymbol{p}}_{\lambda}^{\ast}(\boldsymbol{\beta}),0).$ The
estimating equations associated to the functional $\boldsymbol{T}_{\alpha}(G)$ are%

\[
\left\{
\begin{split}
-\alpha\sigma^{-\frac{2\alpha+1}{\alpha+1}}\int\phi_{1,\alpha}\left(
\frac{y-\boldsymbol{x}^{T}\boldsymbol{\beta}}{\sigma}\right)  \boldsymbol{x}%
dG(y,\boldsymbol{x})+diag(\tilde{\boldsymbol{p}}_{\lambda}^{\ast
}(\boldsymbol{\beta}))  &  =\boldsymbol{0}_{p},\\
-\alpha\sigma^{-\frac{2\alpha+1}{\alpha+1}}\int\phi_{2,\alpha}\left(
\frac{y-\boldsymbol{x}^{T}\boldsymbol{\beta}}{\sigma}\right)
dG(y,\boldsymbol{x})  &  =0.
\end{split}
\right.
\]

Now, using Lemma 1 in Avella-Medina (2017), we have the following result:

\begin{theorem}
\label{thm421} Let $\tilde{\boldsymbol{p}}_{\lambda}(s)$ be twice
differentiable in $s$. We denote,
\[
\boldsymbol{J}_{\alpha}(G;\boldsymbol{\beta},\sigma)=E_{Y,\boldsymbol{X}%
}\left[  \nabla\boldsymbol{\Psi}_{\alpha}(\boldsymbol{\beta},\sigma)\right]
=\int\nabla\boldsymbol{\Psi}_{\alpha}(\boldsymbol{\beta},\sigma
)dG(y,\boldsymbol{x}),
\]
where $\boldsymbol{\Psi}_{\alpha}(\boldsymbol{\beta},\sigma)$ was defined in
(\ref{psi}), $\boldsymbol{T}_{\alpha}(G)=$\ $(\boldsymbol{\beta}_{\ast}%
,\sigma_{\ast})^{T}$ and $\tilde{\boldsymbol{p}}_{\lambda}^{\ast\ast
}(\boldsymbol{\beta})=(p_{\lambda}^{\prime\prime}(\beta_{1}),..,p_{\lambda
}^{\prime\prime}(\beta_{p})).$ If the matrix $\boldsymbol{J}_{\alpha}^{\ast
}(G;\boldsymbol{\beta},\sigma)=\boldsymbol{J}_{\alpha}(G;\boldsymbol{\beta
},\sigma)+diag(\tilde{\boldsymbol{p}}_{\lambda}^{\ast\ast}(\boldsymbol{\beta
}),0)$ is invertible at $(\boldsymbol{\beta}_{\ast},\sigma_{\ast})$, the IF
associated to the MNPRPE exists and its expression is given by
\[
\operatorname{IF}\left(  (y_{t},\boldsymbol{x}_{t}),\boldsymbol{T}_{\alpha
},G\right)  =-\boldsymbol{J}_{\alpha}^{\ast}\left(  G,\left(
\boldsymbol{\beta}_{\ast},\sigma_{\ast}\right)  \right)  ^{-1}\left(
\begin{array}
[c]{c}%
-\alpha\sigma_{\ast}^{-\frac{2\alpha+1}{\alpha+1}}\phi_{1,\alpha}\left(
\frac{y-\boldsymbol{x}^{T}\boldsymbol{\beta}_{\ast}}{\sigma\ast}\right)
\boldsymbol{x}+\tilde{\boldsymbol{p}}_{\lambda}^{\ast}(\boldsymbol{\beta
}_{\ast})\\
-\alpha\sigma_{\ast}^{-\frac{2\alpha+1}{\alpha+1}}\phi_{2,\alpha}\left(
\frac{y-\boldsymbol{x}^{T}\boldsymbol{\beta}_{\ast}}{\sigma\ast}\right)
\end{array}
\right)  .
\]

\end{theorem}

\begin{remark}
If we assume that there exist $\boldsymbol{\beta}_{0}$ and $\sigma_{0}$ so
that the conditional density of $Y$ given $\boldsymbol{X}=\boldsymbol{x}$,
$g(y/\boldsymbol{x})$, belongs to the LRM with parameters $\boldsymbol{\beta
}_{0}$ and $\sigma_{0}$; i.e., we assume that $\boldsymbol{\beta}_{0}$ and
$\sigma_{0}$ are the true value of the parameters, we have%

\[
\boldsymbol{J}_{\alpha}(G,\left(  \boldsymbol{\beta}_{0},\sigma_{0}\right)
)=E_{\boldsymbol{X}}E_{Y/\boldsymbol{X}}\left[  \nabla\boldsymbol{\Psi
}_{\alpha}(\boldsymbol{\beta}_{0},\sigma_{0})\right]  =-\alpha\sigma
_{0}^{-\frac{2\alpha+1}{\alpha+1}-1}\left[
\begin{matrix}
\frac{-1}{(\alpha+1)^{\frac{3}{2}}}\mathbb{E}_{\boldsymbol{X}}[\boldsymbol{X}%
\boldsymbol{X}^{T}] & \boldsymbol{0}\\
\boldsymbol{0} & \frac{-2}{(\alpha+1)^{\frac{5}{2}}}%
\end{matrix}
\right]  .
\]
For brevity, the computation of the above matrix $\boldsymbol{J}%
_{\alpha}(G,\left(  \boldsymbol{\beta},\sigma\right)  )$ is presented in the Online Supplement (Section 1).
\end{remark}

\subsection{Non-concave penalty functions}

Fan and Li (2001) stated that a desirable property of the penalty function is
not to be differentiable at zero. This property is satisfied by the SCAD and
MCP penalties. If the penalty function is not differentiable, the conditions of Theorem \ref{thm421} do not hold. In this case we are going to study, following
Avella-Medina (2017), the limiting form of the IF of the
MNPRPE using a sequence of continuous and infinitely differentiable functions,
$p_{m,\lambda}(s)$, that converge in the Sobolev space $W^{2,2}(\Theta)$ to
$p_{\lambda}(|s|)$, i.e., $\lim_{m\rightarrow\infty}p_{m,\lambda
}(s)=p_{\lambda}(|s|).$ \ We denote by $\boldsymbol{T}_{m,\alpha}(G)$ the
 MNPRPE functional obtained with the penalty $p_{m,\lambda
}(\cdot)$, and $\boldsymbol{T}_{\alpha}(G)$ the MNPRPE functional obtained with the
penalty $p_{\lambda}(\cdot)$. The IF of the functional
$\boldsymbol{T}_{m,\alpha}(G)$ is given by Theorem \ref{thm421} and the
IF of the functional $\boldsymbol{T}_{\alpha}(G)$ is then defined as
\begin{equation}
\operatorname{IF}\left(  (y_{t},\boldsymbol{x}_{t}),\boldsymbol{T}_{\alpha
},G\right)  =\lim_{m\rightarrow\infty}\operatorname{IF}\left(  (y_{t}%
,\boldsymbol{x}_{t}),\boldsymbol{T}_{m,\alpha},G\right)  . \label{eqlimIF}%
\end{equation}

\begin{theorem}
\label{thm422}Consider the above-mentioned set-up with the general penalty
function $p_{\lambda}(|s|)$ where $p_{\lambda}(s)$ is twice differentiable in
$s$. We assume that $L^{\ast\alpha}(\boldsymbol{\beta},\sigma),$
$\mathbb{E}_{Y,\boldsymbol{X}}\left[  \boldsymbol{\Psi}_{\alpha}%
(\boldsymbol{\beta},\sigma)\right]  $ and $\boldsymbol{J}_{\alpha
}(G;\boldsymbol{\beta},\sigma)=\mathbb{E}_{Y,\boldsymbol{X}}\left[
\nabla\boldsymbol{\Psi}_{\alpha}(\boldsymbol{\beta},\sigma)\right]  $ 
exist and are finite. For any $\boldsymbol{v}=\left(  v_{1},...,v_{p}\right)
^{T}$ with $v_{j}\neq0,$ $j=1,...,p,$ we define ,%
\[
\tilde{\boldsymbol{p}}_{\lambda}^{\ast}(\boldsymbol{v})=\left(  p_{\lambda
}^{\prime}(|v_{1}|)\operatorname{sgn}(v_{1}),..,p_{\lambda}^{\prime}%
(|v_{p}|)\operatorname{sg}(v_{p})\right)  ^{T}\text{ and }\tilde
{\boldsymbol{p}}_{\lambda}^{\ast\ast}(\boldsymbol{v})=\operatorname{diag}%
\left(  p_{\lambda}^{^{\prime\prime}}(|v_{1}|),..,p_{\lambda}^{^{\prime\prime
}}(|v_{p}|)\right)  ^{T}.%
\]
 Then,

\begin{enumerate}
\item[i)] Denote $\boldsymbol{\beta}_{\ast}=\boldsymbol{T}_{\alpha
}^{\boldsymbol{\beta}}(G)$ and assume that it has no null components
($p\leq n$). Then, the IF of the MNPRPE functional $T_{\alpha}$ $(G)$ is given by
\[
\operatorname{IF}\left(  (y_{t},\boldsymbol{x}_{t}),\boldsymbol{T}_{\alpha
},G\right)  =-\boldsymbol{J}_{\alpha}^{\ast}\left(  G,\left(
\boldsymbol{\beta}_{\ast},\sigma_{\ast}\right)  \right)  ^{-1}\left(
\begin{array}
[c]{l}%
-\alpha\sigma_{\ast}^{-\frac{2\alpha+1}{\alpha+1}}\phi_{1,\alpha}\left(
\frac{y-\boldsymbol{x}^{T}\boldsymbol{\beta}_{\ast}}{\sigma\ast}\right)
\boldsymbol{x}+\tilde{\boldsymbol{p}}_{\lambda}^{\ast}(\boldsymbol{\beta
}_{\ast})\\
-\alpha\sigma_{\ast}^{-\frac{2\alpha+1}{\alpha+1}}\phi_{2,\alpha}\left(
\frac{y-\boldsymbol{x}^{T}\boldsymbol{\beta}_{\ast}}{\sigma_{\ast}}\right)
\end{array}
\right)  ,
\]
with $\boldsymbol{J}_{\alpha}^{\ast}(G,\left(  \boldsymbol{\beta}_{\ast
},\sigma_{\ast}\right)  )=\boldsymbol{J}_{\alpha}(G,\left(  \boldsymbol{\beta
}_{\ast},\sigma_{\ast}\right)  )+diag(\tilde{\boldsymbol{p}}_{\lambda}%
^{\ast\ast}(\boldsymbol{\beta}),0).$

\item[ii)] If $\boldsymbol{\beta}_{\ast}$ has $s$ ($s<n)$ non zero
components, i.e., $\boldsymbol{\beta}_{\ast}=\left(  \left(  \boldsymbol{\beta
}_{1}^{\ast}\right)  ^{T},\boldsymbol{0}_{p-s}^{T}\right)  ^{T}$ ( where
$\boldsymbol{\beta}_{1}^{\ast}$ contains all and only s-non-zero elements of
$\boldsymbol{\beta}_{\ast})$, the
corresponding partition of the MNPRPE functional $\boldsymbol{T}_{\alpha
}\left(  G\right)  $ by $\left(  \boldsymbol{T}_{1,\alpha}^{\boldsymbol{\beta
}}(G)^{T},\boldsymbol{T}_{2,\alpha}^{\boldsymbol{\beta}}(G)^{T},T_{\alpha
}^{\sigma}(G)\right)  ^{T}$. Then, whenever the associated quantities exists,
the IF of $\boldsymbol{T}_{2,\alpha}^{\boldsymbol{\beta}}(G)$ is identically
zero \ and the IF of $\left(  \boldsymbol{T}_{1,\alpha
}^{\boldsymbol{\beta}}(G)^{T},T_{\alpha}^{\sigma}(G)\right)  ^{T}$ is given
by
\[
\operatorname{IF}\left(  (y_{t},\boldsymbol{x}_{t}),\left(  \boldsymbol{T}%
_{1,\alpha}^{\boldsymbol{\beta}},T_{\alpha}^{\sigma}\right)  ,G\right)
=-\boldsymbol{J}_{\alpha}^{\ast}\left(  G,\left(  \boldsymbol{\beta}_{1}%
^{\ast},\sigma_{\ast}\right)  \right)  ^{-1}\left(
\begin{array}
[c]{c}%
-\alpha\sigma_{\ast}^{-\frac{2\alpha+1}{\alpha+1}}\phi_{1,\alpha
}\left(  \frac{y-\boldsymbol{x}^{T}\boldsymbol{\beta}_{1}^{\ast}}{\sigma
_{\ast}}\right)  \boldsymbol{x}+\tilde{\boldsymbol{p}}_{\lambda}^{\ast
}(\boldsymbol{\beta}_{1}^{\ast})\\
-\sigma_{\ast}{}^{-\frac{2\alpha+1}{\alpha+1}}\phi_{2,\alpha}\left(
\frac{y-\boldsymbol{x}^{T}\boldsymbol{\beta}_{1}^{\ast}}{\sigma_{\ast}}\right)
\end{array}
\right)  .
\]
\end{enumerate}
\end{theorem}

Note that the boundedness of the IF of the model parameters does not depend on the penalty function. Figure \ref{IFfigure} shows the IF of the functionals associated to $\boldsymbol{\beta}$ and $\sigma$ for different tunning parameters $\alpha$. Explanatory variables have been generated under a standard normal disbribution, and the true parameters are fixed as $\boldsymbol{\beta}_0 = (0.5,0.5)^T$ and $\sigma_0 = 0.1$. The abcissa axis contains variables $u=\frac{y-\boldsymbol{x}^T\boldsymbol{\beta}}{\sigma}.$ 
The increasing robustness of the MNPRPE with the tunning parameter $\alpha$ is highlighted, as well as the lack of robustness of the MLE, corresponding to the value $\alpha=0$, having unbounded IF.
\begin{figure}[H]
	\centering
	\includegraphics[scale=0.4]{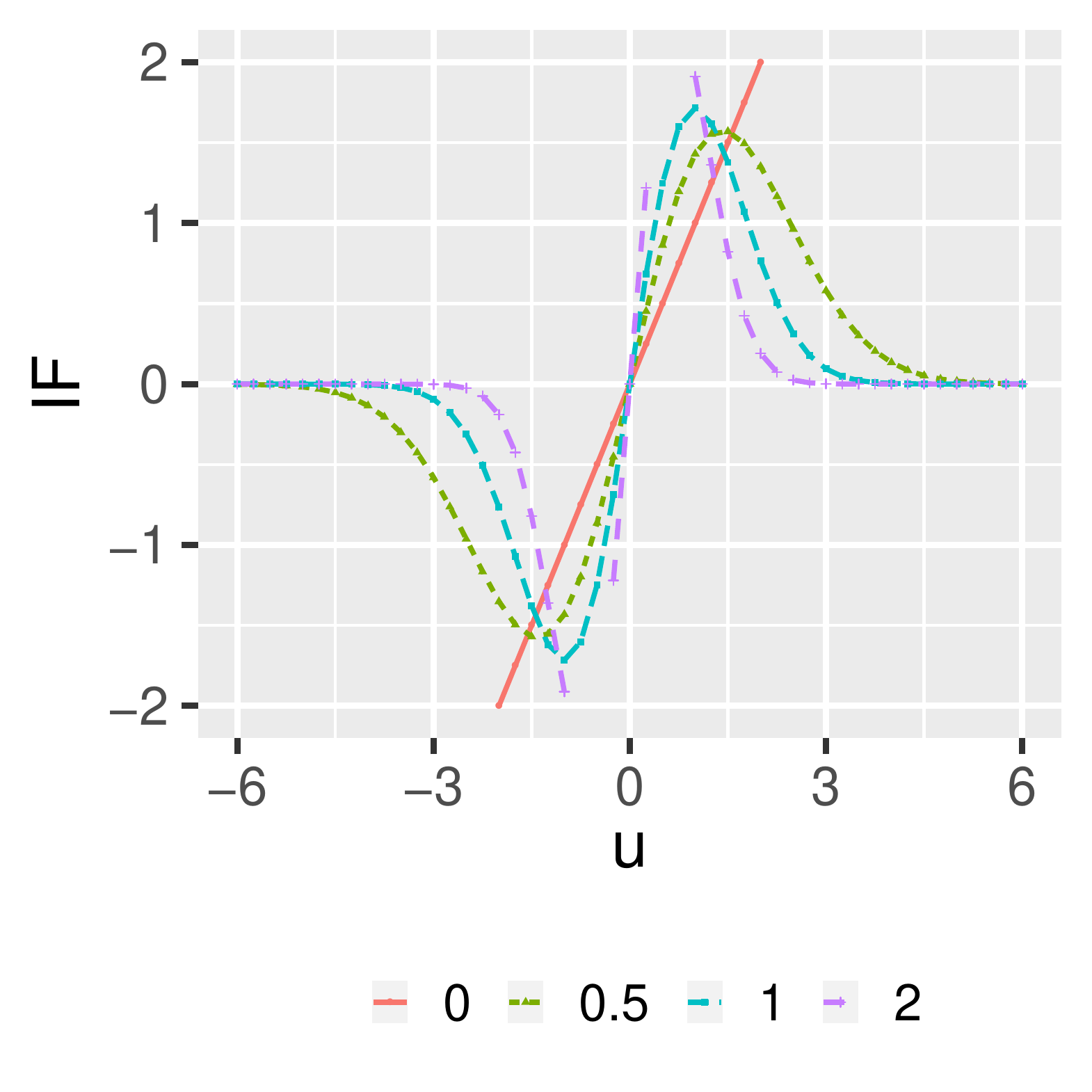}
	\includegraphics[scale=0.4]{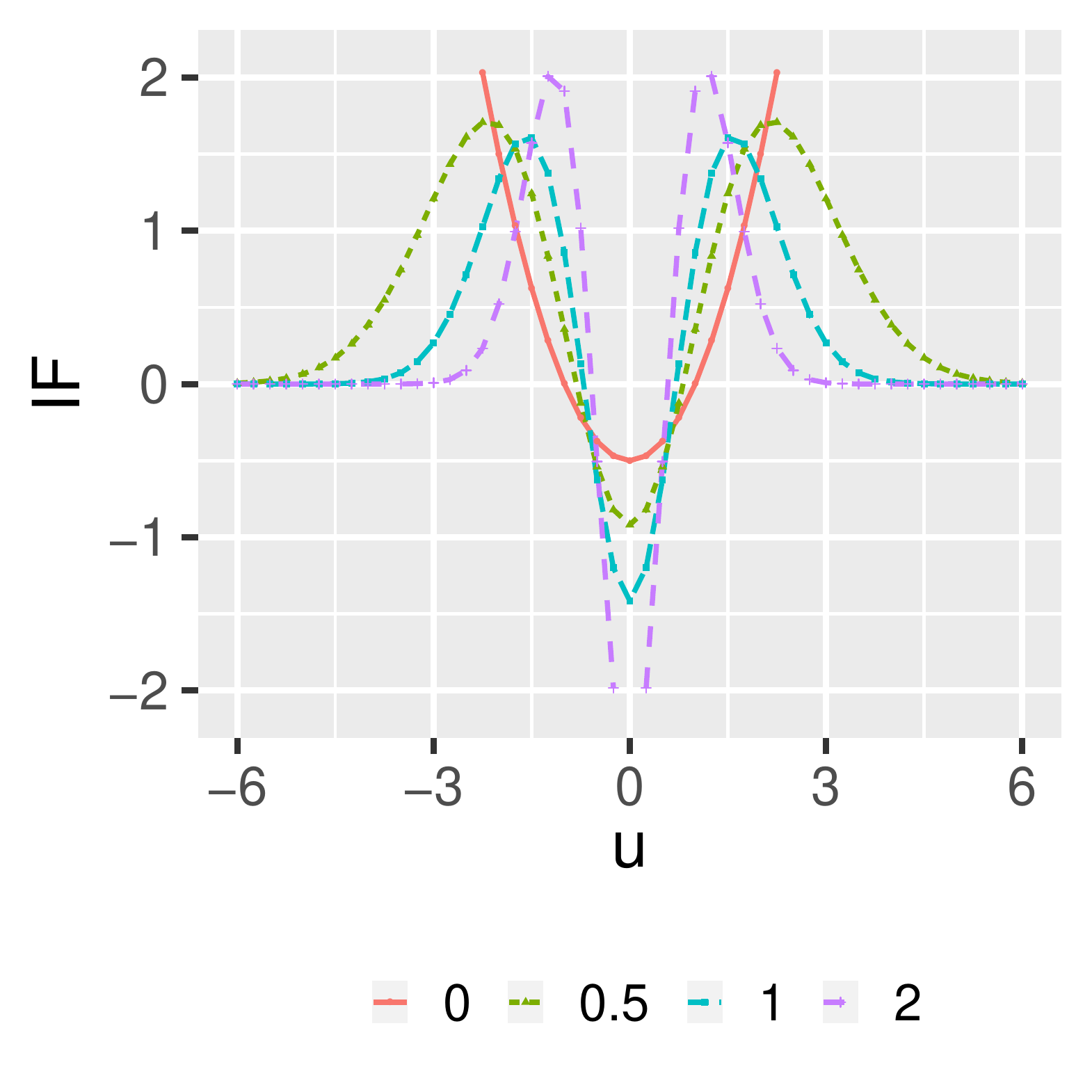}
	\caption{IF of the functional associated to $\boldsymbol{\beta}$ (left) and $\sigma$ (right)}
	\label{IFfigure}
\end{figure}

\section{\bigskip Asymptotic properties for the MNPRPE}

In this section we present the asymptotic theory for the MNPRPE. The proofs are developed in the Online Supplement with
special attention to the oracle properties. Let $\boldsymbol{\theta}_0^T = (\boldsymbol{\beta}%
_{0},\sigma_{0})$ be the true value of the parameters for the LRM with
$\boldsymbol{\beta}_{0}=\left(  \beta_{10},...,\beta_{p0}\right)  ^{T}$ and we
denote $\mathcal{S}=\{j|$ $\boldsymbol{\beta}_{j0}\neq0\}$ $\ $with
cardinality $s<p,$ i.e., $\left\vert \mathcal{S}\right\vert =s.$ An estimator,
$\widehat{\boldsymbol{\theta}}^{T}=\left(  \widehat{\boldsymbol{\beta}%
},\widehat{\sigma}\right)  ,$ obtained by minimizing the objective function
$Q_{n,\lambda}\left(  \boldsymbol{\beta},\sigma\right)  ,$ $\ $given in
(\ref{0.21}), has the oracle properties, if it identifies  the true subset model, i.e.,
$\{j|\widehat{\boldsymbol{\beta}}_{j}\neq0\}=\mathcal{S}$, with
probability tending to $1$ as $n\rightarrow\infty.$


We shall assume in accordance with Fan and Lv (2011) and Ghosh and Majunder
(2020) that the penalty function $\tilde{\boldsymbol{p}}_{\lambda
}(\boldsymbol{\beta})=\sum_{j=1}^{p}p_{\lambda}(\beta_{j})$ verify the
following condition:

\begin{enumerate}
\item[(C1)] \label{(C1)} $p_{\lambda}(s)$ is increasing, continuously
differentiable and concave in $s\in\lbrack0,\infty)$. Also $p_{\lambda
}^{\prime}(s)/\lambda$ is an increasing function of $\lambda$ with
$\rho(p_{\lambda}):=p_{\lambda}^{\prime}(0+)/\lambda$ being positive and
independent of $\lambda$.
\end{enumerate}

It is not difficult to see, \ Li and Fan (2009), that the penalties $\ell_{1}%
$, SCAD and MCP verify condition (C1).

Following Lv and Fan (2009) and Zhang (2010) we define the local
and maximum concavity of a penalty function:

\begin{definition}
The local concavity of the penalty function $p_{\lambda}(\cdot)$ at
$\boldsymbol{b}=(b_{1},..,b_{p})^{T}\in\mathbb{R}^{p}$ is defined as
\[
\xi(p_{\lambda},\boldsymbol{b})=\lim_{\varepsilon\downarrow0}\max_{1\leq j\leq
p}\left[  \hspace{0.1cm}\sup_{t_{1}<t_{2}\in(|b_{j}|\pm\varepsilon)}%
-\frac{p_{\lambda}^{\prime}(t_{2})-p_{\lambda}^{\prime}(t_{1})}{t_{2}-t_{1}%
}\right]
\]

and the maximum concavity is defined as $\xi(p_{\lambda})=\sup_{t_{1}<t_{2}%
\in(0,\infty)}-\frac{p_{\lambda}^{\prime}(t_{2})-p_{\lambda}^{\prime}(t_{1}%
)}{t_{2}-t_{1}}.$
\end{definition}

It is not difficult to establish, using Condition (C1), that $\xi(p_{\lambda
},\boldsymbol{b})\geq0$. Additionally, $\xi(p_{\lambda})\geq0$ and using the mean-value theorem and assuming that the second derivative of $p_{\lambda}(\cdot)$ $\ $is
continuous, we have $\xi(p_{\lambda},\boldsymbol{b})=\max_{j}\left(
-p_{\lambda}^{^{\prime\prime}}(|b_{j}|)\right)  .$ In the case of the
SCAD\ penalty $\xi(p_{\lambda},\boldsymbol{b})=0$ except if some component of
the vector $\boldsymbol{b}$ varies in the interval $\left[  \lambda,\lambda
a\right]  $ for which  $\xi(p_{\lambda},\boldsymbol{b}%
)=(a-1)^{-1}\lambda^{-1}.$ For more details see Fan and Lv (2011).

Let $\boldsymbol{\theta=}$ $(\boldsymbol{\beta},\sigma)$ be the unknown
parameters of the LRM and we denote, following Ghosh and Majunder (2020),
$r_{i}(\boldsymbol{\theta})=\left(  y_{i}-\boldsymbol{x}_{i}^{T}%
\boldsymbol{\beta}\right)  /\sigma,$ $i=1,...,n$ and $\boldsymbol{r}%
(\boldsymbol{\theta})=\left(  r_{1}(\boldsymbol{\theta}),...,r_{n}%
(\boldsymbol{\theta})\right)  ^{T}.$ We shall establish necessary and sufficient
conditions for the existence of a local minimizer of the objective function,
$Q_{n,\lambda}^{\alpha}(\boldsymbol{\theta}),$ given in (\ref{CF}).

\begin{theorem}
\label{thm4220}Assume that the penalty function verifies Condition C1. Then,
$\widehat{\boldsymbol{\theta}}^{T}=\left(  \widehat{\boldsymbol{\beta}%
},\widehat{\sigma}\right)  ,$ is a strict minimizer of the objective function,
$Q_{n,\lambda}^{\alpha}(\boldsymbol{\theta}),$ given in (\ref{CF}), for a
fixed $\alpha\geq0$, if and only if,
\begin{align}
\alpha\left(  \widehat{\sigma}^{\alpha}\right)  ^{-\frac{2\alpha+1}{\alpha+1}}%
{\textstyle\sum_{i=1}^{n}}
\boldsymbol{\phi}_{1,\alpha}(r_{i}(\widehat{\boldsymbol{\theta}}%
))\boldsymbol{x}_{1i}+\tilde{\boldsymbol{p}}_{\lambda}^{\ast}%
(\widehat{\boldsymbol{\beta}}_{1})  &  =\boldsymbol{0}\label{4.1}\\
\left\Vert \frac{1}{\lambda}\alpha\left(  \widehat{\sigma}^{\alpha}\right)
^{-\frac{2\alpha+1}{\alpha+1}}%
{\textstyle\sum_{i=1}^{n}}
\boldsymbol{\phi}_{1,\alpha}(r_{i}(\widehat{\boldsymbol{\theta}}%
))\boldsymbol{x}_{2i}\right\Vert _{\infty}  &  <\rho(p_{\lambda})\label{4.2}\\
\alpha\left(  \widehat{\sigma}^{\alpha}\right)  ^{-\frac{2\alpha+1}{\alpha+1}}%
{\textstyle\sum_{i=1}^{n}}
\boldsymbol{\phi}_{2,\alpha}(r_{i}(\widehat{\boldsymbol{\theta}}))  &
=0\label{4.3}\\
\Lambda_{\min}\left(  -\alpha\left(  \widehat{\sigma}^{\alpha}\right)
^{-\frac{2\alpha+1}{\alpha+1}}%
{\textstyle\sum_{i=1}^{n}}
\left[
\begin{matrix}
\boldsymbol{J}_{11,\alpha}\left(  r_{i}(\widehat{\boldsymbol{\theta}})\right)
\boldsymbol{x}_{1i}\boldsymbol{x}_{1i}^{T} & \boldsymbol{J}_{12,\alpha}\left(
r_{i}(\widehat{\boldsymbol{\theta}})\right)  \boldsymbol{x}_{1i}^{T}\\
\boldsymbol{J}_{21,\alpha}\left(  r_{i}(\widehat{\boldsymbol{\theta}})\right)
\boldsymbol{x}_{1i} & \boldsymbol{J}_{22,\alpha}\left(  r_{i}%
(\widehat{\boldsymbol{\theta}})\right)
\end{matrix}
\right]  \right)   &  >\xi(p_{\lambda},\widehat{\boldsymbol{\beta}}_{1})
\label{4.4}%
\end{align}
where $\widehat{\boldsymbol{\beta}}_{1}$ is the subvector of
$\widehat{\boldsymbol{\beta}}$ formed by all noncero components,
$\boldsymbol{x}_{i}=(\boldsymbol{x}_{1i}^{T},\boldsymbol{x}_{2i}^{T})^{T}$ is
the corresponding partition of $\boldsymbol{x}_{i}$ in such a way that the
number of components of $\boldsymbol{x}_{1i}$ coincides with the components of
$\widehat{\boldsymbol{\beta}}_{1}$, the matrices $\boldsymbol{J}_{ij,\alpha
}\left(  \cdot\right)  $ are the derivatives of $\phi_{i,\alpha}$, $i=1,2$ ,
with respect to $\boldsymbol{\beta}$ for $j=1$ and $\sigma$ for $j=2$ ,
$\tilde{\boldsymbol{p}}_{\lambda}^{\ast}(\widehat{\boldsymbol{\beta}}%
_{1})=\left(  p_{\lambda}^{\prime}(\beta_{1}),..,p_{\lambda}^{\prime}%
(\beta_{p})\right)  ^{T}$ and $\Lambda_{\min}(\boldsymbol{A})$ denotes the
minimum eigenvalue of the symmetric matrix $\boldsymbol{A}$.
\end{theorem}

Conditions (\ref{4.1}), (\ref{4.3}) and (\ref{4.4}) ensure that
$\widehat{\boldsymbol{\theta}}^{T}=\left(  \widehat{\boldsymbol{\beta}%
},\widehat{\sigma}\right)  $ is a strict local minimizer of $Q_{n,\lambda
}^{\alpha}(\boldsymbol{\theta})$ when constrained on the subspace
$\mathcal{B}=\{(\boldsymbol{\beta}^{T},\sigma)^{T}\in\mathbb{R}^{p}%
\times\mathbb{R}^{+}:\beta_{j}=0\hspace{0.2cm}\forall j>s\}.$ Condition
(\ref{4.2}) ensure that $\ \left(  \widehat{\boldsymbol{\beta}}%
,\widehat{\sigma}^{\alpha}\right)  $ is a strict local minimizer of
$Q_{n,\lambda}^{\alpha}(\boldsymbol{\theta})$ in the whole space.

Now we are going to give some conditions in order to establish the oracle
properties of MNPRPE, $\widehat{\boldsymbol{\theta}}.$ It is necessary to
introduce some notation:
Assume that the first $s$ components of
$\boldsymbol{\beta}_{0}$ $\ $\ are non-zero and the vector
$\boldsymbol{\beta}_{0}$ can be written as $\boldsymbol{\beta}_{0}^{T}=\left(
\boldsymbol{\beta}_{\mathcal{S}0},\boldsymbol{0}_{p-s}\right)  $ with
$\boldsymbol{\beta}_{\mathcal{S}0}\in\mathbb{R}^{s}.$ In the following we
denote $\boldsymbol{\beta}^{T}=\left(  \boldsymbol{\beta}_{\mathcal{S}%
},\boldsymbol{\beta}_{\mathcal{N}}\right)  $ and $\mathbb{X}=\left[
\mathbb{X}_{\mathcal{S}},\mathbb{X}_{\mathcal{N}}\right]  $ where
$\mathbb{X}_{\mathcal{S}}\in\mathbb{R}^{n\times s}$ and $\mathbb{X}%
_{\mathcal{N}}\in\mathbb{R}^{n\times(p-s)}$ and we define the following
matrices:%
\[
\mathbb{X}_{h}^{\ast}=\text{Block-diag}\left(  \mathbb{X}_{h},\boldsymbol{1}%
_{n}\right),  \hspace{0.3cm}h=\mathcal{S},\mathcal{N};\text{ }%
\boldsymbol{\operatorname{J}}_{ij}^{(\alpha)}(\boldsymbol{\theta
})=\operatorname{diag}\{\operatorname{J}_{ij,\alpha}(r_{1}(\boldsymbol{\theta
})),...,\operatorname{J}_{ij,\alpha}(r_{n}(\boldsymbol{\theta}))\}\hspace
{0.3cm}i,j=1,2\text{ }%
\]
and%
\[
\text{ }\boldsymbol{\Sigma}_{\alpha}(\boldsymbol{\theta})=\left[
\begin{matrix}
\boldsymbol{\operatorname{J}}_{11}^{\alpha}(\boldsymbol{\theta}) &
\boldsymbol{\operatorname{J}}_{12}^{\alpha}(\boldsymbol{\theta})\\
\boldsymbol{\operatorname{J}}_{21}^{\alpha}(\boldsymbol{\theta}) &
\boldsymbol{\operatorname{J}}_{22}^{\alpha}(\boldsymbol{\theta})
\end{matrix}
\right]  ,
\]
where $\boldsymbol{1}_{n}=(1,...,1)^T\in\mathbb{R}^{n}.$ Based on this notation, Equation (\ref{4.4}) can be written as $\Lambda_{\min}\left(
\mathbb{X}_{\mathcal{S}}^{\ast T}\boldsymbol{\Sigma}_{\alpha}%
(\widehat{\boldsymbol{\theta}})\mathbb{X}_{\mathcal{S}}^{\ast}\right)
>\xi(p_{\lambda},\widehat{\boldsymbol{\beta}}_{1}).$

\begin{enumerate}
\item[(A1)] Let $\boldsymbol{x}^{(j)}$ be the $j$-th column of matrix
$\mathbb{X}$, \hspace{0.1cm} $j=1,..,p.$ 
Then $||\boldsymbol{x}^{(j)}%
||_{2}=O(\sqrt{n})$.

\item[(A2)] The design matrix $\mathbb{X}$ verifies:
\begin{align}
||\left(  \mathbb{X}_{\mathcal{S}}^{\ast T}\boldsymbol{\Sigma}_{\alpha
}(\boldsymbol{\theta}_{0})\mathbb{X}_{\mathcal{S}}^{\ast}\right)
^{-1}||_{\infty}  &  =O\left(  \frac{b_{s}}{n}\right) \label{4.5}\\
||\left(  \mathbb{X}_{\mathcal{N}}^{\ast T}\boldsymbol{\Sigma}_{\alpha
}(\boldsymbol{\theta}_{0})\mathbb{X}_{\mathcal{S}}^{\ast}\right)  \left(
\mathbb{X}_{\mathcal{S}}^{\ast T}\boldsymbol{\Sigma}_{\alpha}%
(\boldsymbol{\theta}_{0})\mathbb{X}_{\mathcal{S}}^{\ast}\right)
^{-1}||_{\infty}  &  <\min\left\{  \frac{Cp_{\lambda}^{\prime}(0+)}%
{p_{\lambda}^{\prime}(d_{n})},O(n^{\tau_{1}})\right\} \label{4.6}\\
\max_{(\boldsymbol{\delta},\sigma)\in\mathcal{N}_{0}}\max_{1\leq j\leq
p+1}\left\{  \Lambda_{\max}\left(  \nabla_{(\boldsymbol{\delta},\sigma)}%
^{2}\gamma_{j,\alpha}(\boldsymbol{\delta},\sigma)\right)  \right\}   &  =O(n)
\label{4.7}%
\end{align}
for $C\in(0,1),$ $\tau_{1}\in\lbrack0,0.5],$ and $\mathcal{N}_{0}=\{\left(
\boldsymbol{\delta},\sigma\right)  \in\mathbb{R}^{s}\times\mathbb{R}%
^{+}:||\boldsymbol{\delta}-\boldsymbol{\beta}_{\mathcal{S}0}||_{\infty}\leq
d_{n},|\sigma-\sigma_{0}|\leq d_{n}\}.$ \ By $\nabla_{(\boldsymbol{\delta
},\sigma)}^{2}$ denote the second order derivative with respect to
$(\boldsymbol{\delta},\sigma)$ and
\begin{equation}%
\begin{split}
\gamma_{j,\alpha}(\boldsymbol{\delta},\sigma)  &  =\alpha\sigma^{-\frac
{2\alpha+1}{\alpha+1}}%
{\textstyle\sum_{i=1}^{n}}
\boldsymbol{\phi}_{1,\alpha}(r_{i}(\boldsymbol{\delta}_{\ast}))x_{i}%
^{(j)}\hspace{0.3cm}j=1,2,...,p,\\
\gamma_{p+1,\alpha}(\boldsymbol{\delta},\sigma)  &  =\alpha\sigma
^{-\frac{2\alpha+1}{\alpha+1}}%
{\textstyle\sum_{i=1}^{n}}
\boldsymbol{\phi}_{2,\alpha}(r_{i}(\boldsymbol{\delta}_{\ast})),
\end{split}
\label{4.8}%
\end{equation}

$\boldsymbol{\delta}$$_{\ast}^{T}=\left(  \boldsymbol{\delta},\boldsymbol{0}%
_{p-s},\sigma\right)  ,$ $b_{s}$ is a diverging sequence of positive numbers
depending on $s$ $\ $\ and hence depend on $n$, $d_{n}=\min_{j\in\mathcal{S}%
}|\beta_{0,j}|/2$ and $\left\Vert \boldsymbol{A}\right\Vert _{\infty}$ the
maximum of $\ell_{1}$ norm of each row of $\boldsymbol{A}$.

\item[(A3)] Assume that $d_{n}\geq\log n/n^\tau$ and
\begin{equation}
b_{s}=o\left(  \min\left(  n^{1/2-\tau}\sqrt{\log n},\frac{n^{\tau}}{(s+1)\log
n}\right)  \right)  \text{ for }\tau\in(0,0.5]. \label{4.8.1}%
\end{equation}
In addition, assume if $s=O(n^{\tau_{0}})$ that the regularization parameter
$\lambda$ satisfy
\begin{equation}
p_{\lambda}^{\prime}(d_{n})=o\left(  b_{s}^{-1}n^{-\tau}\log n\right)  \text{
and }\hspace{0.3cm}\lambda\geq(\log n)^{2}n^{-\tau^{\ast}} \label{4.9}%
\end{equation}
with $\tau^{\ast}=\min\left(  0.5,2\tau-\tau_{0}\right)  -\tau_{1}$. Also,
$\max_{\boldsymbol{\delta}\in\mathcal{N}_{0}}\xi(p_{\lambda}%
,\boldsymbol{\delta})=o\left(  \max_{\boldsymbol{\delta}\in\mathcal{N}_{0}%
}\Lambda_{\min}\left[  \frac{1}{n}\mathbb{X}_{\mathcal{S}}^{\ast
T}\boldsymbol{\Sigma}_{\alpha}(\boldsymbol{\delta})\mathbb{X}_{\mathcal{S}%
}^{\ast}\right]  \right)  $ and $\max_{1\leq j\leq p}||\boldsymbol{x}%
^{(j)}||_{\infty}=o\left(  n^{\tau^{\ast}}/\sqrt{\log n}\right)  .$
\end{enumerate}

Based on the previous assumptions we are going to establish a weak oracle
property of the MNPRPE. Note that these assumptions are in line with those used by Ghosh and Majumder (2020). We start with the following proposition.

\begin{proposition} \label{Hoeffding}
\label{propo1}For all $\boldsymbol{a}\in\mathbb{R}^{n}$ and $0<\varepsilon
<\frac{||\boldsymbol{a}||_{2}}{||\boldsymbol{a}||_{\infty}}$, we have,
\[
\Pr\left(  \left\vert \alpha\sigma^{-\frac{2\alpha+1}{\alpha+1}}%
{\textstyle\sum_{i=1}^{n}}
a_{i}\boldsymbol{\phi}_{1,\alpha}(r_{i}(\boldsymbol{\theta}_{0}))\right\vert
>||\boldsymbol{a}||_{2}\varepsilon\right)  \leq2\exp(-c_{1}\varepsilon^{2}).
\]

\end{proposition}

In Ghosh and Majunder (2020), the result in Proposition \ref{Hoeffding} is considered as an assumption,
namely (A4); however, in our case, it always holds as can be seen from the proof of Proposition 7. 

\begin{theorem}
\label{thm531}Let us consider the objective function, $Q_{n,\lambda}^{\alpha
}(\boldsymbol{\theta}),$ given in (\ref{CF}) for a fixed $\alpha\geq0$, with
$p_{\lambda}(|.|)$ verifying Condition C1. We shall assume that $s=o(n)$,
$\log p=O(n^{1-2\tau^{\ast}})$ and conditions (A1)-(A3) are verified. Then,
there exists a MNPRPE , $\widehat{\boldsymbol{\beta}}^{T}=\left(
\widehat{\boldsymbol{\beta}}_{\mathcal{S}},\widehat{\boldsymbol{\beta}%
}_{\mathcal{N}}\right)  $ of parameter $\boldsymbol{\beta}$,
$\widehat{\boldsymbol{\beta}}_{\mathcal{S}}\in\mathbb{R}^{s}$, and
$\widehat{\sigma}^{{}}$ of $\sigma$ in such a way that
$\widehat{\boldsymbol{\theta}}^{T}=(\widehat{\boldsymbol{\beta}}%
,\widehat{\sigma})$ is an strict local minimizer of $Q_{n,\lambda}^{\alpha
}(\boldsymbol{\theta})$, with
\end{theorem}

\begin{enumerate}
\item $\widehat{\boldsymbol{\beta}}_{\mathcal{N}}=\boldsymbol{0}_{p-s}$, and

\item $||\widehat{\boldsymbol{\beta}}_{\mathcal{S}}-\boldsymbol{\beta
}_{\mathcal{S}0}||_{\infty}=O\left(  n^{-\tau}\log n\right)  $ and
$|\widehat{\sigma}-\sigma_{0}|=O\left(  n^{-\tau}\log n\right)  $ with
probability at least $$1-2\left[ \frac{1+s}{n} +(p-s)\exp(-n^{1-2\tau^{\ast}%
}\log n)\right]  .$$
\end{enumerate}

It is possible to get stronger results if we consider stronger conditions than
(A2) and (A3).

\bigskip


\begin{enumerate}
\item[(A2)$^{\ast}$] The design matrix $\mathbb{X}$ verifies 
\begin{align}
& \label{A2510}\min_{(\boldsymbol{\delta},\sigma) \in \mathcal{N}_0} \Lambda_{\min}\left[\mathbb{X}_\mathcal{S}^{*T} \boldsymbol{\Sigma}_\alpha\left((\boldsymbol{\delta}^T,\boldsymbol{0}_{p-s}, \sigma)^T\right) \mathbb{X}^T_\mathcal{S}\right] \geq cn\\
& \bigg| \bigg| \left(\mathbb{X}_\mathcal{N}^{*T} \boldsymbol{\Sigma}_\alpha\left(\boldsymbol{\theta}_0\right) \mathbb{X}^T_\mathcal{S}\right) \bigg|\bigg|_{2,\infty} = \mathcal{O}(n)\\
& \max_{(\boldsymbol{\delta},\sigma)\in \mathcal{N}_0} \max_{1 \leq j \leq p+1} \Lambda_{\max} \left(\mathbb{X}_\mathcal{S}^{*T} \left[\nabla_{\boldsymbol{\theta}}^2 \boldsymbol{\gamma}_{j,\alpha}(\boldsymbol{\delta},\sigma)\right] \mathbb{X}_\mathcal{S}^{*}\right) = \mathcal{O}(n) \label{512}
\end{align}

for some $c>0$ and $\mathcal{N}_{0}=\{(\boldsymbol{\delta},\sigma
)\in\mathbb{R}^{s}\times\mathbb{R}^{+}:||\boldsymbol{\delta}-\boldsymbol{\beta
}_{\mathcal{S0}}||_{\infty}\leq d_{n},|\sigma-\sigma_{0}|\leq d_{n}\}$ and
$\left\Vert \boldsymbol{A}\right\Vert _{2,\infty}=\max_{\left\Vert
\boldsymbol{v}\right\Vert _{2,}=1}\left\Vert \boldsymbol{Av}\right\Vert
_{\infty}$. Further
\[
\mathbb{E}\hspace{0.1cm}\left[  \left\Vert \alpha\sigma_{0}^{-\frac{2\alpha
+1}{\alpha+1}}\sum_{i=1}^{n}\phi_{1,\alpha}(r_{i}(\boldsymbol{\theta}%
_{0}))\boldsymbol{x}_{\mathcal{S}i}\right\Vert _{2}^{2}\right]  =O\left(
\frac{s}{n}\right)  \text{ and }\mathbb{E}\hspace{0.1cm}\left[  \left\vert
\alpha\sigma_{0}^{-\frac{2\alpha+1}{\alpha+1}}%
{\textstyle\sum_{i=1}^{n}}
\phi_{2,\alpha}(r_{i}(\boldsymbol{\theta}_{0}))\boldsymbol{x}_{\mathcal{S}%
i}\right\vert ^{2}\right]  =O\left(  \frac{1}{n}\right)  .
\]

\item[(A3)$^{\ast}$] We have $p_{\lambda}^{\prime}(d_{n})=O(n^{-1/2}%
);\hspace{0.3cm}d_{n}\gg\lambda\gg\min\{\sqrt{\frac{s}{n}},n^{\frac{\tau-1}%
{2}}\sqrt{\log n}\}$ and
\begin{equation}
\max_{(\boldsymbol{\delta},\sigma)\in\mathcal{N}_{0}}\xi(p_{\lambda
},\boldsymbol{\delta})=O(1). \label{A3513}%
\end{equation}
Further, $\max_{1\leq j\leq p}||\boldsymbol{x}||_{\infty}=O\left(
n^{(1-\tau)/2}/\sqrt{\log n}\right)  .$
\end{enumerate}

\begin{theorem}
\label{thm532} Let $s\ll n$ and $\log p=O(n^{\tau^{\ast}})$ for some
$\tau^{\ast}\in(0,0.5)$, we shall assume Condition (C1) and Assumptions (A1),
(A2)$^{\ast}$ and (A3)$^{\ast}$ are verified for some fixed $\alpha.$ Then,
there exists an strict local minimizer $\widehat{\boldsymbol{\theta}}%
^{T}=(\widehat{\boldsymbol{\beta}},\widehat{\sigma})$ of the objective
function $Q_{n}^{\alpha}(\boldsymbol{\theta}),$ verifying:

\begin{enumerate}
\item $\widehat{\boldsymbol{\beta}}_{\mathcal{N}}^{\alpha}=\boldsymbol{0}$,
where $\widehat{\boldsymbol{\beta}}^{T}=(\widehat{\boldsymbol{\beta}%
}_{\mathcal{S}},\widehat{\boldsymbol{\beta}}_{\mathcal{N}})$ and
$\widehat{\boldsymbol{\beta}}_{\mathcal{S}}\in\mathbb{R}^{s},$

\item $||\widehat{\boldsymbol{\beta}}-\boldsymbol{\beta}_{0}||=O(\sqrt{s/n})$
and $|\widehat{\sigma}-\sigma_{0}|=O(n^{-1/2}),$
\end{enumerate}

with probability tending to 1 when $n\rightarrow\infty.$
\end{theorem}

To establish the asymptotic normality, we need an additional assumptions related to the Liapunov condition. We define the following matrices
\begin{align*}
\boldsymbol{V}_{\alpha}(\boldsymbol{\theta})  &  =\operatorname{Var}%
_{G}\left[  \boldsymbol{\Psi}_{\alpha}(\boldsymbol{\theta})\right] \\
&  =\alpha\sigma^{-\frac{2\alpha+1}{\alpha+1}}\mathbb{E}\left[  \left(
\begin{matrix}
\phi_{1,\alpha}^{2}(r(\boldsymbol{\theta}))\boldsymbol{X}\boldsymbol{X}^{T} &
\phi_{1,\alpha}(r(\boldsymbol{\theta}))\phi_{2,\alpha}(r(\boldsymbol{\theta
}))\boldsymbol{X}\\
\phi_{1,\alpha}(r(\boldsymbol{\theta}))\phi_{2,\alpha}(r(\boldsymbol{\theta
}))\boldsymbol{X}^{T} & \phi_{2,\alpha}^{2}(r_{1}(\boldsymbol{\theta}))
\end{matrix}
\right)  \right] \\
&  =\alpha\sigma^{-\frac{2\alpha+1}{\alpha+1}}\left[
\begin{matrix}
\phi_{1,\alpha}^{2}(r(\boldsymbol{\theta}))\mathbb{E}\left[  \boldsymbol{X}%
\boldsymbol{X}^{T}\right]  & \boldsymbol{0}\\
\boldsymbol{0} & \phi_{2,\alpha}^{2}(r_{1}(\boldsymbol{\theta}))
\end{matrix}
\right], \\
\boldsymbol{K}_{ij}^{\alpha}(\boldsymbol{\theta})  &  =\alpha\sigma
^{-\frac{2\alpha+1}{\alpha+1}}\operatorname{diag}\left(  \phi_{i,\alpha}%
(r_{1}(\boldsymbol{\theta}))\phi_{j,\alpha}(r_{1}(\boldsymbol{\theta}%
)),\cdot\cdot\cdot,\phi_{i,\alpha}(r_{n}(\boldsymbol{\theta}))\phi_{j,\alpha
}(r_{n}(\boldsymbol{\theta}))\right)  \hspace{0.3cm}i,j=1,2.
\end{align*}
A consistent estimator of $\boldsymbol{V}_{\alpha}(\boldsymbol{\theta})$ is
$\frac{1}{n}\mathbb{X}_{\mathcal{S}}^{\ast,T}\boldsymbol{\Sigma}_{\alpha
}^{\ast}(\boldsymbol{\theta})\mathbb{X}_{\mathcal{S}}^{\ast}$ with
\[
\boldsymbol{\Sigma}_{\alpha}^{\ast}(\boldsymbol{\theta})=\left(
\begin{matrix}
\boldsymbol{K}_{11}^{\alpha}(\boldsymbol{\theta}) & \boldsymbol{K}%
_{12}^{\alpha}(\boldsymbol{\theta})\\
\boldsymbol{K}_{11}^{\alpha}(\boldsymbol{\theta}) & \boldsymbol{K}%
_{22}^{\alpha}(\boldsymbol{\theta})
\end{matrix}
\right)  .
\]
We now need to assume the following additional assumption.
\begin{enumerate}
\item[(A5)] The penalty and loss function verify%
\[
p_{\lambda}^{\prime}(d_{n})=\mathcal{O}\left(  \left(  sn\right)
^{-1/2}\right)  \text{ and }\max_{1\leq i\leq n}\mathbb{E}\left[
|\phi_{k,\alpha}(r_{i}(\boldsymbol{\theta}_{0}))|\right]  ^{3}=O(1),\hspace
{0.3cm}k=1,2,
\]
and the design matrix verifies :%
\[
\min_{(\boldsymbol{\delta},\sigma)\in\mathcal{N}_{0}}\Lambda_{\min}\left[
\mathbb{X}_{\mathcal{S}}^{\ast,T}\boldsymbol{\Sigma}_{\alpha}^{\ast
}(\boldsymbol{\delta}_{\ast})\mathbb{X}_{\mathcal{S}}^{\ast}\right]  \geq
cn\text{ and }%
{\textstyle\sum_{i=1}^{n}}
\left[  \boldsymbol{x}_{\mathcal{S}i}^{\ast,T}\left(  \mathbb{X}_{\mathcal{S}%
}^{\ast,T}\boldsymbol{\Sigma}_{\alpha}^{\ast}(\boldsymbol{\theta}%
_{0})\mathbb{X}_{\mathcal{S}}^{\ast}\right)  ^{-1}\boldsymbol{x}%
_{\mathcal{S}i}^{\ast}\right]  ^{3/2}=o(1),
\]

\end{enumerate}

where $\boldsymbol{x}_{\mathcal{S}i}^{\ast}:=(\boldsymbol{x}_{\mathcal{S}%
i}^{T},1)^{T}$ .

\begin{theorem}
\label{thm533} In addition to the conditions of Theorem (\ref{thm532}), if Assumption
(A5) holds and $s=o(n^{1/3}),$ then with probability tending to 1 as
$n\rightarrow\infty,$ the MNPRPE, $\widehat{\boldsymbol{\theta}}%
^{T}=(\widehat{\boldsymbol{\beta}},\widehat{\sigma}),$ verifies:

\begin{enumerate}
\item $\widehat{\boldsymbol{\beta}}_{\mathcal{N}}^{\alpha}=\boldsymbol{0}$,
con $\widehat{\boldsymbol{\beta}}^{T}=(\widehat{\boldsymbol{\beta}%
}_{\mathcal{S}},\widehat{\boldsymbol{\beta}}_{\mathcal{N}})$ and
$\widehat{\boldsymbol{\beta}}_{\mathcal{S}}\in\mathbb{R}^{s}$

\item Let's $\boldsymbol{A}_{n}\in\mathbb{R}^{q\times(s+1)}$ a matrix \ such
that $\boldsymbol{A}_{n}\boldsymbol{A}_{n}^{T}\underset{n\rightarrow
\infty}{\rightarrow}\boldsymbol{G}$ , $\boldsymbol{G}$ is a symmetric positive
definite matrix,
\begin{equation}
\boldsymbol{A}_{n}\left(  \mathbb{X}_{\mathcal{S}}^{\ast,T}\boldsymbol{\Sigma
}_{\alpha}^{\ast}(\boldsymbol{\theta}_{0})\mathbb{X}_{\mathcal{S}}^{\ast
}\right)  ^{-\frac{1}{2}}\left(  \mathbb{X}_{\mathcal{S}}^{\ast,T}%
\boldsymbol{\Sigma}_{\alpha}^{\ast}(\boldsymbol{\theta}_{0})\mathbb{X}%
_{\mathcal{S}}^{\ast}\right)  \left[  (\widehat{\boldsymbol{\beta}%
}_{\mathcal{S}},\widehat{\sigma})^{T}-(\boldsymbol{\beta}_{\mathcal{S}%
0},\sigma_{0})^{T}\right]  \overset{L}{\underset{n\rightarrow\infty
}{\rightarrow}}\mathcal{N}_{q}(\boldsymbol{0}_{q},\boldsymbol{G})
\end{equation}

\end{enumerate}
\end{theorem}

\section{Computational Algorithm}

In this section, we discuss algorithms for minimizing the penalized objetive function $Q_n(\boldsymbol{\beta},\sigma)$, given in (\ref{CF}), with nonconcave penalties like SCAD and MCP. 

Recall that, the efficient algorithms for the least squares regression and group LASSO penalties, usually use the local convex nature of the objetive function. 
For non-convex objective functions involving penalties like SCAD or MCP, Fan and Li (2001) proposed the LQA  and Zou and Li (2008)  introduced the LLA algorithms, using  local quadratic and linear approximations, respectively. These algorithms are inherently inefficient to some extent, in that it uses the path-tracing least angle regression algorithm (LARS)  to produce updates to the regression coefficients.
Fan and  Lv (2011) used iterative coordinate ascent (ICA) optimization for penalized least squares with nonconcave penalty functions, which is especially appealing for large scale problems with both $n$ and $p$ large. Breheny and Huang (2011) established a coordinate descent algorithm for nonconcave penalized regression with squared loss and SCAD or MCP penalties. 
On the other hand, Kawashima and Fujisawa (2017) employed the Majorize-Minimization (MM) algorithm for the $\gamma-$divergence loss function penalized with LASSO penalty, which  iteratively bounds the loss, resulting in a weighted least squared regression.

In this paper, we propose to combine MM-algorithm and coordinate descent minimization for the RP loss function and the non-concave penalties including  SCAD and MCP. 
The proposed method is iterative, and it updates the estimates of the parameters $\boldsymbol{\beta}$ and $\sigma$ separately at each step. Before describing our proposal, let us briefly mention the MM and the coordinate descent algorithm to understand the underlying reasonings. 

\subsection{MM-algorithm}

The MM optimization algorithm iteratively updates a current solution by finding a surrogate function that majorizes the objective function. Optimizing the surrogate function will then drive the actual objective function downward until a local minimum is reached (Hunter and Lange, (2004)).

Mathematically, let be $h(\nu)$ a real-valued  objective function.  A function $h_{MM}(\nu|\nu^{(m)})$  is said to majorize $h(\nu)$ at a given  point $\nu^{(m)}$ (current solution) if 
\begin{equation} \label{MMeq}
h_{MM}(\nu^{(m)}|\nu^{(m)}) = h(\nu^{(m)}) \hspace{0.5cm} \text{and} \hspace{0.5 cm}
h_{MM}(\nu|\nu^{(m)}) \geq h(\nu).
\end{equation}

Then, in the MM-algorithm, the next updated solution is obtained as 
$$ \nu^{(m+1)} = \operatorname{arg}\operatorname{min}_{\nu} h_{MM}(\nu | \nu^{(m)}). $$
The process is repeated for $m=0, 1, 2,...$ until convergence is reached. 
The notation $h_{MM}(\nu^{(m)}|\nu^{(m)})$ emphasizes the dependence of the current solution $\nu^{(m)}$, a crucial requirement of the MM-algorithm. .

\begin{proposition}
	MM-algorithm, using $h_{MM}(\nu|\nu^{(m)})$ majorization function satisfying (\ref{MMeq}), converges to the required minimizer of $h(\nu)$.
\end{proposition} 

\begin{proof}
	The first equation on (\ref{MMeq}) ensures both functions match at  $\nu^{(m)}$, while the second one guarantees the stricly downward. 
	The objective function $h(\nu)$ monotonically decreases at each step
	\begin{equation} \label{descentprop}
		h (\nu^{(m+1)}) \leq h_{MM}(\nu^{(m+1)} | \nu^{(m)}) \leq  h_{MM}(\nu^{(m)} | \nu^{(m)}) = h(\nu^{(m)}),
	\end{equation}  
	and hence, the MM-algorithm converges to a local minimum of $h(\nu)$. The descent property (\ref{descentprop}) lends an MM-algorithm remarkable numerical stability. 	
\end{proof}

Note that, in view of (\ref{descentprop}),  $\nu^{(m+1)}$ is not necessary a minimizer of  $h_{MM}(\nu | \nu^{(m)})$, but it will suffices if only 
$ h_{MM}(\nu^{(m)} | \nu^{(m)}) \geq h_{MM}(\nu^{(m+1)} | \nu^{(m)}).$

Thus, any other simple iterative algorithms may be used to minimize the majorization function. Moreover, Hunter and Lange (2004) proved that MM-algorithms boast a linear rate of convergence 
$$\lim_{m\rightarrow \infty} \frac{||\nu^{(m+1)}-\nu^{\ast}||}{||\nu^{(m)}-\nu^{\ast}||} = c < 1.$$ 

Kawashima and Fujisawa (2017) constructed the majorization function for  $\gamma-$divergence loss function, which coincides with our RP loss for the LRM, by Jensen's inequality
$$\kappa(\boldsymbol{z}^T \boldsymbol{\nu}) = \kappa\left( \frac{z_i\nu_i^{(m)}}{\boldsymbol{z}^T \boldsymbol{\nu}^{(m)}} \nu_i \frac{\boldsymbol{z}^T \boldsymbol{\nu}^{(m)}}{\nu_i^{(m)}}\right) \leq \sum_{i} \frac{z_i \nu_i^{(m)}}{\boldsymbol{z}^T \boldsymbol{\nu}^{(m)}} \kappa\left(\nu_i \frac{\boldsymbol{z}^T \boldsymbol{\nu}^{(m)}}{\nu_i^{(m)}}\right),$$
where $\kappa(\nu)$ is a convex function and  $\boldsymbol{z}$, $\boldsymbol{\nu}$ and $\boldsymbol{\nu^{(m)}}$ are postive vectors.
In our case of the RP loss function given in (\ref{RPL}), taking 
$\boldsymbol{z} = (\frac{1}{n},..,\frac{1}{n})$, $\nu_i = f_{\boldsymbol{\beta},\sigma}(y_i|\boldsymbol{x}_i)^\alpha$, $\nu^{(m)}_i = f_{\boldsymbol{\beta}^{(m)},\sigma^{(m)}}(y_i|\boldsymbol{x}_i)^\alpha$ and $\kappa(u) = - \log (u)$,
we get
\begin{equation} 
\begin{aligned}
\mathcal{R}_{\alpha}\left(  P_{\boldsymbol{\beta},\sigma},P_{n}%
^{\boldsymbol{\beta}}\right)   &  =\frac{1}{\alpha+1}\log\left\{  \left(
\frac{1}{\sqrt{2\pi}\sigma}\right)  ^{\alpha}\frac{1}{\sqrt{\alpha+1}%
}\right\}  +\frac{1}{\alpha(\alpha+1)}\log\left(  \frac{1}{n}\right)
^{\alpha}\\
&  -\frac{1}{\alpha}\log\left\{  \frac{1}{n}\sum_{i=1}^{n}\left(  \frac
{1}{\sqrt{2\pi}\sigma}\right)  ^{\alpha}\exp\left(  -\frac{\alpha}{2}\left(
\frac{y_{i}-\boldsymbol{x}_{i}^{T}\boldsymbol{\beta}}{\sigma}\right)
^{2}\right)  \right\}  .
\end{aligned}
\end{equation}
\begin{equation} \label{MMalg}
\begin{split}h(\boldsymbol{\beta},\sigma) 
&= \frac{1}{\alpha+1} \sum_{i=1}^n \frac{1}{n}\log \left(\int f_{\boldsymbol{\beta},\sigma}(y|\boldsymbol{x}_i)^{\alpha+1} dy \right)   +\frac{1}{\alpha(\alpha+1)}\log\left(  \frac{1}{n}\right)
^{\alpha} - \frac{1}{\alpha} \log \frac{1}{n}\sum_{i=1}^n f_{\boldsymbol{\beta},\sigma}(y_i|\boldsymbol{x}_i)^\alpha   \\ 
&  \leq	\frac{1}{\alpha+1} \sum_{i=1}^n \frac{1}{n}\log \left(\int f_{\boldsymbol{\beta},\sigma}(y|\boldsymbol{x}_i)^{\alpha+1} dy \right)   +\frac{1}{\alpha(\alpha+1)}\log\left(  \frac{1}{n}\right)
^{\alpha} \\
& - \frac{1}{\alpha}\sum_{i=1}^n \mu_i^{(m)} \log \left\{ f_{\boldsymbol{\beta},\sigma}(y_i|\boldsymbol{x}_i)^{\alpha} \frac{\frac{1}{n} \sum_{l=1}^n f_{\boldsymbol{\beta}^{(m)},\sigma^{(m)}}(y_l|\boldsymbol{x}_l)^{\alpha+1}}{f_{\boldsymbol{\beta}^{(m)},\sigma^{(m)}}(y_i|\boldsymbol{x}_i)^{\alpha}}\right\} \\
&= \frac{1}{\alpha+1} \sum_{i=1}^n \frac{1}{n}\log \left(\int f_{\boldsymbol{\beta},\sigma}(y|\boldsymbol{x}_i)^{\alpha+1} dy \right) 
   +\frac{1}{\alpha(\alpha+1)}\log\left(  \frac{1}{n}\right)
 ^{\alpha} - \sum_{i=1}^n \mu_i^{(m)} \log \left(f_{\boldsymbol{\beta},\sigma}(y_i|\boldsymbol{x}_i) \right) \\
&= \frac{1}{\alpha+1} \log \left[ \frac{1}{\sigma^\alpha (\sqrt{2 \pi})^\alpha \sqrt{\alpha+1}}\right]   +\frac{1}{\alpha(\alpha+1)}\log\left(  \frac{1}{n}\right)
^{\alpha} - \sum_{i=1}^n \mu_i^{(m)} \log \left(\frac{1}{(\sqrt{2\pi})^\alpha \sigma^\alpha} \right)\\
& +  \sum_{i=1}^n \mu_i^{(m)}  \frac{1}{2} \left( \frac{y_i-\boldsymbol{x}_i^T \boldsymbol{\beta}}{\sigma}\right)^2\\
&= h_{MM}(\boldsymbol{\beta}, \sigma | \boldsymbol{\beta}^{(m)}, \sigma^{(m)})
\end{split}
\end{equation}
with $$\mu_i^{(m)} = \frac{f_{\boldsymbol{\beta}^{(m)},\sigma^{(m)}}(y_i|\boldsymbol{x}_i)^{\alpha}}{\sum_{l=1}^n f_{\boldsymbol{\beta}^{(m)},\sigma^{(m)}}(y_l|\boldsymbol{x}_l)^{\alpha}}.$$
Therefore, it is enough to minimize the majorization function $h_{MM}(\boldsymbol{\beta}, \sigma | \boldsymbol{\beta}^{(m)}, \sigma^{(m)})$ to downward  $L_n^\alpha (\boldsymbol{\beta},\sigma)$ at each step. Note that only the last term of $h_{MM}$ depends on $\boldsymbol{\beta}$. The same is also true to their penalized versions as required to compute the MNPRPE. 

\subsection{Coordinate descent algorithm}

In order to minimize the majorization function in each step during the computation of the MNPRPE, we use the popular coordinate descent algorithm which optimizes the objective function with respect to every single parameter at a time, iteratively cycling through all parameters until convergence. 
This algorithm is specially appropriate for very high-dimensional problems, as each pass over the parameters requires only $O(np)$ operations, and  computational burden increases only linearly with $p$. 

While working with penalized objective functions with SCAD or MCP penalties, Breheny and Huang (2011)  showed that, for the squared error loss in a  univariate penalized regression, the minimization problem  has an explicit solution. Consider the soft-thresholding operator (Donoho and Johnstone (1994))
$$ S(z,\lambda) = \left\{ \begin{matrix}
z - \lambda & \text{if} & z > \lambda \\
0 & \text{if} & |z|<\lambda\\
z + \lambda & \text{if} & z<-\lambda\\
\end{matrix} \right. $$
and the simple linear regression model 
$$ y = x\beta + \varepsilon.$$
Given a random sample $\left((y_1,x_1),..,(y_n,x_n)\right)$ and assuming for simplicity that the explanatory variable $x$ is centered, the objective function for penalized least squares regression is
\begin{equation} \label{squaresimple}
	\frac{1}{n} \sum_{i=1}^n (y-x\beta)^2 + p_{\lambda}(\beta).
\end{equation}
If the $p_\lambda$ is the MCP penalty, then the minimizer of the objective function in (\ref{squaresimple}) has the explicit form  
$$\widehat{\beta} = f_{\text{MCP}}(z,\lambda) = \left\{\begin{matrix}
\frac{S(z,\lambda)}{1-1/a} & \text{if} & |z|\leq \lambda a,\\
z & \text{if} & |z| > \lambda a.
\end{matrix}\right. $$
and for the SCAD penalty, the corresponding minimizer of (\ref{squaresimple}) has the form 
$$\widehat{\beta} = f_{\text{SCAD}}(z,\lambda) = \left\{\begin{matrix}
S(z,\lambda) & \text{if} & |z| \leq 2\lambda \\
\frac{S(z,a\lambda/(a-1))}{1-1/(a-1)} & \text{if} & 2\lambda < |z|\leq \lambda a,\\
z & \text{if} & |z| > \lambda a.
\end{matrix}\right. $$
where $z= \frac{1}{n}\boldsymbol{x}^T\boldsymbol{y}$ is the solution of unpenalized univariate least squares regression.

Coordinate descent minimization considers, on each iteration, $p$ simple linear regression problems, and optimizes with respect to each and every parameter separately  employing the univariate solution. 
Introducing the notation ($-j$) to refer to the portion that remains after the $j$-th column or element is removed, 			
the partial residuals of $\boldsymbol{x}_j$ are $\boldsymbol{r}_{-j} = \boldsymbol{y} - \mathbb{X}_{-j}\widehat{\boldsymbol{\beta}}_{-j}$, where $\widehat{\boldsymbol{\beta}}$ is the most recently updated value of $\boldsymbol{\beta}$. Thus, for given ﬁxed value of parameters $\{ \widehat{\beta}_k : k\neq j\}$, at a current estimates $\widehat{\boldsymbol{\beta}}$, we wish to partially minimize the objective function 
$$U_{n,\lambda}(\boldsymbol{\beta}) := \frac{1}{2n}\sum_{i=1}^n\left(y_i-\boldsymbol{x}_i^T\boldsymbol{\beta}\right)^2 + \sum_{j=1}^p p_{\lambda}(|\beta_j|)$$ 
with respect to $\beta_j$ yielding  to the simple linear regression problem
$$\min_{\beta_j} \left[\frac{1}{n}\sum_{i=1}^n(r_{-j,i} - x_{j,i}\beta_j)^2 + p_\lambda(\beta_j) \right].$$
Therefore, the  Coordinate Descent Algorithm is constructed as follows:
\begin{enumerate}
	\item Set $m=0$. Fix initial value $\widehat{\boldsymbol{\beta}}^{0}$, tuning parameter $\lambda$ and tolerance $\varepsilon$ (for convergence).
	\item For $j=1,..,p$, update $\beta_j$ following three calculations
	\begin{enumerate}
		\item Calculate $z_j = \frac{1}{n} \boldsymbol{x}_j \boldsymbol{r}_{-j} = \frac{1}{n} \boldsymbol{x}_j \boldsymbol{r} +\widehat{\beta}^{(m)}. $
		\item Update $\widehat{\beta}^{(m+1)} \leftarrow f_{\text{MCP}} (z_j, \lambda)$ or or $f_{\text{SCAD}}(z_j, \lambda)$ [depending on the choice of penalty function].
		\item Update $\boldsymbol{r} \leftarrow \boldsymbol{r} - \left(\widehat{\beta}^{(m+1)}-\widehat{\beta}^{(m)}\right)\boldsymbol{x}_j.$
	\end{enumerate}
	\item If $\big|U_{n,\lambda}\left(\widehat{\boldsymbol{\beta}}^{(m+1)}\right) - U_{n,\lambda}\left(\widehat{\boldsymbol{\beta}}^{(m+1)}\right) \big| \leq \varepsilon$ : Stop\\
	Else : set $m \leftarrow m+1$ and go to step 2.
\end{enumerate}

Breheny and Huang (2011) showed that coordinate descent algorithm for the penalized squared loss with SCAD or MCP (with parameter $a>2$ or $a>1$ respectively) downward the objetive function at each iteration, i.e., 
$U_{n,\lambda}\left(\widehat{\boldsymbol{\beta}}^{(m)} \right) \geq U_{n,\lambda}\left(\widehat{\boldsymbol{\beta}}^{(m+1)} \right).$ 
Furthermore, the sequence is guaranteed to converge to a point that is both
a local minimum and a global coordinate-wise minimum of $U_{n,\lambda}$.

\begin{remark} If the simple regression problem has not an explicit solution, but the penalty admits a decomposition as in (\ref{Des}), then 
	Concave-Convex Procedure (CCCP) (An and Tao, (1997); Yuille and Rangarajan, (2003) ) may be used to bound the penalty with a convex approximation at which univariate regression possess an explicit solution and the coordinate descent algorithm can be applied (Lee (2015) \cite{Sangin}). In this case, the convergence of the method is guaranteed by the convexity of the objective function.
\end{remark}

\subsection{The proposed algorithm for computation of the MNPRPE}
We propose to combine both optimization algorithms in order to compute the proposed MNPRPE of the regression parameter $\boldsymbol{\beta}$  along with the error variance $\sigma$  at each step.
Let us consider the  objective function $Q_n(\boldsymbol{\beta},\sigma)$ defined  in (\ref{CF}), and denote by  $(\widehat{\boldsymbol{\beta}}^{(m)}, \sigma^{(m)})$ the current estimates at step $m$, $m=1,2,..$ We first apply MM-algorithm to bound $L_n(\boldsymbol{\beta},\sigma)$ as in (\ref{MMalg}). Then, the function to minimize
$$h_{MM}(\boldsymbol{\beta}, \sigma | \boldsymbol{\beta}^{(m)}, \sigma^{(m)}) + \sum_{j=1}^p p_\lambda(\beta_j)$$
is a weighted version of mean squared loss, so iterative coordinate descent algorithm can be used to update the current solution of $\boldsymbol{\beta}$ as $\boldsymbol{\beta}^{(m+1)}$. The convergence of the method is guaranteed by the convergence of both algorithms, as both decrease its objective function in each iteration.

Next to obtain $\sigma^{(m+1)}$, the update for $\sigma$,   we consider the following derivative
\begin{align*}
\frac{\partial Q_{n,\lambda}^\alpha(\boldsymbol{\beta}^{(m+1)},\sigma)}{\partial \sigma } &=\frac{\partial}{\partial \sigma} \sum_{i=1}^n-\sigma^\frac{-\alpha}{\alpha+1} \frac{1}{n} \exp\left(-\frac{\alpha}{2}\left(\frac{y_i-\boldsymbol{x}_i^T\boldsymbol{\beta}}{\sigma}\right)^2\right)\\
&=\frac{\alpha}{n} \sum_{i=1}^n-\sigma^{\frac{-\alpha}{\alpha+1}-1} \exp\left(-\frac{\alpha}{2}\left(\frac{y_i-\boldsymbol{x}_i^T\boldsymbol{\beta}}{\sigma}\right)^2\right) \left(-\frac{1}{\alpha+1} +\left(\frac{y_i-\boldsymbol{x}_i^T\boldsymbol{\beta}}{\sigma}\right)^2 \right).\\ 
\end{align*}
Note that the equation $$\frac{\partial Q_{n,\lambda}^\alpha(\boldsymbol{\beta}^{(m+1)},\sigma)}{\partial \sigma } =0$$ does not have an explicit solution. So, we should approximate it defining 
$$w_i^{(m)}= \exp\left(-\frac{\alpha}{2}\left(\frac{y_i-\boldsymbol{x}_i^T\widehat{\boldsymbol{\beta}}^{(m)}}{\widehat{\sigma}^{(m)}}\right)^2\right)$$
and then $\sigma^{(m+1)}$ is obtained as 
\begin{equation} \label{sigmak}
\widehat{\sigma}^{2(m+1)} = \sum_{i=1}^n \left(y_i - \boldsymbol{x}_i^T\widehat{\boldsymbol{\beta}}^{(m+1)} \right)^2\left[ \sum_{i=1}^n \frac{w_i^{(m)}}{\alpha+1} \right]^{-1}.
\end{equation}
The full algorithm is described on the following pseudocode. \\
\textbf{Algorithm 1.} (Robust non-concave penalized linear regression using RP)
\begin{enumerate}
	\item Set $m=0$. Fix initial values $\widehat{\boldsymbol{\beta}}^{(0)}$ and $\widehat{\sigma}^{(0)}$, tuning parameter $\lambda$ and tolerances $\varepsilon_1$,$\varepsilon_2$ (for convergence).
	\item \label{step2} Calculate $\mu_i^{(m)} \leftarrow \frac{f_{\boldsymbol{\beta}^{(m)},\sigma^{(m)}}(y_i|\boldsymbol{x}_i)^{\alpha}}{\sum_{l=1}^n f_{\boldsymbol{\beta}^{(m)},\sigma^{(m)}}(y_l|\boldsymbol{x}_l)^{\alpha}}$
	and  $ h_{MM}(\boldsymbol{\beta}, \sigma | \boldsymbol{\beta}^{(m)}, \sigma^{(m)}) \leftarrow \sum_{i=1}^n \mu_i^{(m)}  \frac{1}{2} \left( \frac{y_i-\boldsymbol{x}_i^T \boldsymbol{\beta}}{\sigma}\right)^2. $
	\item For $i=1,..,n$ define $\boldsymbol{x}_i^{w} := \frac{ \mu_i^{(m)}}{\widehat{\sigma}^{(m)}}\boldsymbol{x}_i$ and $y_i^{w} := \frac{\mu_i^{(m)}}{\widehat{\sigma}^{(m)}} y_i$
	and update $\widehat{\boldsymbol{\beta}}^{(m)}$ as follows.
	\begin{enumerate}
		\item \label{k=0} Set $k=0$ and $\widehat{\boldsymbol{\beta}}^{*0} = \widehat{\boldsymbol{\beta}}^{(m)}.$
		\item For $j=1,..,p$, 
		\begin{enumerate}
			\item Calculate $z_j = \frac{1}{n} \boldsymbol{x}_j^{w} \boldsymbol{r}_{-j} = \frac{1}{n} \boldsymbol{x}_j^{w} \boldsymbol{r} +\widehat{\beta}^{*(k)}. $
			\item Update $\widehat{\beta}_*^{(m+1)} \leftarrow f_{\text{MCP}} (z_j, \lambda)$ or $f_{\text{SCAD}}(z_j, \lambda)$ [depending on teh choice of penalty function].
			\item Update $\boldsymbol{r} - \left(\widehat{\beta}^{*(k+1)}-\widehat{\beta}^{*(k)}\right)\boldsymbol{x}_j^{w}.$
		\end{enumerate}
		\item If $\big|Q_n(\widehat{\boldsymbol{\beta}}^{(k+1)},\sigma^{(m)}) - Q_n(\widehat{\boldsymbol{\beta}}^{(k)},\sigma^{(m)}) \big| \leq \varepsilon_1$ : Update $\widehat{\boldsymbol{\beta}}^{(m+1)} := \widehat{\boldsymbol{\beta}}^{(k+1)}$\\
		Else : set $k \leftarrow k+1$ and go to step \ref{k=0}.
	\end{enumerate}
	\item For $i=1,..,n,$ define $w_i^{(m)} \leftarrow \exp\left(-\frac{\alpha}{2}\left(\frac{y_i-\boldsymbol{x}_i^T\widehat{\boldsymbol{\beta}}^{(m)}}{\widehat{\sigma}^{(m)}}\right)^2\right)$
	and update $\widehat{\sigma}^{(m)}$ using
	$$\widehat{\sigma}^{2(m+1)} \leftarrow \sum_{i=1}^n \left(y_i - \boldsymbol{x}_i^T\widehat{\boldsymbol{\beta}}^{(m+1)} \right)^2\left[ \sum_{i=1}^n \frac{w_i^{(m)}}{\alpha+1} \right]^{-1}.$$
	
	\item If $\big|Q_n(\widehat{\boldsymbol{\beta}}^{(m+1)},\sigma^{(m+1)}) - Q_n(\widehat{\boldsymbol{\beta}}^{(m)},\sigma^{(m)}) \big| \leq \varepsilon_2$ : Stop\\
	Else : set $m \leftarrow m+1$ and go to step \ref{step2}.
\end{enumerate}
The performance of Algorithm 1 depends on choice of initial values, and the tuning parameter $\lambda$. For the first we could apply any robust regression method such as RLARS, sLTS or RANSAC as a starting point. To select the best $\lambda$ we use the High-dimensional Bayesian Information Criterion (HBIC) (Kim et al., (2012) ; Wang et al., (2013)) which has demonstrably better performance compared to standard BIC in the case of NP- dimensionality (Fan and Tang, (2013)). We deﬁne a robust version of the HBIC as:
\begin{equation} \label{HBIC}
\text{HBIC}(\lambda) = \log(\widehat{\sigma}_\lambda^2) + \frac{\log \log(n)\log p}{n}\|\widehat{\boldsymbol{\beta}}_\lambda \|_0.
\end{equation}
and select the optimal $\lambda$ that minimizes the HBIC over a pre-determined set of values.

\section{Simulation study}
\subsection{Experimental Set-up}
We now present an extensive  simulation study so as to evaluate the robustness and efficiency of the proposal MNPRPE under the LRM.
We also estimate the regression parameters $(\boldsymbol{\beta}, \sigma)$ using other exiting  robust and non-robust methods of high-dimensional LRM to compare their performances with our proposed method.

The data are generated from the LRM (\ref{0.1}) following a set-up similar to the one considered in Ghosh and Basu (2020). We set the sample size $n=100$ and the true deviation error $\sigma_0 = 0.5$, and chose the number of explanatory variables to be $p=100,200,500$ and differenr values of the true regression coefficients $\boldsymbol{\beta}_0.$ We repeat the simulations over $R=100$ replications.
Rows of the design matrix $\mathbb{X}$ are drawn from the normal distribution $\mathcal{N}\left(\boldsymbol{0, \boldsymbol{\Sigma}}\right)$, where $\boldsymbol{\Sigma}$ is a positive deﬁnite matrix with $(i,j)$-th element given by $0.5^{|i-j|}$. Given a parameter dimension $p$, we consider two settings for the coefficient vector $\boldsymbol{\beta}_0$: 
\begin{itemize}
	\item  Setting A (strong signal): we set $\beta_j = j$ for $j\in \{ 1,2,4,7,11\}$ and $\beta_j = 0$ for the rest of $p-5$ components.
	\item  Setting B (weak signal): we set $\beta_1 = \beta_7= 1.5$, $\beta_2= 0.5$, $\beta_4 = \beta_{11} = 1$ and  the rest of the $p-5$ entries of $\boldsymbol{\beta}_0$  are set at 0. 
\end{itemize}
To evaluate the performance of our proposed method, we calculate the mean square error (MSE) for the true non-zero and zero coefficients separately, Absolute Prediction Bias (APrB) using an unused test sample of size $n=100$, denoted by  $(\boldsymbol{y}_{\text{test}},\mathbb{X}_{\text{test}}),$ generated in the same way as train data,  True Positive proportion (TP), True Negative proportion (TN) and Model Size (MS) of the estimated regression coefficient $\widehat{\boldsymbol{\beta}}$, and Estimation Error (EE) of the estimate $\widehat{\sigma}$ as follows.
\begin{align*}
\operatorname{MSES}(\widehat{\boldsymbol{\beta}}) &= \frac{1}{s} \parallel \widehat{\boldsymbol{\beta}}_{\mathcal{S}} - \boldsymbol{\beta}_{0\mathcal{S}} \parallel^2 \\
\operatorname{MSEN}(\widehat{\boldsymbol{\beta}}) &= \frac{1}{p-s} \parallel \widehat{\boldsymbol{\beta}}_{\mathcal{N}} ||^2\\
\operatorname{APrB}(\widehat{\boldsymbol{\beta}}) &= \parallel \boldsymbol{y}_{\text{test}}-\mathbb{X}_{\text{test}}\widehat{\boldsymbol{\beta}} \parallel_1\\
\operatorname{EE}(\widehat{\sigma}) &= |\widehat{\sigma}-\sigma_0|\\
\operatorname{TP}(\widehat{\boldsymbol{\beta}}) &= \frac{|\text{supp}(\widehat{\boldsymbol{\beta}})\cap\text{supp}(\boldsymbol{\beta}_0)|}{|\text{supp}(\boldsymbol{\beta}_0)|}\\
\operatorname{TN}(\widehat{\boldsymbol{\beta}}) &= \frac{|\text{supp}^c(\widehat{\boldsymbol{\beta}})\cap\text{supp}^c(\boldsymbol{\beta}_0)|}{|\text{supp}^c(\boldsymbol{\beta}_0)|}\\
\operatorname{MS}(\widehat{\boldsymbol{\beta}}) &= |\text{supp}(\widehat{\boldsymbol{\beta}})|
\end{align*}
Finally, in order to examine the efficiency loss against non-robust methods in absence of any contamination, as well as compare the performance in the presence of contamination in the data, we consider different scenarios:
\begin{itemize}
	\item Absence of contamination (pure data)
	\item Contaminated data
	\begin{itemize}
		\item $Y$-outliers : We add $20$ to the response variables of a random $10\%$ of samples.
		\item $\boldsymbol{X}$-outliers : We add $20$ to each of the elements in the ﬁrst $10$ rows of $\mathbb{X}$ for a random $10\%$ of samples. 
	\end{itemize}
\end{itemize}
\subsection{Competing methods}
In order to compare our results with existing competitors, we calculate the same performance  for measures the following estimation procedures under the same simulation experiments. In particular, we consider the robust least angle regression (RLARS; Khan et al. (2007)), sparse least trimmed squares (sLTS; Alfons et al.(2013)), random sample consensus (RANSAC), the LASSO penalized  regression using least absolute deviation loss (LAD-LASSO; Wang et al. (2007)), DPD loss (DPD-LASSO, Zhang et al. (2017)) and log DPD loss (LDPD-LASSO, Kawashima and Fujisawa (2017)), and the nonconcave penalized DPD loss with the SCAD penalty (DPD-ncv, Ghosh and Majundar (2020)). For the methods DPD-LASSO, log DPD-LASSO and DPD-ncv, the starting points are chosen as the RLARS estimates because of time computational efficiency. Moreover, we also use three standard non-robust methods, namely the ones considering the least squared loss with LASSO, SCAD and MCP penalties, which we will refer to as LS-LASSO, LS-SCAD and LS-MCP, respectively,  for comparison in terms of efficiency loss. We use 5-fold cross-validation for the selection of the regularized parameter $\lambda$ in all the above competing methods except LAD-Lasso, for which we use BIC, and DPD-lasso and LDPD-lasso  for which uses HBIC criterion.

For the proposed MNPRPE we use the two most common penalties: SCAD and MCP. The results are very similar and hence, for brevity, we only report the ﬁndings for the SCAD penalty. RLARS is used to initialize the computation of the MNPRPE  and HBIC  criterion (\ref{HBIC}) is applied to choose the regularizer parameter $\lambda$.

\subsection{Results}
Tables \ref{p500e0signal1}-\ref{p500e1signal0x1} summarize the simulation results  for $p=500$ covariates; the results for $p=100$ and $p=200$ are presented in the Online Supplement for brevity. 

The results evidence that our MNPRPE selects the true model  better than any other method, and it is also more accurate in the estimation of the vector $\boldsymbol{\beta}$. However, its indisputable advantage is its accuracy on the estimation of $\sigma$. The estimation error on $\sigma$ is lower than that of  any other method for all values of $\alpha$.

On the other hand, the optimum value of $\alpha$ hover around $\alpha = 0.3$, in keeping with the best values for the LDPD-lasso. Finally, from results it is apparent that the use of nonconcave penalization improves the global performance of the method.

To examine the performance of the proposed method with increasing dimensions, Figure \ref{pVSerror} shows the mean root square error (RMSE) in prediction against the number of covariates in absence of contamination and $10\%$ of $Y-$outliers respectively. The RMSE is calculated as 
$\operatorname{RMSE}(\widehat{\boldsymbol{\beta}}) = \sqrt{\frac{1}{n}\parallel \boldsymbol{y}_{\text{test}}-\mathbb{X}_{\text{test}}\widehat{\boldsymbol{\beta}} \parallel_2^2}$
 In both cases low values of the tuning parameter $\alpha$ register lower error. Moreover, the behavior of the method for the different tuning parameters is similar for any number of covariates, suggesting that the election of $\alpha$ should only be based on the compromise between efficiency and robustness (as described previously) .
\begin{figure}[H]
	\centering
	\includegraphics[height=6cm, width=8cm]{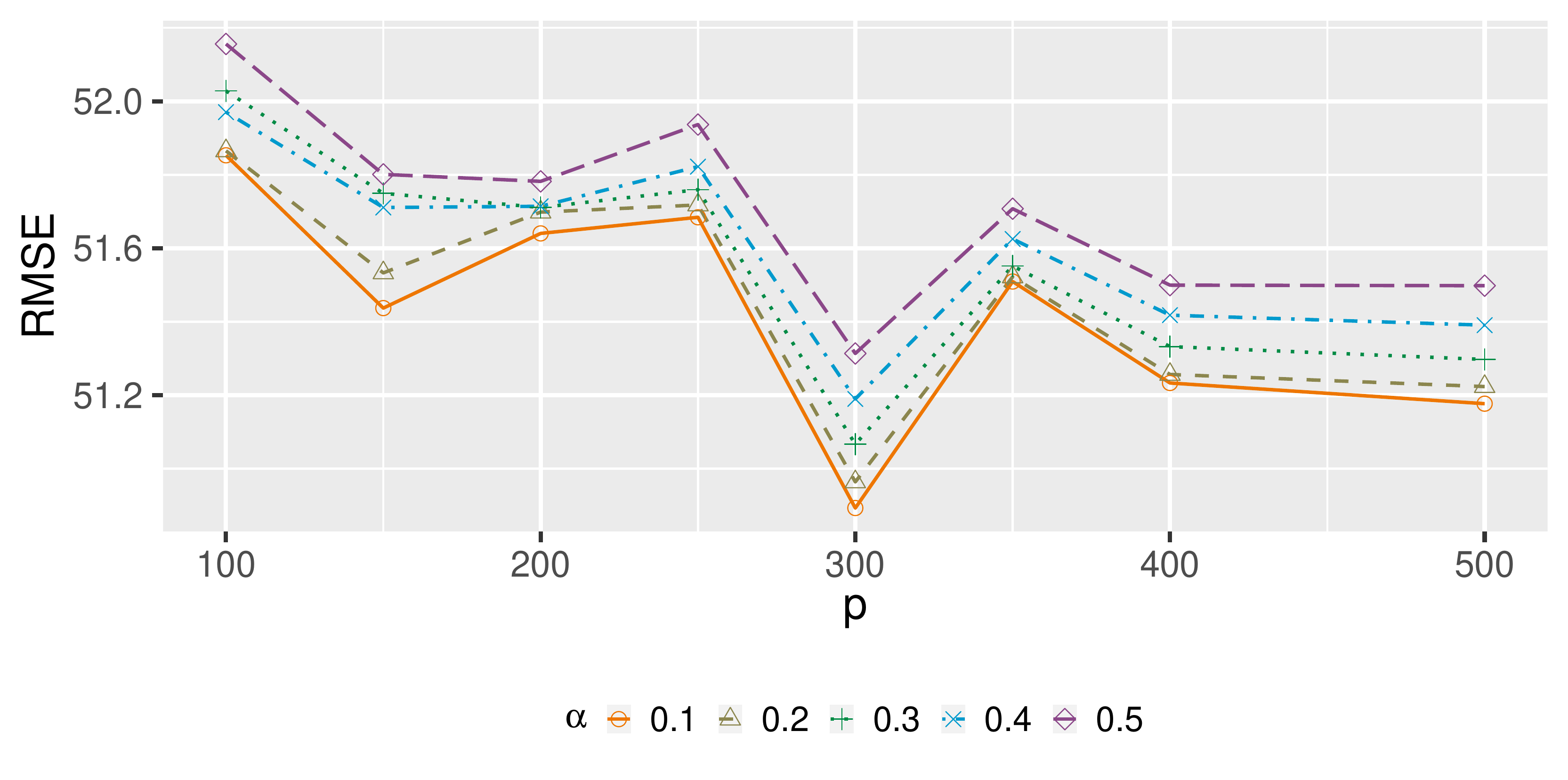}
	\includegraphics[height=6cm, width=8cm]{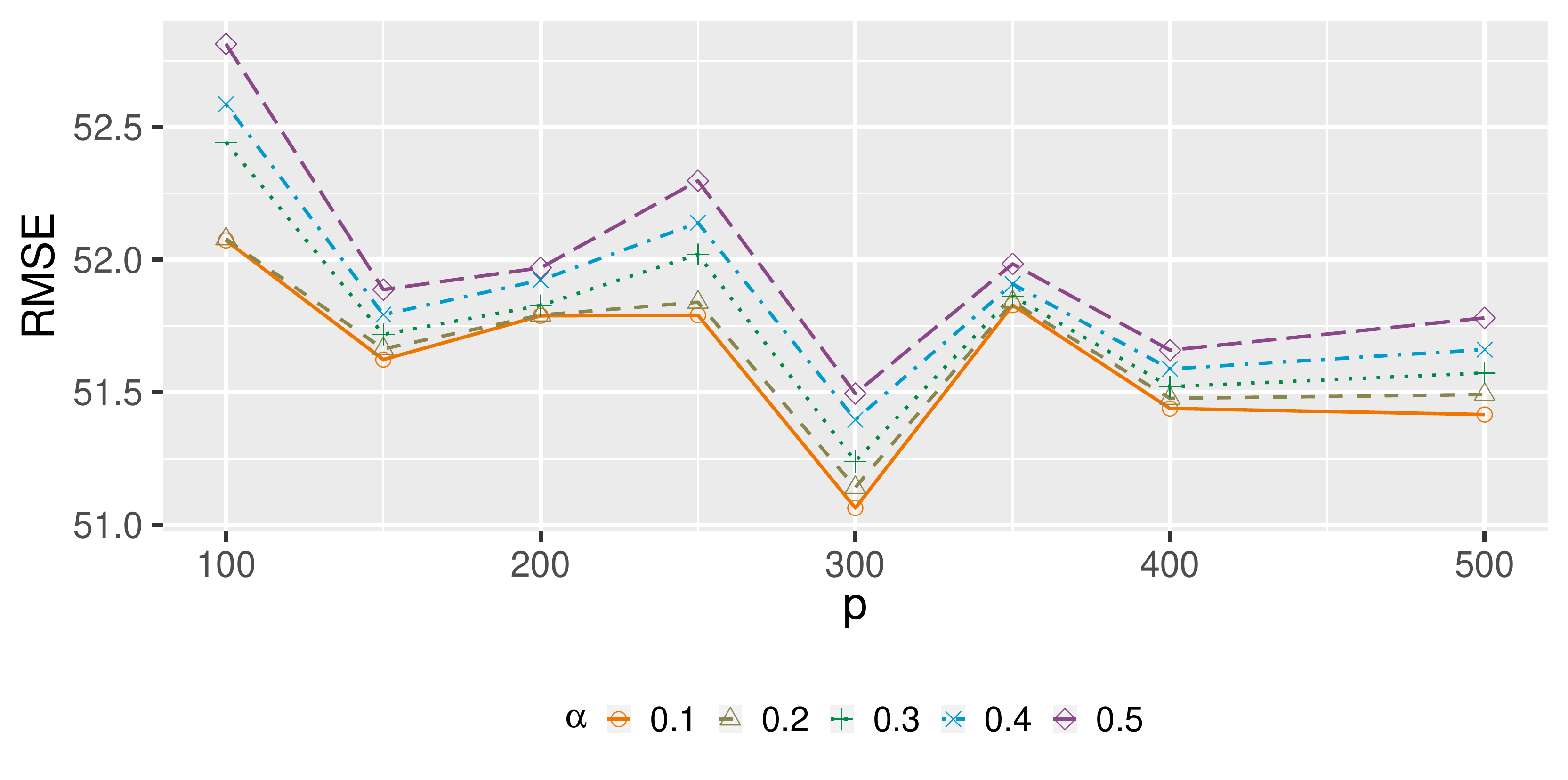}
	
	\caption{Number of covariates against RMSE in absence of contamination (right) and $10\%$ of $Y-$outliers (left)}
	\label{pVSerror}
\end{figure}
Finally, we present the RMSE against data contaminatination (Y-outliers) for $p=100$, $p=200$ and $p=500$ covariates in Figure \ref{contaminacionVSerror}, bringing to light the increasing robustness of the method with the tuning parameter $\alpha.$ In absence of contamination all tuning parameters yield low RMSE, although lower values register lower error, indicating its major efficiency. Nonetheless, from $10\%$ of $Y-$outliers, greater tuning parameters continue having low error while RMSE result with small values of $\alpha$ increases significantly.
\begin{figure}[H]
	\centering
	\includegraphics[height=5.5cm, width=8cm]{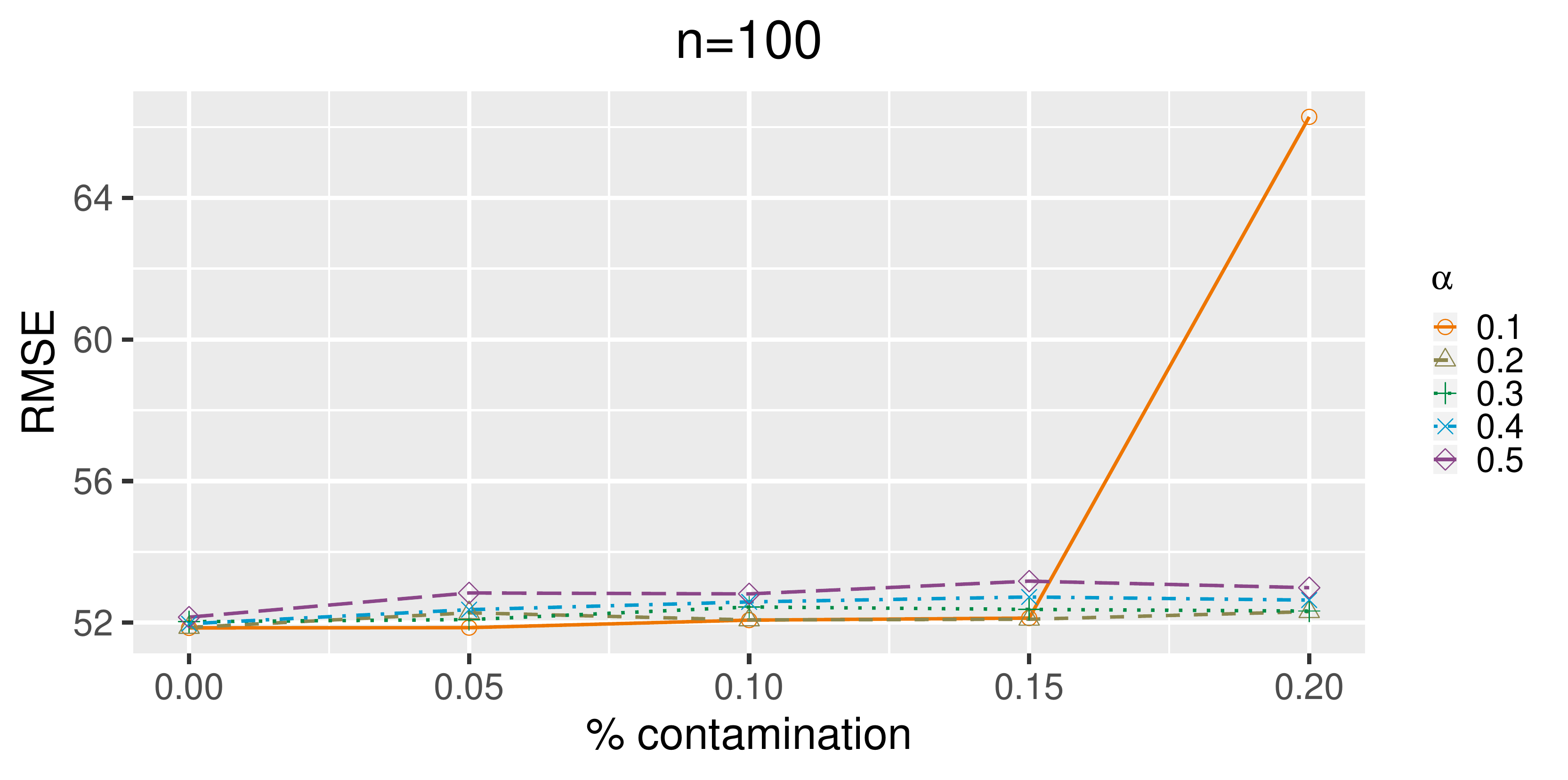}
	\includegraphics[height=5.5cm, width=8cm]{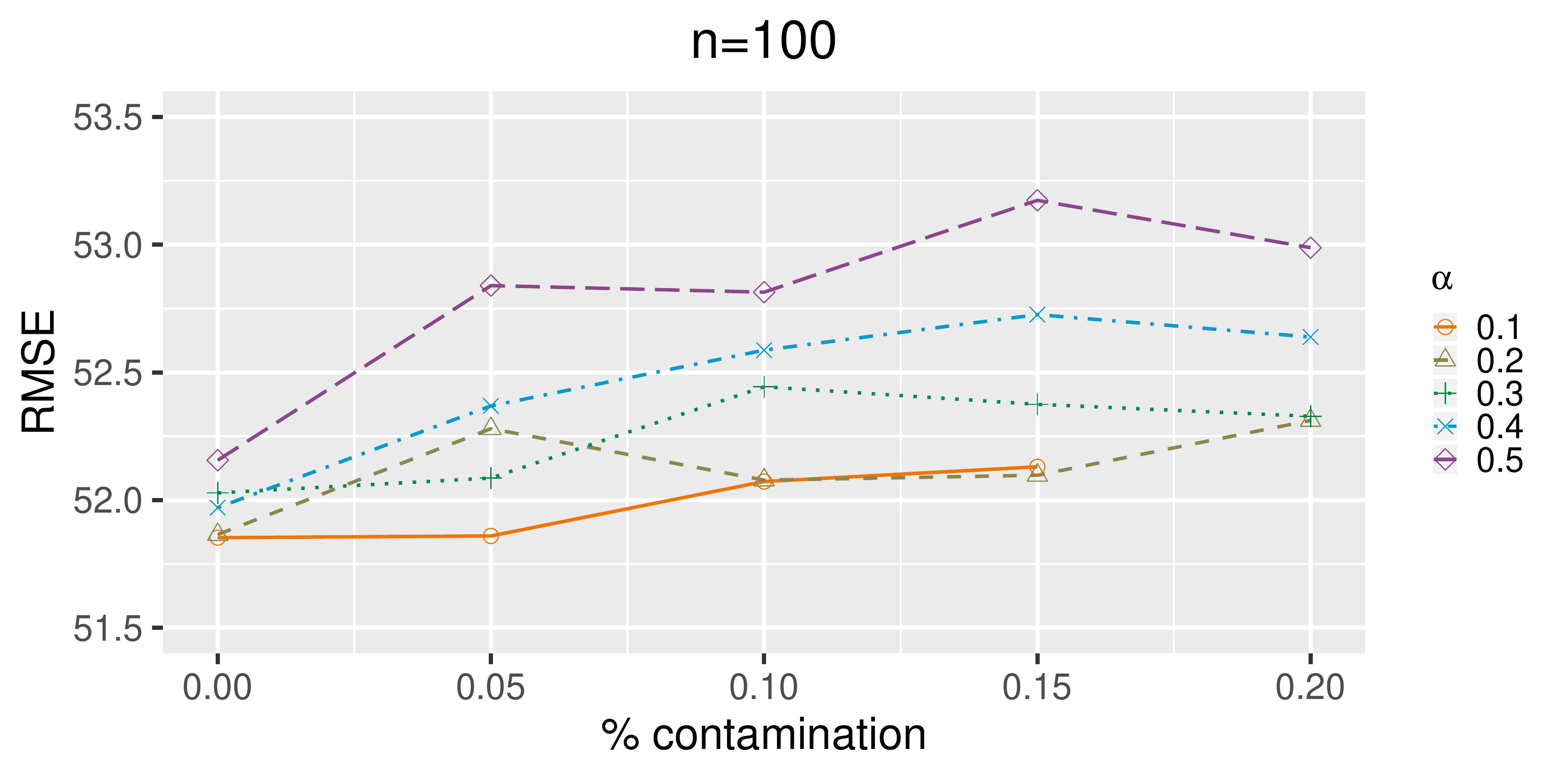}
	
	\includegraphics[height=5.5cm, width=8cm]{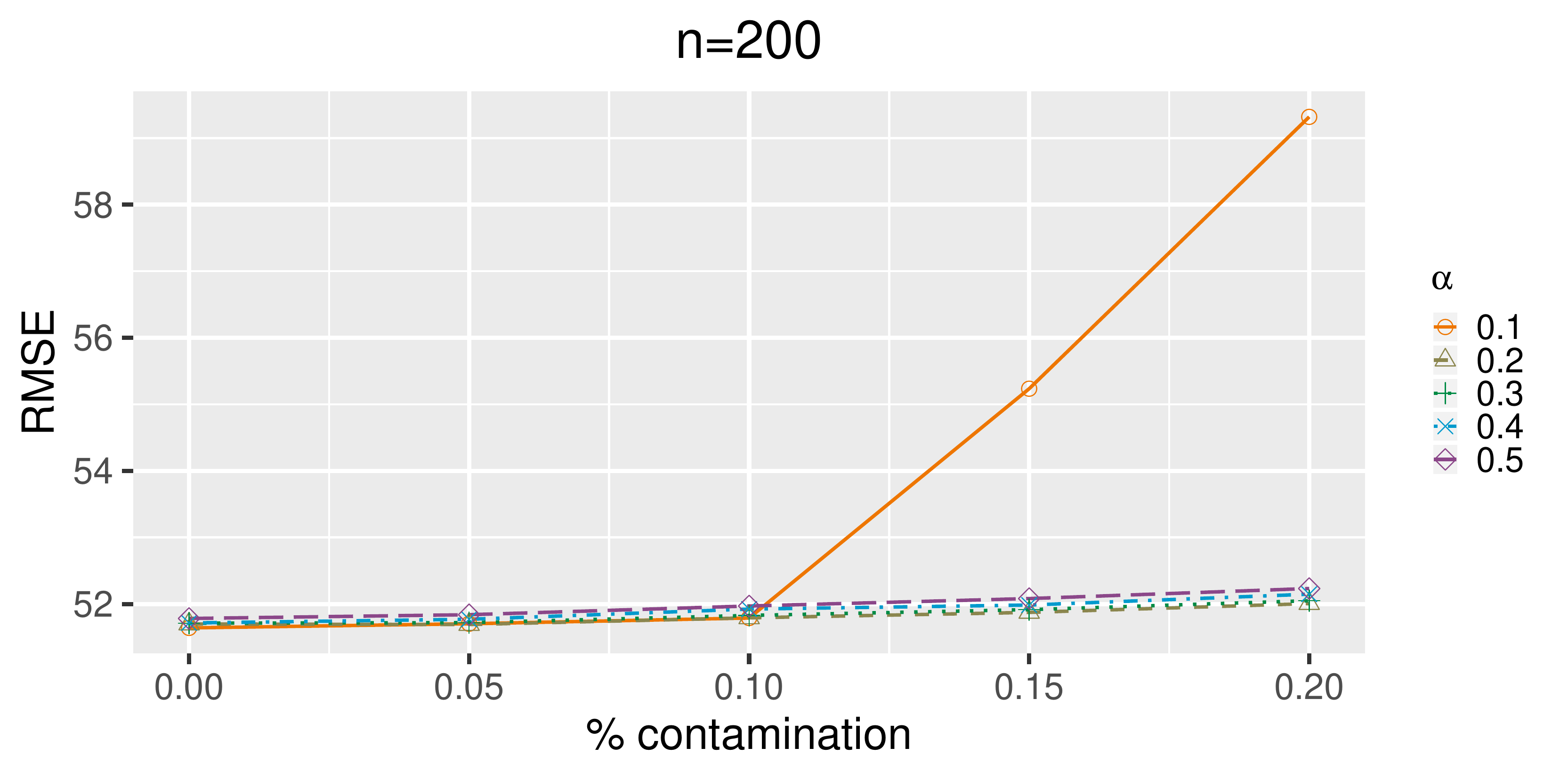}
	\includegraphics[height=5.5cm, width=8cm]{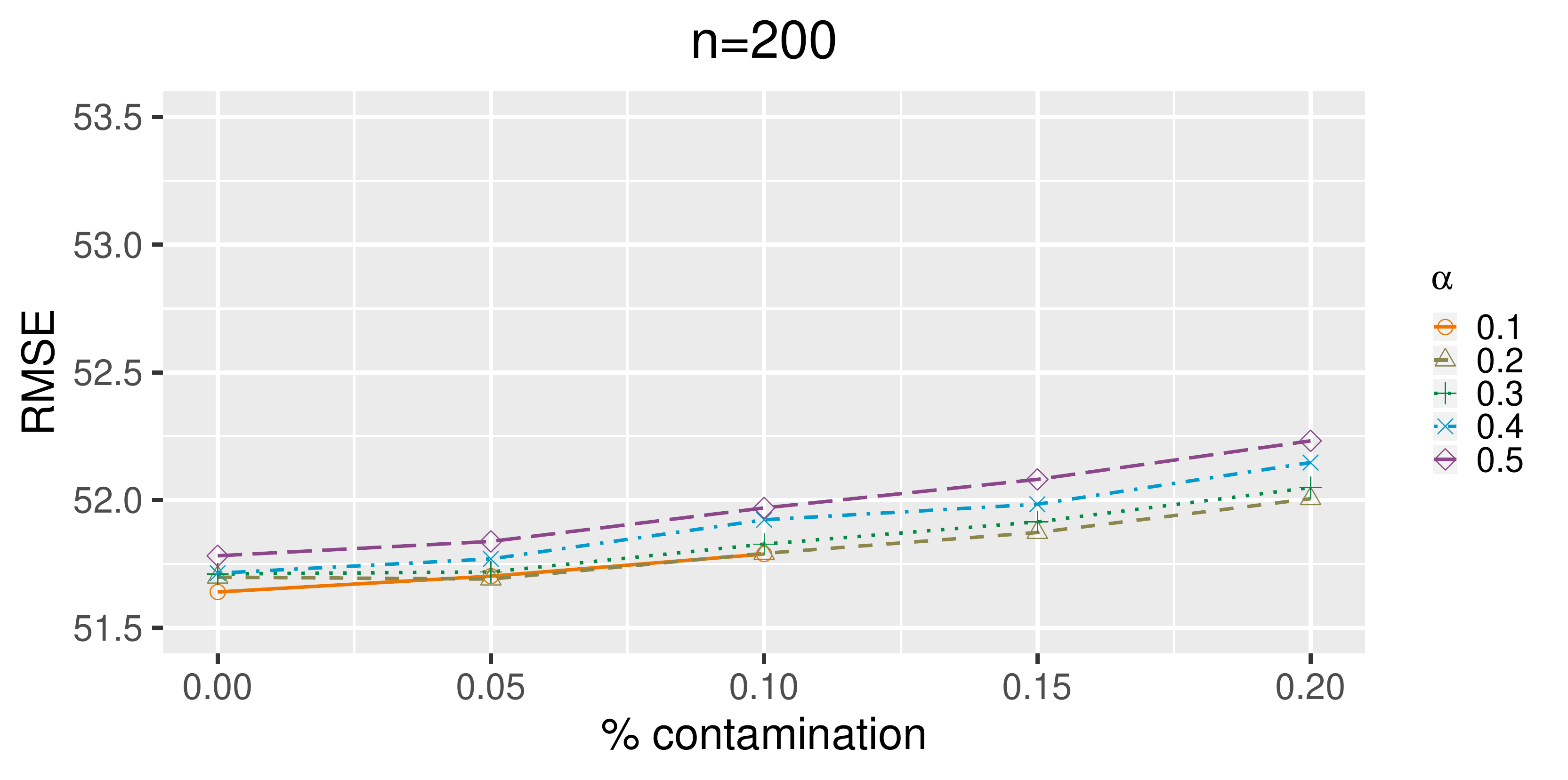}
	
	\includegraphics[height=5.5cm, width=8cm]{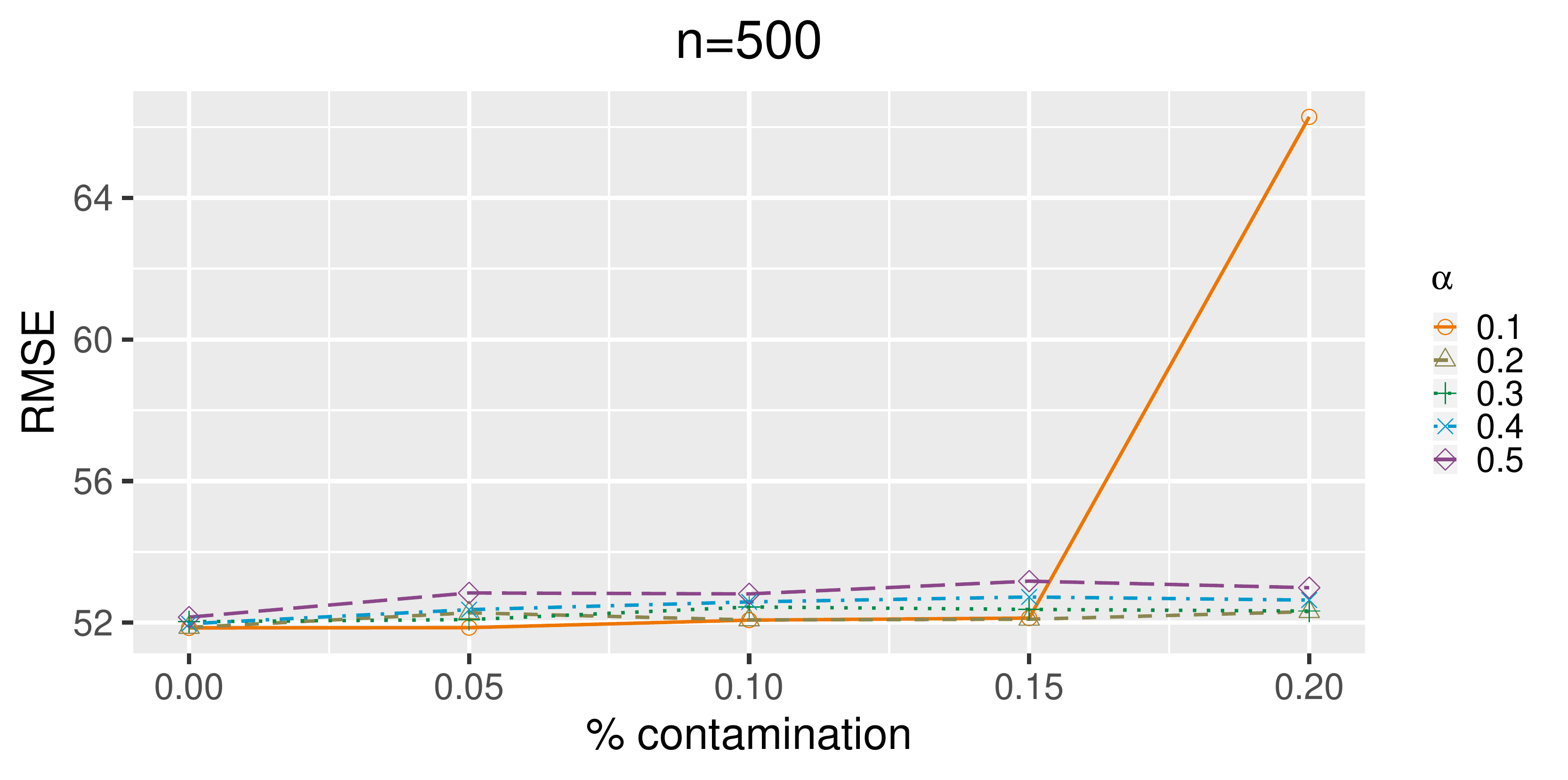}
	\includegraphics[height=5.5cm, width=8cm]{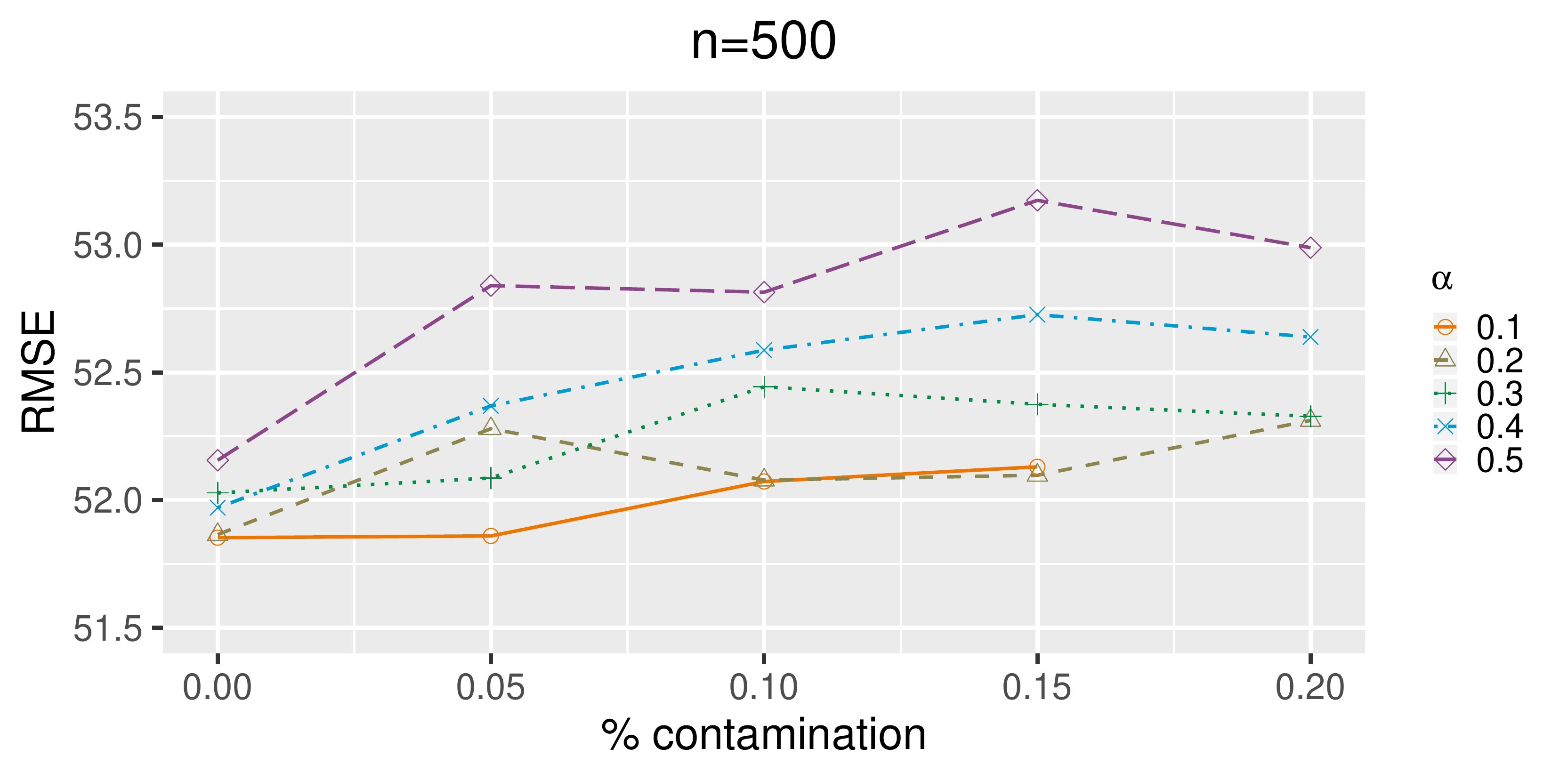}
	\caption{Data contamination agaisnt RMSE. On the right the figures are zoomed to $[51.5,53.5]$}
	\label{contaminacionVSerror}
\end{figure}

\begin{table}[H]
	\centering
	\caption{Performance measures obtained by different methods for $p=500$, strong signal and no outliers}
	\begin{tabular}{l|rrr|rrr|r}
	\hline
	Method	& MS($\widehat{\boldsymbol{\beta}}$) & TP($\widehat{\boldsymbol{\beta}}$) & TN($\widehat{\boldsymbol{\beta}}$) & MSES($\widehat{\boldsymbol{\beta}}$) &  MSEN($\widehat{\boldsymbol{\beta}}$)  &  EE($\widehat{\sigma}$) & APrB($\widehat{\boldsymbol{\beta}}$)  \\ 
	& & & &  $(10^{-2})$ & $(10^{-5})$ &  $(10^{-2})$ &  $(10^{-2})$ \\
	\hline
	L-lasso & 7.06 & 1.00 & 1.00 & 2.61 & 0.55 & 37.03 & 5.53  \\ 
	LS-SCAD & 4.94 & 0.99 & 1.00 & 16.26 & 0.00 & 72.30 & 7.33  \\ 
	LS-MCP & 4.94 & 0.99 & 1.00 & 11.18 & 0.00 & 55.69 & 6.47  \\ 
	LAD-lasso & 6.40 & 1.00 & 1.00 & 4.86 & 1.29 & 44.61 & 5.97  \\ 
	RLARS & 8.27 & 0.99 & 0.99 & 1.16 & 4.69 & 7.26 & 4.55  \\ 
	sLTS & 6.45 & 1.00 & 1.00 & 6.90 & 0.85 & 25.39 & 6.74  \\ 
	RANSAC & 10.99 & 1.00 & 0.99 & 6.90 & 12.89 & 11.17 & 6.78 \\ 
	 \hline
	DPD-lasso $\alpha = $ 0.1 & 8.15 & 1.00 & 0.99 & 4.80 & 0.69 & 18.48 & 5.92 \\ 
	DPD-lasso $\alpha = $ 0.3 & 8.57 & 1.00 & 0.99 & 4.90 & 1.05 & 18.84 & 5.95 \\ 
	DPD-lasso $\alpha = $ 0.5 & 11.38 & 0.99 & 0.99 & 77.40 & 77.35 & 19.31 & 9.88 \\ 
	DPD-lasso $\alpha = $ 0.7 & 6.47 & 1.00 & 1.00 & 28.03 & 51.73 & 23.73 & 7.63 \\ 
	DPD-lasso $\alpha = $ 1 & 9.91 & 0.98 & 0.99 & 48.91 & 125.45 & 20.27 & 9.27 \\ 
	 \hline
	LDPD-lasso $\alpha = $ 0.1 & 5.45 & 1.00 & 1.00 & 6.08 & 0.25 & 24.54 & 6.53  \\ 
	LDPD-lasso $\alpha = $ 0.2 & 5.37 & 1.00 & 1.00 & 6.19 & 0.27 & 24.92 & 6.55 \\ 
	LDPD-lasso $\alpha = $ 0.3 & 5.34 & 1.00 & 1.00 & 6.69 & 0.29 & 26.42 & 6.69  \\ 
	LDPD-lasso $\alpha = $ 0.4 & 5.31 & 1.00 & 1.00 & 8.50 & 0.32 & 31.29 & 7.18 \\ 
	LDPD-lasso $\alpha = $ 0.5 & 5.26 & 0.99 & 1.00 & 18.22 & 0.38 & 45.51 & 8.84 \\ 
	 \hline
	DPD-ncv $\alpha = $ 0.1 & 4.86 & 0.91 & 1.00 & 57.43 & 0.03 & 69.07 & 16.41 \\ 
	DPD-ncv $\alpha = $ 0.3 & 5.04 & 0.95 & 1.00 & 51.30 & 0.05 & 30.46 & 15.62 \\ 
	DPD-ncv $\alpha = $ 0.5 & 5.07 & 0.96 & 1.00 & 46.09 & 0.05 & 15.76 & 15.16 \\ 
	DPD-ncv $\alpha = $ 0.7 & 5.10 & 0.96 & 1.00 & 40.90 & 0.05 & 8.93 & 14.33 \\ 
	DPD-ncv $\alpha = $ 1 & 5.11 & 0.96 & 1.00 & 38.49 & 0.05 & 7.10 & 13.86 \\ 
	\hline
	MNPRPE-ncv $\alpha = $ 0.1 & 5.00 & 1.00 & 1.00 & 0.32 & 0.00 & 3.24 & 4.47  \\ 
	MNPRPE-ncv $\alpha = $ 0.2 & 5.00 & 1.00 & 1.00 & 0.34 & 0.00 & 3.43 & 4.48 \\ 
	MNPRPE-ncv $\alpha = $ 0.3 & 5.00 & 1.00 & 1.00 & 0.36 & 0.00 & 3.64 & 4.49  \\ 
	MNPRPE-ncv $\alpha = $ 0.4 & 5.00 & 1.00 & 1.00 & 0.37 & 0.00 & 3.88 & 4.50  \\ 
	MNPRPE-ncv $\alpha = $ 0.5 & 5.00 & 1.00 & 1.00 & 0.39 & 0.00 & 4.14 & 4.51  \\ 
	\hline
	\end{tabular}
	\label{p500e0signal1}
\end{table}

\begin{table}[H]
	\centering
	\caption{Performance measures obtained by different methods for $p=500$, weak signal and no outliers}
	\begin{tabular}{l|rrr|rrr|r}
	\hline
	Method	& MS($\widehat{\boldsymbol{\beta}}$) & TP($\widehat{\boldsymbol{\beta}}$) & TN($\widehat{\boldsymbol{\beta}}$) & MSES($\widehat{\boldsymbol{\beta}}$) &  MSES($\widehat{\boldsymbol{\beta}}$)  &  EE($\widehat{\sigma}$) & APrB($\widehat{\boldsymbol{\beta}}$)  \\ 
	& & & &  $(10^{-2})$ & $(10^{-5})$ &  $(10^{-2})$ &  $(10^{-2})$ \\
	\hline 
	L-lasso & 8.14 & 1.00 & 0.99 & 2.47 & 0.94 & 34.15 & 5.27  \\ 
	LS-SCAD & 9.69 & 1.00 & 0.99 & 0.41 & 0.60 & 18.26 & 4.47\\ 
	LS-MCP & 6.58 & 1.00 & 1.00 & 0.33 & 0.80 & 19.27 & 4.40  \\ 
	LAD-lasso & 6.38 & 1.00 & 1.00 & 4.92 & 1.30 & 43.61 & 5.80 \\ 
	RLARS & 14.27 & 1.00 & 0.98 & 0.76 & 16.16 & 13.46 & 5.17 \\ 
	sLTS & 37.70 & 0.99 & 0.93 & 8.16 & 16.92 & 20.54 & 6.87  \\ 
	RANSAC & 14.71 & 0.99 & 0.98 & 5.24 & 18.34 & 24.63 & 5.88 \\ 
	 \hline
	DPD-lasso $\alpha = $ 0.1 & 8.72 & 1.00 & 0.99 & 4.54 & 0.77 & 17.46 & 5.76 \\ 
	DPD-lasso $\alpha = $ 0.3 & 9.16 & 1.00 & 0.99 & 4.75 & 1.21 & 18.53 & 5.93 \\ 
	DPD-lasso $\alpha = $ 0.5 & 10.07 & 1.00 & 0.99 & 6.24 & 2.13 & 19.69 & 6.47 \\ 
	DPD-lasso $\alpha = $ 0.7 & 5.89 & 1.00 & 1.00 & 6.91 & 0.93 & 23.59 & 6.67 \\ 
	DPD-lasso $\alpha = $ 1 & 13.67 & 0.92 & 0.98 & 21.96 & 17.04 & 21.02 & 9.91 \\ 
	\hline
	LDPD-lasso $\alpha = $ 0.1 & 5.45 & 1.00 & 1.00 & 6.09 & 0.25 & 24.57 & 6.54  \\ 
	LDPD-lasso $\alpha = $ 0.2 & 5.38 & 1.00 & 1.00 & 6.19 & 0.27 & 24.92 & 6.55  \\ 
	LDPD-lasso $\alpha = $ 0.3 & 5.33 & 1.00 & 1.00 & 6.67 & 0.28 & 26.40 & 6.69 \\ 
	LDPD-lasso $\alpha = $ 0.4 & 5.29 & 1.00 & 1.00 & 8.24 & 0.31 & 31.00 & 7.14 \\ 
	LDPD-lasso $\alpha = $ 0.5 & 5.22 & 0.99 & 1.00 & 13.09 & 0.37 & 42.54 & 8.27 \\ 
	\hline
	DPD-ncv $\alpha = $ 0.1 & 5.12 & 0.99 & 1.00 & 0.89 & 0.04 & 5.70 & 4.52\\ 
	DPD-ncv $\alpha = $ 0.3 & 5.01 & 0.99 & 1.00 & 1.11 & 0.04 & 11.10 & 4.64  \\ 
	DPD-ncv $\alpha = $ 0.5 & 4.98 & 0.98 & 1.00 & 1.37 & 0.06 & 15.07 & 4.71  \\ 
	DPD-ncv $\alpha = $ 0.7 & 5.00 & 0.98 & 1.00 & 1.59 & 0.14 & 18.11 & 4.66  \\ 
	DPD-ncv $\alpha = $ 1 & 4.97 & 0.98 & 1.00 & 1.97 & 0.16 & 21.39 & 4.70  \\ 
	\hline
	MNPRPE-ncv $\alpha = $ 0.1 & 5.15 & 0.99 & 1.00 & 0.80 & 0.04 & 3.40 & 4.51  \\ 
	MNPRPE-ncv $\alpha = $ 0.2 & 5.06 & 0.99 & 1.00 & 0.78 & 0.02 & 3.65 & 4.55 \\ 
	MNPRPE-ncv $\alpha = $ 0.3 & 5.05 & 0.99 & 1.00 & 0.78 & 0.02 & 3.88 & 4.52  \\ 
	MNPRPE-ncv $\alpha = $ 0.4 & 5.01 & 0.99 & 1.00 & 0.80 & 0.01 & 4.10 & 4.54  \\ 
	MNPRPE-ncv $\alpha = $ 0.5 & 5.00 & 0.99 & 1.00 & 0.82 & 0.01 & 4.32 & 4.58  \\ 
	\hline
	
	\end{tabular}
	\label{p500e0signal0}
\end{table}

\begin{table}[H]
	\centering
	\caption{Performance measures obtained by different methods for $p=500$, strong signal and $Y-$outliers}
	\begin{tabular}{l|rrr|rrr|r}
	\hline
	Method	& MS($\widehat{\boldsymbol{\beta}}$) & TP($\widehat{\boldsymbol{\beta}}$) & TN($\widehat{\boldsymbol{\beta}}$) & MSES($\widehat{\boldsymbol{\beta}}$) &  MSES($\widehat{\boldsymbol{\beta}}$)  &  EE($\widehat{\sigma}$) & APrB($\widehat{\boldsymbol{\beta}}$)  \\ 
	& & & &  $(10^{-2})$ & $(10^{-5})$ &  $(10^{-2})$ &  $(10^{-2})$ \\
	\hline
	LS-lasso & 6.16 & 0.79 & 1.00 & 345.92 & 72.66 & 673.30 & 36.24  \\ 
	LS-SCAD & 13.74 & 0.88 & 0.98 & 104.87 & 246.36 & 311.34 & 20.26  \\ 
	LS-MCP & 6.52 & 0.79 & 0.99 & 103.49 & 175.28 & 317.72 & 19.80 \\ 
	LAD-lasso & 9.78 & 0.93 & 0.99 & 85.45 & 63.01 & 280.92 & 19.21 \\ 
	RLARS & 11.62 & 0.91 & 0.99 & 18.02 & 21.70 & 37.14 & 6.81  \\ 
	sLTS & 6.97 & 1.00 & 1.00 & 5.28 & 1.10 & 32.79 & 6.22  \\ 
	RANSAC & 11.90 & 1.00 & 0.99 & 10.39 & 20.98 & 11.85 & 7.77  \\ 
	  \hline
	 DPD-lasso $\alpha = $ 0.1 & 8.78 & 1.00 & 0.99 & 4.69 & 1.15 & 17.37 & 6.18 \\ 
	 DPD-lasso $\alpha = $ 0.3 & 8.49 & 1.00 & 0.99 & 4.90 & 1.16 & 19.12 & 5.97 \\ 
	 DPD-lasso $\alpha = $ 0.5 & 10.47 & 0.99 & 0.99 & 27.59 & 36.49 & 17.33 & 7.06 \\ 
	 DPD-lasso $\alpha = $ 0.7 & 6.89 & 1.00 & 1.00 & 6.13 & 1.59 & 22.42 & 6.44 \\ 
	 DPD-lasso $\alpha = $ 1 & 12.92 & 0.99 & 0.98 & 55.27 & 77.52 & 16.89 & 10.59 \\
	 \hline
	 LDPD-lasso $\alpha = $ 0.1 & 5.51 & 1.00 & 1.00 & 9.40 & 0.40 & 30.35 & 7.31  \\ 
	 LDPD-lasso $\alpha = $ 0.2 & 5.57 & 1.00 & 1.00 & 6.25 & 0.41 & 24.59 & 6.65 \\ 
	 LDPD-lasso $\alpha = $ 0.3 & 5.50 & 1.00 & 1.00 & 6.91 & 0.42 & 26.42 & 6.84  \\ 
	 LDPD-lasso $\alpha = $ 0.4 & 5.48 & 1.00 & 1.00 & 9.17 & 0.44 & 31.87 & 7.42  \\ 
	 LDPD-lasso $\alpha = $ 0.5 & 5.35 & 0.99 & 1.00 & 39.59 & 0.52 & 54.03 & 10.67  \\ 
	 \hline
	DPD-ncv $\alpha = $ 0.1 & 5.01 & 1.00 & 1.00 & 0.95 & 0.02 & 7.51 & 4.56  \\ 
	DPD-ncv $\alpha = $ 0.3 & 5.01 & 1.00 & 1.00 & 1.12 & 0.02 & 6.43 & 4.60 \\ 
	DPD-ncv $\alpha = $ 0.5 & 5.00 & 1.00 & 1.00 & 1.29 & 0.00 & 7.76 & 4.69  \\ 
	DPD-ncv $\alpha = $ 0.7 & 5.00 & 1.00 & 1.00 & 1.62 & 0.00 & 8.93 & 4.77 \\ 
	DPD-ncv $\alpha = $ 1 & 5.00 & 1.00 & 1.00 & 2.47 & 0.00 & 10.55 & 4.95  \\ 
	\hline
	MNPRPE-ncv $\alpha = $ 0.1 & 5.03 & 1.00 & 1.00 & 0.34 & 0.02 & 3.40 & 4.49  \\ 
	MNPRPE-ncv $\alpha = $ 0.2 & 5.02 & 1.00 & 1.00 & 0.35 & 0.00 & 3.61 & 4.51  \\ 
	MNPRPE-ncv $\alpha = $ 0.3 & 5.02 & 1.00 & 1.00 & 0.37 & 0.00 & 3.83 & 4.52 \\ 
	MNPRPE-ncv $\alpha = $ 0.4 & 5.01 & 1.00 & 1.00 & 0.39 & 0.00 & 4.10 & 4.54  \\ 
	MNPRPE-ncv $\alpha = $ 0.5 & 5.00 & 1.00 & 1.00 & 0.41 & 0.00 & 4.41 & 4.54  \\ 
	\hline
		\end{tabular}
	\label{p500e1signal1}
\end{table}
\begin{table}[H]
	\centering
	\caption{TPerformance measures obtained by different methods for $p=500$, weak signal and $Y-$outliers}
		\begin{tabular}{l|rrr|rrr|r}
		\hline
		Method	& MS($\widehat{\boldsymbol{\beta}}$) & TP($\widehat{\boldsymbol{\beta}}$) & TN($\widehat{\boldsymbol{\beta}}$) & MSES($\widehat{\boldsymbol{\beta}}$) &  MSES($\widehat{\boldsymbol{\beta}}$)  &  EE($\widehat{\sigma}$) & APrB($\widehat{\boldsymbol{\beta}}$)  \\ 
		& & & &  $(10^{-2})$ & $(10^{-5})$ &  $(10^{-2})$ &  $(10^{-2})$ \\
		\hline
		LS-lasso & 0.81 & 0.04 & 1.00 & 131.43 & 18.32 & 459.52 & 23.75  \\ 
		LS-SCAD & 10.14 & 0.34 & 0.98 & 101.51 & 255.24 & 353.54 & 22.00 \\ 
		LS-MCP & 4.29 & 0.25 & 0.99 & 104.64 & 203.06 & 364.78 & 20.91  \\ 
		LAD-lasso & 6.34 & 0.65 & 0.99 & 67.10 & 38.00 & 277.60 & 17.32  \\ 
		RLARS & 8.22 & 0.94 & 0.99 & 2.92 & 7.52 & 12.30 & 5.19  \\ 
		sLTS & 41.82 & 1.00 & 0.93 & 4.87 & 15.94 & 22.23 & 6.01  \\ 
		RANSAC & 14.56 & 0.97 & 0.98 & 7.38 & 25.21 & 23.99 & 7.58  \\ 
		 \hline
		DPD-lasso $\alpha = $ 0.1 & 8.15 & 1.00 & 0.99 & 4.80 & 0.69 & 18.48 & 5.92 \\ 
		DPD-lasso $\alpha = $ 0.3 & 8.57 & 1.00 & 0.99 & 4.90 & 1.05 & 18.84 & 5.95 \\ 
		DPD-lasso $\alpha = $ 0.5 & 11.38 & 0.99 & 0.99 & 77.40 & 77.35 & 19.31 & 9.88 \\ 
		DPD-lasso $\alpha = $ 0.7 & 6.47 & 1.00 & 1.00 & 28.03 & 51.73 & 23.73 & 7.63 \\ 
		DPD-lasso $\alpha = $ 1 & 9.91 & 0.98 & 0.99 & 48.91 & 125.45 & 20.27 & 9.27 \\ 
		 \hline
		LDPD-lasso $\alpha = $ 0.1 & 5.52 & 1.00 & 1.00 & 6.23 & 0.40 & 24.56 & 6.67  \\ 
		LDPD-lasso $\alpha = $ 0.2 & 5.56 & 1.00 & 1.00 & 6.29 & 0.41 & 24.65 & 6.67  \\ 
		LDPD-lasso $\alpha = $ 0.3 & 5.49 & 1.00 & 1.00 & 6.87 & 0.42 & 26.37 & 6.84  \\ 
		LDPD-lasso $\alpha = $ 0.4 & 5.47 & 1.00 & 1.00 & 8.83 & 0.44 & 31.53 & 7.37  \\ 
		LDPD-lasso $\alpha = $ 0.5 & 5.32 & 0.98 & 1.00 & 14.66 & 0.51 & 44.40 & 8.57 \\ 
		\hline
		DPD-ncv $\alpha = $ 0.1 & 5.26 & 0.99 & 1.00 & 1.15 & 0.06 & 4.15 & 4.51 \\ 
		DPD-ncv $\alpha = $ 0.3 & 5.08 & 0.99 & 1.00 & 1.21 & 0.03 & 5.87 & 4.63  \\ 
		DPD-ncv $\alpha = $ 0.5 & 5.06 & 0.99 & 1.00 & 1.27 & 0.02 & 7.77 & 4.65  \\ 
		DPD-ncv $\alpha = $ 0.7 & 4.98 & 0.98 & 1.00 & 1.61 & 0.01 & 9.61 & 4.65  \\ 
		DPD-ncv $\alpha = $ 1 & 4.88 & 0.97 & 1.00 & 2.15 & 0.02 & 12.03 & 4.75  \\ 
		 \hline
		MNPRPE-ncv $\alpha = $ 0.1 & 5.52 & 1.00 & 1.00 & 1.10 & 0.09 & 3.88 & 4.60  \\ 
		MNPRPE-ncv $\alpha = $ 0.2 & 5.27 & 0.99 & 1.00 & 1.14 & 0.05 & 4.04 & 4.63  \\ 
		MNPRPE-ncv $\alpha = $ 0.3 & 5.14 & 0.99 & 1.00 & 1.17 & 0.03 & 4.27 & 4.66  \\ 
		MNPRPE-ncv $\alpha = $ 0.4 & 5.00 & 0.98 & 1.00 & 1.28 & 0.01 & 4.72 & 4.71  \\ 
		MNPRPE-ncv $\alpha = $ 0.5 & 4.96 & 0.98 & 1.00 & 1.38 & 0.00 & 4.96 & 4.76  \\ 
		\hline
	\end{tabular}
	\label{p500e1signal0}
\end{table}

\begin{table}[H]
	\centering
	\caption{Performance measures obtained by different methods for $p=500$, strong signal and $\boldsymbol{X}-$outliers}
		\begin{tabular}{l|rrr|rrr|r}
		\hline
		Method	& MS($\widehat{\boldsymbol{\beta}}$) & TP($\widehat{\boldsymbol{\beta}}$) & TN($\widehat{\boldsymbol{\beta}}$) & MSES($\widehat{\boldsymbol{\beta}}$) &  MSES($\widehat{\boldsymbol{\beta}}$)  &  EE($\widehat{\sigma}$) & APrB($\widehat{\boldsymbol{\beta}}$)  \\ 
		& & & &  $(10^{-2})$ & $(10^{-5})$ &  $(10^{-2})$ &  $(10^{-2})$ \\
		\hline
		LS-lasso & 6.88 & 1.00 & 1.00 & 2.70 & 0.65 & 37.13 & 5.33  \\ 
		LS-SCAD & 4.95 & 0.99 & 1.00 & 15.92 & 0.00 & 71.73 & 6.83  \\ 
		LS-MCP & 4.95 & 0.99 & 1.00 & 10.62 & 0.00 & 55.53 & 5.98  \\ 
		LAD-lasso & 6.41 & 1.00 & 1.00 & 4.90 & 1.27 & 43.57 & 5.86  \\ 
		RLARS & 8.05 & 1.00 & 0.99 & 0.70 & 4.31 & 5.96 & 4.61  \\ 
		sLTS & 6.96 & 1.00 & 1.00 & 7.50 & 1.36 & 25.43 & 6.73 \\ 
		RANSAC & 10.54 & 1.00 & 0.99 & 6.33 & 9.54 & 15.11 & 6.10  \\ 
		 \hline
		 DPD-lasso $\alpha = $ 0.1 & 5.42 & 1.00 & 1.00 & 6.44 & 0.33 & 25.42 & 6.53 \\ 
		 DPD-lasso $\alpha = $ 0.3 & 5.47 & 1.00 & 1.00 & 6.96 & 0.37 & 26.54 & 6.65 \\ 
		 DPD-lasso  $\alpha = $ 0.5 & 5.36 & 1.00 & 1.00 & 17.44 & 0.60 & 46.84 & 8.89 \\ 
		 DPD-lasso $\alpha = $ 0.7 & 5.95 & 1.00 & 1.00 & 6.91 & 0.97 & 23.41 & 6.64 \\ 
		 DPD-lasso $\alpha = $ 1 & 10.57 & 0.99 & 0.99 & 50.44 & 58.10 & 19.78 & 9.68 \\ 
		\hline
		LDPD-lasso $\alpha = $ 0.1 & 5.45 & 1.00 & 1.00 & 6.08 & 0.25 & 24.54 & 6.53 \\ 
		LDPD-lasso $\alpha = $ 0.2 & 5.37 & 1.00 & 1.00 & 6.19 & 0.27 & 24.92 & 6.55  \\ 
		LDPD-lasso $\alpha = $ 0.3 & 5.34 & 1.00 & 1.00 & 6.69 & 0.29 & 26.42 & 6.69 \\ 
		LDPD-lasso $\alpha = $ 0.4 & 5.31 & 1.00 & 1.00 & 8.50 & 0.32 & 31.29 & 7.18 \\ 
		LDPD-lasso $\alpha = $ 0.5 & 5.26 & 0.99 & 1.00 & 18.22 & 0.38 & 45.51 & 8.84 \\ 
		 \hline
		DPD-ncv $\alpha = $ 0.1 & 5.00 & 1.00 & 1.00 & 0.34 & 0.00 & 4.53 & 4.47  \\ 
		DPD-ncv $\alpha = $ 0.3 & 5.00 & 1.00 & 1.00 & 0.43 & 0.00 & 8.49 & 4.47  \\ 
		DPD-ncv $\alpha = $ 0.5 & 5.00 & 1.00 & 1.00 & 0.59 & 0.00 & 11.79 & 4.47  \\ 
		DPD-ncv $\alpha = $ 0.7 & 4.99 & 1.00 & 1.00 & 1.03 & 0.00 & 14.42 & 4.57  \\ 
		DPD-ncv $\alpha = $ 1 & 4.99 & 1.00 & 1.00 & 1.37 & 0.00 & 17.44 & 4.56  \\ 
		\hline
		MNPRPE-ncv $\alpha = $ 0.1 & 5.00 & 1.00 & 1.00 & 0.32 & 0.00 & 3.26 & 4.49 \\ 
		MNPRPE-ncv $\alpha = $ 0.2 & 5.00 & 1.00 & 1.00 & 0.34 & 0.00 & 3.45 & 4.50  \\ 
		MNPRPE-ncv $\alpha = $ 0.3 & 5.00 & 1.00 & 1.00 & 0.36 & 0.00 & 3.67 & 4.50  \\ 
		MNPRPE-ncv $\alpha = $ 0.4 & 5.00 & 1.00 & 1.00 & 0.38 & 0.00 & 3.91 & 4.51  \\ 
		MNPRPE-ncv $\alpha = $ 0.5 & 5.00 & 1.00 & 1.00 & 0.40 & 0.00 & 4.17 & 4.53  \\ 
		\hline
	\end{tabular}
	\label{p500e1signal1x1}
\end{table}

\begin{table}[H]
	\centering
	\caption{Performance measures obtained by different methods for $p=500$, weak signal and $\boldsymbol{X}-$outliers}
	\begin{tabular}{l|rrr|rrr|r}
	\hline
	Method	& MS($\widehat{\boldsymbol{\beta}}$) & TP($\widehat{\boldsymbol{\beta}}$) & TN($\widehat{\boldsymbol{\beta}}$) & MSES($\widehat{\boldsymbol{\beta}}$) &  MSES($\widehat{\boldsymbol{\beta}}$)  &  EE($\widehat{\sigma}$) & APrB($\widehat{\boldsymbol{\beta}}$)  \\ 
	& & & &  $(10^{-2})$ & $(10^{-5})$ &  $(10^{-2})$ &  $(10^{-2})$ \\
	\hline
	LS-lasso & 8.18 & 1.00 & 0.99 & 2.47 & 0.94 & 34.17 & 5.27  \\ 
	LS-SCAD & 9.68 & 1.00 & 0.99 & 0.41 & 0.60 & 18.27 & 4.48  \\ 
	LS-MCP & 6.58 & 1.00 & 1.00 & 0.33 & 0.80 & 19.27 & 4.40  \\ 
	LAD-lasso & 6.39 & 1.00 & 1.00 & 4.92 & 1.30 & 43.61 & 5.80\\ 
	RLARS & 14.27 & 1.00 & 0.98 & 0.76 & 16.16 & 13.46 & 5.17 \\ 
	sLTS & 37.70 & 0.99 & 0.93 & 8.16 & 16.92 & 20.54 & 6.87  \\ 
	RANSAC & 14.58 & 1.00 & 0.98 & 4.94 & 19.09 & 24.44 & 6.58 \\ 
	\hline 
	DPD-lasso $\alpha = $ 0.1 & 5.42 & 1.00 & 1.00 & 6.43 & 0.33 & 25.39 & 6.53 \\ 
	DPD-lasso  $\alpha = $ 0.3 & 5.46 & 1.00 & 1.00 & 6.99 & 0.37 & 26.64 & 6.66 \\ 
	DPD-lasso $\alpha = $ 0.5 & 5.37 & 0.99 & 1.00 & 11.71 & 0.50 & 39.55 & 8.01 \\ 
	DPD-lasso  $\alpha = $ 0.7 & 6.11 & 0.99 & 1.00 & 7.70 & 1.48 & 23.58 & 6.69 \\ 
	DPD-lasso  $\alpha = $ 1 & 12.65 & 0.93 & 0.98 & 17.91 & 11.24 & 20.89 & 9.03 \\  
	\hline
	LDPD-lasso $\alpha = $ 0.1 & 5.45 & 1.00 & 1.00 & 6.09 & 0.25 & 24.57 & 6.54  \\ 
	LDPD-lasso $\alpha = $ 0.2 & 5.38 & 1.00 & 1.00 & 6.19 & 0.27 & 24.92 & 6.55 \\ 
	LDPD-lasso $\alpha = $ 0.3 & 5.33 & 1.00 & 1.00 & 6.67 & 0.28 & 26.40 & 6.69 \\ 
	LDPD-lasso $\alpha = $ 0.4 & 5.29 & 1.00 & 1.00 & 8.24 & 0.31 & 31.00 & 7.14 \\ 
	LDPD-lasso $\alpha = $ 0.5 & 5.22 & 0.99 & 1.00 & 13.09 & 0.37 & 42.54 & 8.27\\ 
	\hline
	DPD-ncv $\alpha = $ 0.1 & 5.11 & 0.99 & 1.00 & 0.92 & 0.04 & 5.67 & 4.54 \\ 
	DPD-ncv $\alpha = $ 0.3 & 4.98 & 0.98 & 1.00 & 1.31 & 0.04 & 11.02 & 4.66\\ 
	DPD-ncv $\alpha = $ 0.5 & 4.93 & 0.98 & 1.00 & 1.61 & 0.06 & 14.98 & 4.75 \\ 
	DPD-ncv $\alpha = $ 0.7 & 4.95 & 0.97 & 1.00 & 1.93 & 0.15 & 18.03 & 4.71  \\ 
	DPD-ncv $\alpha = $ 1 & 4.93 & 0.96 & 1.00 & 2.71 & 0.30 & 21.34 & 4.92 \\ 
	\hline
	MNPRPE-ncv $\alpha = $ 0.1 & 5.17 & 0.99 & 1.00 & 0.79 & 0.04 & 3.39 & 4.51\\ 
	MNPRPE-ncv $\alpha = $ 0.2 & 5.06 & 0.99 & 1.00 & 0.78 & 0.02 & 3.59 & 4.56  \\ 
	MNPRPE-ncv $\alpha = $ 0.3 & 5.06 & 0.99 & 1.00 & 0.75 & 0.03 & 3.82 & 4.54\\ 
	MNPRPE-ncv $\alpha = $ 0.4 & 5.05 & 0.99 & 1.00 & 0.72 & 0.02 & 4.04 & 4.55 \\ 
	MNPRPE-ncv $\alpha = $ 0.5 & 5.05 & 0.99 & 1.00 & 0.76 & 0.03 & 4.35 & 4.57 \\ 
	\hline
	\end{tabular}
	\label{p500e1signal0x1}
\end{table}

\section{Glioblastoma gene expression data analysis}

We now apply our proposed method to glioblastoma gene expression data from Hovarth et al. (2006). Glioblastoma is the most prevalent primary malignant brain tumor among adults and one of the most lethal cancers. Patients with such tumor have a median survival of 15 months from the time of diagnosis despite surgery, radiation, and chemotherapy. The dataset contains  global gene expression for 3600 genes on two independent groups of patients obtained by high-density Affymetrix arrays; Group 1 and Group 2 include $55$ and $65$ observations, respectively. However both groups contain few patients who were alive at the last followup and they must be excluded in our analysis, resulting in $n_1=50$ patients on Group 1 and $n_2=61$ on Group 2. Wang et al. (2011) and Rajaratnam et al. (2019) have used this dataset to test random LASSO and influence-LASSO  respectively.

To fit the LRM each patient's gene expression is scaled and logarithm (in base 10) transformation is applied on each observation. We use the logarithm of time to death as the response variable. 
We use Group 1 as train set to compute the parameter estimates $\widehat{\boldsymbol{\beta}}$ and $\hat{\sigma}$ and Group 2 as test set. Then we evaluate the Prediction Bias (BIAS), Mean absolute error (ABS), Mean Square Prediction Error (MSPE) and the maximum and minimum absolute error (MAXerror and MINerror) in both datasets to compare the estimate with observed data. These error measures are calculated as follows
\begin{equation*}
\operatorname{BIAS} = \frac{1}{n}\sum_{i=1}^n\left(y_i-\boldsymbol{x}_i^T\widehat{\boldsymbol{\beta}}\right),\hspace{0.5cm}
\operatorname{ABS} = \parallel \boldsymbol{y}-\mathbb{X} \widehat{\boldsymbol{\beta}} \parallel_1,\hspace{0.5cm}
\operatorname{MSPE} = \frac{1}{n} \parallel \boldsymbol{y}-\mathbb{X} \widehat{\boldsymbol{\beta}} \parallel^2_2, 
\end{equation*}
\begin{equation*}
\operatorname{MAX} = \operatorname{max}_{1\leq i \leq n}|y_i-\boldsymbol{x}_i^T\widehat{\boldsymbol{\beta}}|, \hspace{0.5cm}
\operatorname{MIN} = \operatorname{min}_{1\leq i \leq n}|y_i-\boldsymbol{x}_i^T\widehat{\boldsymbol{\beta}}|.
\end{equation*}
Due to scarce sample size the model is more sensitive to hyperparameter selection. If large values of the hyperparameter $\lambda$ are chosen, all $\boldsymbol{\beta}$ coefficients are estimates as zero. To avoid the null estimate, we select $\lambda$ over a grid from value $0.01$ to $0.037$ according to HBIC criterion. 

In order to assess the accuracy of the proposed method, the data are fitted on several competing methods including penalized least square methods such as LS-LASSO and LS-SCAD, robust methods like RLARS, LASSO penalized DPD and LDPD (with $\alpha= 0.3,0.6,0.9$), and the nonconcave penalized DPD  with SCAD penalty (DPD-ncv) and the hyperparameter values $\alpha= 0.3,0.6,0.9$. Moreover, our proposed MNPRPE is fitted for hyperparameter values $\alpha = 0.1,0.2,0.3,0.5$.

Tables \ref{resultsX1} and \ref{resultsX2} contain the five error measures for the seven methods to study model fitness on train data (Group 1) and test data (Group 2). DPD-ncv, LDPD-LASSO and MNPRPE are the best  estimating methods in all settings, for both train and test data.
The lowest error on train data corresponds to DPD-ncv, followed by our proposed method MNPRPE. However, on test set both DPD-ncv and MNPRPE have similar performance. 

\begin{table}[H]
	\centering
	\caption{Error measures for Group 1 (train) dataset}
	\begin{tabular}{r|rrrrr}
		\hline
		& BIAS &  ABS  & MSPE &  MAX & MIN  \\ 
		\hline
		LS-LASSO & -0.00 & 0.75 & 0.94 & 3.50 & 0.02 \\ 
		LS-SCAD & 0.00 & 0.72 & 0.87 & 3.41 & 0.01 \\ 
		RLARS & -0.11 & 0.34 & 0.44 & 3.83 & 0.00 \\ 
		\hline
		MNPRPE-SCAD  $\alpha = $ 0.1 & -0.00 & 0.21 & 0.26 & 3.35 & 0.00\\ 
		MNPRPE-SCAD  $\alpha = $ 0.2 & -0.00 & 0.39 & 0.52 & 4.12 & 0.04 \\ 
		MNPRPE-SCAD  $\alpha = $ 0.3 & -0.00 & 0.36 & 0.52 & 3.97 & 0.01 \\ 
		MNPRPE-SCAD $\alpha = $ 0.5 & -0.00 & 0.34 & 0.47 & 4.10 & 0.00  \\ 
		\hline
	\end{tabular}

\end{table}
\begin{table}
\centering
\begin{tabular}{r|rrrrr}
	\hline
	& BIAS &  ABS  & MSPE &  MAX & MIN  \\ 
	\hline
	DPD-SCAD  $\alpha = $ 0.3 & -0.00 & 0.12 & 0.12 & 2.36 & 0.01 \\ 
	DPD-SCAD  $\alpha = $ 0.6 & -0.00 & 0.12 & 0.09 & 1.67 & 0.00 \\ 
	DPD-SCAD $\alpha = $ 0.9 & -0.00 & 0.20 & 0.34 & 3.43 & 0.00 \\ 
	\hline
	DPD-LASSO $\alpha = $ 0.3 & -0.07 & 0.54 & 0.56 & 3.61 & 0.02 \\ 
	DPD-LASSO  $\alpha = $ 0.6 & -0.10 & 0.64 & 0.74 & 3.53 & 0.01 \\ 
	DPD-LASSO  $\alpha = $ 0.9 & -0.02 & 0.58 & 0.83 & 2.61 & 0.00 \\ 
	\hline
	LDPD-LASSO  $\alpha = $ 0.1 & -0.12 & 0.46 & 0.66 & 3.14 & 0.00 \\ 
	LDPD-LASSO  $\alpha = $  0.2 & -0.07 & 0.33 & 0.36 & 3.50 & 0.01 \\ 
	LDPD-LASSO $\alpha = $ 0.3 & -0.07 & 0.34 & 0.37 & 3.47 & 0.01 \\ 
	LDPD-LASSO  $\alpha = $ 0.5 & -0.09 & 0.64 & 0.74 & 3.56 & 0.00 \\
	\hline 
\end{tabular}
	\label{resultsX1}
\end{table}

\begin{table}[H]
	\centering
	\caption{Error measures for Group 2 (test)  dataset}
	\label{resultsX2}
	\begin{tabular}{r|rrrrr}
	 \hline
	& BIAS & ABS & MSPE & MAX & MIN \\ 
	\hline
	LS-LASSO & -0.00 & 0.68 & 0.78 & 3.37 & 0.02 \\ 
	LS-SCAD & -0.00 & 0.67 & 0.77 & 3.33 & 0.01 \\ 
	RLARS & -0.11 & 1.03 & 1.78 & 3.65 & 0.02 \\ 
	\hline
	MNPRPE-SCAD  $\alpha = $ 0.1 & -0.00 & 1.02 & 1.62 & 2.88 & 0.01 \\ 
	MNPRPE-SCAD  $\alpha = $ 0.2 & -0.00 & 0.96 & 1.43 & 3.21 & 0.02 \\ 
	MNPRPE-SCAD  $\alpha = $ 0.3 & -0.00 & 1.05 & 1.85 & 4.53 & 0.01 \\ 
	MNPRPE-SCAD  $\alpha = $ 0.5 & -0.00 & 0.97 & 1.46 & 3.52 & 0.03 \\ 
	\hline
	DPD-SCAD  $\alpha = $ 0.3 & -0.00 & 0.86 & 1.12 & 3.04 & 0.03 \\ 
	DPD-SCAD  $\alpha = $ 0.6 & -0.00 & 1.09 & 1.85 & 3.40 & 0.02 \\ 
	DPD-SCAD   $\alpha = $ 0.9 & -0.00 & 0.98 & 1.45 & 3.19 & 0.11 \\
	\hline
	DPD-LASSO $\alpha = $ 0.3 & -0.07 & 0.75 & 0.94 & 3.35 & 0.04 \\ 
	DPD-LASSO  $\alpha = $ 0.6 & -0.10 & 0.68 & 0.84 & 3.54 & 0.01 \\ 
	DPD-LASSO $\alpha = $ 0.9 & -0.02 & 0.86 & 1.20 & 3.22 & 0.01 \\ 
	\hline 
	LDPD-LASSO $\alpha = $ 0.1 & -0.12 & 0.93 & 1.51 & 4.07 & 0.01 \\ 
	LDPD-LASSO $\alpha = $ 0.2 & -0.07 & 0.86 & 1.17 & 3.32 & 0.02 \\ 
	LDPD-LASSO $\alpha = $ 0.3 & -0.07 & 0.83 & 1.11 & 3.25 & 0.02 \\ 
	LDPD-LASSO $\alpha = $ 0.5 & -0.09 & 0.69 & 0.85 & 3.53 & 0.01 \\
\hline
\end{tabular}
\end{table}

Finally, Rajaratnam et al. (2019) showed that observations 27 and 29 were outliers; patient 29 has the smallest survival time of 7 days, with the next smallest value being 43 days, and observation 27 was the observation with the single largest (in magnitude) covariate value. We could analyze the robustness of our method in high dimensional setting by fitting the model after removing these observations and compare these new results with the previous ones obtained from the full data. Table \ref{comparation} contains the error measures as employed before, but now for difference between the predictions obtained from the model fitted with the (full) contaminated and the clean data for each method; the lower the values of these error measures, greater the stability is for the corresponding method. The difference on estimation when deleting outlier observation is lower for the MNPRPE than for any other method, illustrating its robustness.

\begin{table}[H]
	\caption{Error measures for the difference between predictions under contaminated and clean data.}
	\centering
\begin{tabular}{r|rrrrr}
	\hline
	& BIAS & ABS & MSPE & MAX & MIN \\ 
	\hline
	LS-LASSO & -0.08 & 0.08 & 0.01 & 0.08 & 0.08 \\ 
	LS-SCAD & -0.08 & 0.08 & 0.01 & 0.23 & 0.00 \\ 
	RLARS & 0.04 & 0.22 & 0.10 & 1.11 & 0.01 \\ 
	\hline
	MNPRPE-SCAD  $\alpha = $ 0.1  & 0.00 & 0.04 & 0.00 & 0.13 & 0.00 \\
	MNPRPE-SCAD  $\alpha = $ 0.2 & -0.01 & 0.16 & 0.04 & 0.45 & 0.00 \\ 
	MNPRPE-SCAD  $\alpha = $ 0.3 & -0.01  & 0.23 & 0.08 & 0.67 & 0.02 \\ 
	MNPRPE-SCAD  $\alpha = $ 0.5  & -0.01 & 0.23 & 0.09 & 0.77 & 0.01 \\
	\hline
	DPD-SCAD  $\alpha = $ 0.3 & 0.05 & 0.23 & 0.11 & 1.21 & 0.01 \\ 
	DPD-SCAD  $\alpha = $ 0.6 & 0.06 & 0.17 & 0.05 & 0.63 & 0.00 \\ 
	DPD-SCAD  $\alpha = $ 0.9 & 0.03 & 0.21 & 0.11 & 1.18 & 0.00 \\  
	\hline
	DPD-LASSO  $\alpha = $ 0.3 & 0.07 & 0.33 & 0.15 & 0.79 & 0.00 \\ 
	DPD-LASSO  $\alpha = $ 0.6 & 0.10 & 0.20 & 0.06 & 0.51 & 0.00 \\ 
	DPD-LASSO  $\alpha = $ 0.9 & 0.02 & 0.48 & 0.36 & 1.78 & 0.05 \\ 
	\hline
	LDPD-LASSO  $\alpha = $ 0.1 & -0.01 & 0.18 & 0.10 & 1.08 & 0.00 \\ 
	LDPD-LASSO  $\alpha = $ 0.2 & -0.03 & 0.08 & 0.01 & 0.44 & 0.00 \\ 
	LDPD-LASSO  $\alpha = $ 0.3 & -0.01 & 0.06 & 0.01 & 0.24 & 0.01 \\ 
	LDPD-LASSO  $\alpha = $ 0.5 & 0.00 & 0.12 & 0.02 & 0.38 & 0.00 \\ 
	\hline
\end{tabular}
\label{comparation}
\end{table}

\section{Conclusions}
In this paper we have presented a robust estimating method for the LRM in ultra-high dimensional settings. As we have shown, the MNPRPE boasts oracle properties and it is asymptotically normal distributed. Moreover, we have proposed a computational algorithm, merging two efficient minimization techniques, MM-algorithm and coordinate descent algorithm. Our results show that MNPRPE performs better than other common methods existing in the literature and estimate the error deviation $\sigma$ more precisely the other nonconcave penalized methods.

The proposed method is based on the combination of a robust loss function and nonconcave penalties. This idea could be extended to other loss and penalty functions to obtain new estimators with similar convenient properties. Further, akin methods could be developed in particular for binary logistic regression, multiple logistic regression, Poisson regression, etc, and in general for generalized linear models. The theory could also be widen to generalized error distributions, i.e., considering a general distribution instead of normal errors, and specifically for heavy-tailed error distributions. Ensuing this objectives
we claim to extend the ideas presented in this paper to other methods existing in high-dimensional data, such as Adaptive LASSO, Relaxed LASSO or Group LASSO. The first goal is the adaptive LASSO procedure, considered by Zou (2006) using quadratic loss.

On the other hand, it is important to have measures controlling, in the problem of variable selection, a type I error (false positive selection), including $p-$values which are adjusted for large-scale multiple testing, or the construction of confidence intervals or regions. In this sense it would be interesting to enhace some robust Wald-type tests based on
MNPRPE for the LRM in ultra-high context, extending to this scenario the ideas considered in Castilla et al. (2020).\\

\textbf{Acknowledgments}: This research is supported by the Spanish Grants no. PGC2018-095 194-B-100 and  no. FPU16/03104. Additionally, the research of AG is also partially supported by the INSPIRE faculty research grant from Department of Science and Technology, Government of India.

\newpage
\appendix 
\section{Supplementary material for ``On regularization methods based on R\'enyi's pseudodistances for sparse high-dimensional linear regression models''}
\subsection{Computation of the matrix $\boldsymbol{J}_{\alpha}(G;\boldsymbol{\beta},\sigma)$}

In order to have the matrix $\boldsymbol{J}_{\alpha}(G;\boldsymbol{\beta
},\sigma)$ it is necessary to get
\[
\nabla\Psi_{\alpha}(\boldsymbol{\beta},\sigma)=\left(
\begin{matrix}
\frac{\partial\Psi_{1,\alpha}}{\partial\boldsymbol{\beta}} & \frac
{\partial\Psi_{1,\alpha}}{\partial\sigma}\\
\frac{\partial\Psi_{2,\alpha}}{\partial\boldsymbol{\beta}} & \frac
{\partial\Psi_{2,\alpha}}{\partial\sigma}%
\end{matrix}
\right)
\]

\begin{align*}
\frac{\partial\Psi_{1,\alpha}}{\partial\boldsymbol{\beta}}  &  =-\alpha
\sigma^{-\frac{2\alpha+1}{\alpha+1}}\exp\left(  \frac{-\alpha}{2}\left(
\frac{y-\boldsymbol{x}^{T}\boldsymbol{\beta}}{\sigma}\right)  ^{2}\right)
\left(  \frac{\alpha}{\sigma}\left(  \frac{y-\boldsymbol{x}^{T}%
	\boldsymbol{\beta}}{\sigma}\right)  ^{2}\boldsymbol{x}\cdot\boldsymbol{x}%
^{T}-\frac{1}{\sigma}\boldsymbol{x}\cdot\boldsymbol{x}^{T}\right) \\
\frac{\partial\Psi_{1,\alpha}}{\partial\sigma}  &  =-\alpha\sigma
^{-\frac{3\alpha+2}{\alpha+1}}\exp\left(  \frac{-\alpha}{2}\left(
\frac{y-\boldsymbol{x}^{T}\boldsymbol{\beta}}{\sigma}\right)  ^{2}\right)
\left(  -\frac{2\alpha+1}{\alpha+1}-\alpha\left(  \frac{y-\boldsymbol{x}%
	^{T}\boldsymbol{\beta}}{\sigma}\right)  ^{2}-1\right)  \left(  \frac
{y-\boldsymbol{x}^{T}\boldsymbol{\beta}}{\sigma}\right)  \boldsymbol{x}^{T}\\
\frac{\partial\Psi_{2,\alpha}}{\partial\boldsymbol{\beta}}  &  =-\alpha
\sigma^{-\frac{3\alpha+2}{\alpha+1}}\exp\left(  \frac{-\alpha}{2}\left(
\frac{y-\boldsymbol{x}^{T}\boldsymbol{\beta}}{\sigma}\right)  ^{2}\right)
\left(  \alpha\left(  \left(  \frac{y-\boldsymbol{x}^{T}\boldsymbol{\beta}%
}{\sigma}\right)  ^{2}-\frac{1}{\alpha+1}\right)  -2\right)  \left(
\frac{y-\boldsymbol{x}^{T}\boldsymbol{\beta}}{\sigma}\right)  \boldsymbol{x}\\
\frac{\partial\Psi_{2,\alpha}}{\partial\sigma}  &  =-\alpha\sigma
^{-\frac{3\alpha+2}{\alpha+1}}\exp\left(  \frac{-\alpha}{2}\left(
\frac{y-\boldsymbol{x}^{T}\boldsymbol{\beta}}{\sigma}\right)  ^{2}\right)
\left[  \alpha\left(  \frac{y-\boldsymbol{x}^{T}\boldsymbol{\beta}}{\sigma
}\right)  ^{4}-\frac{5\alpha+3}{\alpha+1}\left(  \frac{y-\boldsymbol{x}%
	^{T}\boldsymbol{\beta}}{\sigma}\right)  ^{2}+\frac{2\alpha+1}{(\alpha+1)^{2}%
}\right].
\end{align*}
Therefore,
\[
\nabla\Psi_{\alpha}(\boldsymbol{\beta},\sigma)=\left[
\begin{array}
[c]{l}%
\frac{\partial\Psi_{\alpha}(\boldsymbol{\beta},\sigma)}{\partial
	\boldsymbol{\beta}}\\
\frac{\partial\Psi_{\alpha}(\boldsymbol{\beta},\sigma)}{\partial
	\boldsymbol{\beta}}%
\end{array}
\right]  =-\alpha\sigma^{-\frac{2\alpha+1}{\alpha+1}-1}\left[
\begin{array}
[c]{ll}%
\left(  \frac{\partial\Psi_{\alpha}}{\partial\boldsymbol{\beta}}\right)
_{11}\boldsymbol{x}\boldsymbol{x}^{T} & \left(  \frac{\partial\Psi_{\alpha}%
}{\partial\boldsymbol{\beta}}\right)  _{12}\boldsymbol{x}^{T}\\
\left(  \frac{\partial\Psi_{\alpha}}{\partial\sigma}\right)  _{21}%
\boldsymbol{x} & \left(  \frac{\partial\Psi_{\alpha}}{\partial\sigma}\right)
_{22}%
\end{array}
\right].
\]
Now we are going to get the expectation of the random vector. We shall use
$\mathbb{E}_{Y,\boldsymbol{X}}=\mathbb{E}_{X}\left[  \mathbb{E}%
_{Y|\boldsymbol{X}}\right]  $. First we calculate the conditional expectations,
\begin{align*}
\mathbb{E}_{Y|\boldsymbol{X}}\left[  \left(  \frac{\partial\Psi_{\alpha}%
}{\partial\boldsymbol{\beta}}\right)  _{11}\right]   &  =\mathbb{E}%
_{Y|\boldsymbol{X}}\left[  \exp\left(  \frac{-\alpha}{2\sigma^{2}%
}(y-\boldsymbol{x}^{T}\boldsymbol{\beta})^{2}\right)  \left(  \alpha\left(
\frac{y-\boldsymbol{x}^{T}\boldsymbol{\beta}}{\sigma}\right)  ^{2}-1\right)
\right] \\
&  =\int\left(  \alpha\left(  \frac{y-\boldsymbol{x}^{T}\boldsymbol{\beta}%
}{\sigma}\right)  ^{2}-1\right)  \frac{1}{\sqrt{2\pi}\sigma}\exp\left(
-\frac{\alpha+1}{2\sigma^{2}}(y-\boldsymbol{x}^{T}\boldsymbol{\beta}%
)^{2}\right)  dy=\frac{-1}{(\alpha+1)^{\frac{3}{2}}},\\
\mathbb{E}_{Y|\boldsymbol{X}}\left[  \left(  \frac{\partial\Psi_{\alpha}%
}{\partial\boldsymbol{\beta}}\right)  _{12}\right]   &  =\mathbb{E}%
_{Y|\boldsymbol{X}}\left[  \exp\left(  \frac{-\alpha}{2\sigma^{2}%
}(y-\boldsymbol{x}^{T}\boldsymbol{\beta})^{2}\right)  \left(  \alpha\left(
\frac{y-\boldsymbol{x}^{T}\boldsymbol{\beta}}{\sigma}\right)  ^{3}-\left(
\frac{3\alpha+2}{\alpha+1}\right)  \left(  \frac{y-\boldsymbol{x}%
	^{T}\boldsymbol{\beta}}{\sigma}\right)  \right)  \right]  =0,
\end{align*}

\begin{align*}
\mathbb{E}_{Y|\boldsymbol{X}}\left[  \left(  \frac{\partial\Psi_{\alpha}%
}{\partial\sigma}\right)  _{21}\right]   &  =\mathbb{E}_{Y|\boldsymbol{X}%
}\left[  \exp\left(  \frac{-\alpha}{2\sigma^{2}}(y-\boldsymbol{x}%
^{T}\boldsymbol{\beta})^{2}\right)  \left(  -\frac{2\alpha+1}{\alpha
	+1}-1\right)  \left(  \frac{y-\boldsymbol{x}^{T}\boldsymbol{\beta}}{\sigma
}\right)  +\alpha\left(  \frac{y-\boldsymbol{x}^{T}\boldsymbol{\beta}}{\sigma
}\right)  ^{3}\right]  =0\\
\mathbb{E}_{Y|\boldsymbol{X}}\left[  \left(  \frac{\partial\Psi_{\alpha}%
}{\partial\sigma}\right)  _{22}\right]   &  =\mathbb{E}_{Y|\boldsymbol{X}%
}\left[  \exp\left(  \frac{-\alpha}{2\sigma^{2}}(y-\boldsymbol{x}%
^{T}\boldsymbol{\beta})^{2}\right)  \left(  \left(  \frac{y-\boldsymbol{x}%
	^{T}\boldsymbol{\beta}}{\sigma}\right)  ^{4}\alpha\right.  \right. \\
&  \hspace{2cm}\left.  -\left.  \left(  \frac{5\alpha+3}{\alpha+1}\right)
\left(  \frac{y-\boldsymbol{x}^{T}\boldsymbol{\beta}}{\sigma}\right)
^{2}+\frac{2\alpha+1}{(\alpha+1)^{2}}\right)  \right]  =\frac{-2}%
{(\alpha+1)^{\frac{5}{2}}}.
\end{align*}
Therefore we have
\[
\boldsymbol{J}_{\alpha}(G;\boldsymbol{\beta},\sigma)=\mathbb{E}%
_{Y,\boldsymbol{X}}\left[  \nabla\Psi_{\alpha}(\boldsymbol{\beta}%
,\sigma)\right]  =\mathbb{E}_{X}\left[  \mathbb{E}_{Y/\boldsymbol{X}}\left[
\nabla\Psi_{\alpha}(\boldsymbol{\beta},\sigma)\right]  \right]  =-\alpha
\sigma^{-\frac{2\alpha+1}{\alpha+1}-1}\left[
\begin{matrix}
\frac{-1}{(\alpha+1)^{\frac{3}{2}}}\mathbb{E}_{\boldsymbol{X}}[\boldsymbol{X}%
\boldsymbol{X}^{T}] & \boldsymbol{0}\\
\boldsymbol{0} & \frac{-2}{(\alpha+1)^{\frac{5}{2}}}%
\end{matrix}
\right]  .
\]

\subsection{Proof of the main results}
\subsubsection{Proof Theorem 4}
A infinitely approximation for the absolute value, $|s|,$ is $\left\{
\sqrt{s^{2}+1/m}\right\}  _{m\in\mathbb{N}}$and the penalty function
$p_{\lambda}(|s|)$ is the limit of the infinitely differentiable penalties
$\left\{  p_{m,\lambda}(s)\right\}  _{m\in\mathbb{N}}$ with $p_{m,\lambda
}(s)=p_{\lambda}\left(  \sqrt{s^{2}+\frac{1}{m}}\right)  .$ The first and
second order derivatives of\ $p_{m,\lambda}(s)$ are given by
\[
\frac{\partial p_{\lambda}}{\partial s}\left(  \sqrt{s^{2}+\frac{1}{m}%
}\right)  \frac{s}{\sqrt{s^{2}+\frac{1}{m}}}\text{ and }\frac{\partial
	^{2}p_{\lambda}}{\partial s^{2}}\left(  \sqrt{s^{2}+\frac{1}{m}}\right) \cdot
\left(  \frac{s}{\sqrt{s^{2}+\frac{1}{m}}}\right)  ^{2}+\frac{\partial
	p_{\lambda}}{\partial s}\left(  \sqrt{s^{2}+\frac{1}{m}}\right)  \frac
{1}{m\left(s^{2}+\frac{1}{m}\right)^{3/2}},
\]
respectively. Avella-Medina (2017) established that the IF
corresponding to the penalty $p_{\lambda}(|s|)$ can be obtained as the limit
of the IF associated to the penalties $\left\{  p_{m,\lambda
}(s)\right\}  _{m\in\mathbb{N}}.$ \ These penalty functions are twice
diffetrentiables and therefore the corresponding IF can be obtained by Theorem
\ref{thm421}. Denoting $(\boldsymbol{\beta}_{m},\sigma_{m})=\boldsymbol{T}%
_{\alpha}^{m}(F_{\boldsymbol{\beta}_{m},\sigma_{m}})$,
\[
\operatorname{IF}\left(  (y_{t},\boldsymbol{x}_{t}),\boldsymbol{T}_{\alpha
}^{m},F_{\boldsymbol{\beta}_{0},\sigma_{0}}\right)  =-\boldsymbol{J}_{\alpha
}^{\ast}\left(  F_{\boldsymbol{\beta}_{m},\sigma_{m}}%
,\widehat{\boldsymbol{\beta}}_{m}^{\alpha},(\boldsymbol{\beta}_{m},\sigma
_{m})\right)  ^{-1}\left(
\begin{array}
[c]{c}%
-\alpha(\sigma_{m})^{-\frac{2\alpha+1}{\alpha+1}}\phi_{1,\alpha}\left(
\frac{y-\boldsymbol{x}^{T}\boldsymbol{\beta}_{m}}{\widehat{\sigma}_{m}%
}\right)  \boldsymbol{x}+\boldsymbol{\tilde{p}}_{\lambda}^{\ast}%
(\boldsymbol{\beta}_{m})\\
-\alpha(\sigma_{m})^{-\frac{2\alpha+1}{\alpha+1}}\phi_{2,\alpha}\left(
\frac{y-\boldsymbol{x}^{T}\boldsymbol{\beta}_{m}}{\sigma_{m}}\right)
\end{array}
\right)  .
\]
When $m\rightarrow\infty$, we have
\[
\frac{\partial p_{m,\lambda}}{\partial s}\rightarrow\frac{\partial p_{\lambda
}}{\partial s}(|s|)\cdot\operatorname{sgn}(s)\text{ and }\frac{\partial
	^{2}p_{m,\lambda}}{\partial s^{2}}\rightarrow\frac{\partial^{2}p_{\lambda}%
}{\partial s^{2}}(|s|)
\]
where $\operatorname{sgn}(\cdot)$ denotes the sign function and $(\boldsymbol{\beta}_{m},\sigma_{m})\rightarrow(\boldsymbol{\beta}_{\ast
},\sigma_{\ast})=\boldsymbol{T}_{\alpha}(G).$

\subsubsection{Proof Theorem 6}

Necessary condition: The classical optimization theory establishes that if
$\widehat{\boldsymbol{\theta}}^{T}=(\widehat{\boldsymbol{\beta}}%
,\widehat{\sigma})$ is a local minimizer of the objective function
$Q_{n}^{\alpha}(\boldsymbol{\theta})$, then it verifies the Karush-Kuhn-Tucker
(KKT) conditions, i.e., there exists some $\boldsymbol{v}=(v_{1}%
,..,v_{p+1})\in\mathbb{R}^{p+1}$ such that
\begin{equation}%
{\textstyle\sum_{i=1}^{n}}
\boldsymbol{\Psi}_{\alpha}\left(  (y_{i},\boldsymbol{x}_{i}%
),\widehat{\boldsymbol{\theta}}\right)  +\boldsymbol{v}=\boldsymbol{0}_{p+1}
\label{A.2.1}%
\end{equation}
where $v_{p+1}=0,$ $v_{j}=p_{\lambda}^{^{\prime}}(|\widehat{\beta}%
_{j}|)\operatorname{sg}(\widehat{\beta}_{j})$ if $\widehat{\beta}_{j}\neq0$
and $v_{j}\in\lbrack-p_{\lambda}^{^{\prime}}(0+),p_{\lambda}^{^{\prime}}(0+)]$
if $\widehat{\beta}_{j}=0$, and $\boldsymbol{\Psi}_{\alpha}\left(
(y_{i},\boldsymbol{x}_{i}),\boldsymbol{\theta}\right)  $ was defined in Equation
(9) of the main paper. Therefore we have
\[
\nabla Q_{n}^{\alpha}(\widehat{\boldsymbol{\theta}})=-\alpha\widehat{\sigma
}^{-\frac{2\alpha+1}{\alpha+1}}\frac{1}{n}%
{\textstyle\sum_{i=1}^{n}}
\left(
\begin{array}
[c]{l}%
\phi_{1,\alpha}\left(  \widehat{\boldsymbol{\beta}}\right)  \boldsymbol{x}\\
\phi_{2,\alpha}\left(  \frac{y_{i}-\boldsymbol{x}_{i}^{T}%
	\widehat{\boldsymbol{\beta}}}{\widehat{\sigma}}\right)
\end{array}
\right)  +\left(
\begin{array}
[c]{l}%
\tilde{\boldsymbol{p}}_{\lambda}(\widehat{\boldsymbol{\beta}})^{T}%
\vspace{0.3cm}\\
\hspace{0.4cm}0
\end{array}
\right)  ,
\]
It is clear that Equations (18) and (20) of the statement are verified. On the
other hand, $$\left\Vert \alpha\left(  \widehat{\sigma}\right)  ^{-\frac
	{2\alpha+1}{\alpha+1}}\sum_{i=1}^{n}\boldsymbol{\phi}_{1,\alpha}%
(r_{i}(\widehat{\boldsymbol{\theta}}))\boldsymbol{x}_{2i}\right\Vert _{\infty
}<p_{\lambda}^{^{\prime}}(0+)=\lambda\rho(p_{\lambda})$$ \ and Equation
(19) of the statement is also verified.

The MNPRPE, $\widehat{\boldsymbol{\theta}}^{T}=(\widehat{\boldsymbol{\beta}%
},\widehat{\sigma}),$ is also a local minimizer of $Q_{n}^{\alpha
}(\boldsymbol{\theta})$ on the constrained subspace $\mathcal{B}%
=\{(\boldsymbol{\beta},\sigma):\beta_{j}=0\hspace{0.2cm}\forall j>s\}\subset$
$\mathbb{R}^{s}\times\mathbb{R}^{+}$and it follows from the second order
condition that $\mathbb{X}_{\mathcal{S}}^{\ast T}\boldsymbol{\Sigma}_{\alpha
}(\widehat{\boldsymbol{\theta}})\mathbb{X}_{\mathcal{S}}^{\ast}%
-\operatorname{diag}\left(  p_{\lambda}^{\prime\prime}(|\widehat{\beta}%
_{1}|),..,p_{\lambda}^{\prime\prime}(|\widehat{\beta}_{p}|)\right)  $ is
positive definite. Therefore $\Lambda_{\min}\left(  \mathbb{X}%
_{\mathcal{S}}^{\ast T}\boldsymbol{\Sigma}_{\alpha}%
(\widehat{\boldsymbol{\theta}})\mathbb{X}_{\mathcal{S}}^{\ast}\right)
\geq\max_{1\leq j\leq p}(-p_{\lambda}^{\prime\prime}(|\widehat{\beta}%
_{j}|))=\xi(p_{\lambda},\widehat{\boldsymbol{\beta}}_{1})$ and Equation
(21) of the statement is verified.

Sufficient condition: We shall assume that conditions (18)-(21) of the main paper
are verified. We first constrain $Q_{n}^{\alpha}(\boldsymbol{\theta})$ on the
subspace $\mathcal{B\subset}\mathbb{R}^{s}\times\mathbb{R}^{+}.$ Assumption
(21) of the statement establishes that $Q_{n}^{\alpha}(\boldsymbol{\theta})$ is strictly concave in a neighborhood $\mathcal{N}_{0}\subset\mathcal{B}$ $\mathcal{\ }$ centered at $\widehat{\boldsymbol{\theta}}.$ This fact, jointly with
(18) and (20) of the statement, establish that $\widehat{\boldsymbol{\theta}},$ as
a critical point of $Q_{n}^{\alpha}(\boldsymbol{\theta})$ in $\mathcal{B},$ is
the unique minimizer of $Q_{n}^{\alpha}(\boldsymbol{\theta})$ in the ball
$\mathcal{N}_{0}.$

Now it is necessary to prove that $\widehat{\boldsymbol{\theta}}%
^{T}=(\widehat{\boldsymbol{\beta}},\widehat{\sigma})$ is indeed a strict local
minimizer of $Q_{n}^{\alpha}(\boldsymbol{\theta})$ on $\mathbb{R}^{p}%
\times\mathbb{R}^{+}$. We consider a sufficiently small ball $\mathcal{N}%
_{1}\subset\mathbb{R}^{p}\times\mathbb{R}^{+}$ centered at
$\widehat{\boldsymbol{\theta}}$ such that $\mathcal{B}\cap\mathcal{N}%
_{1}\subset\mathcal{N}_{0}.$ Let $\boldsymbol{\gamma}_{2}$ be the projection
of $\boldsymbol{\gamma}_{1}$ onto $\mathcal{B}.$ Then $\boldsymbol{\gamma}%
_{2}\in\mathcal{N}_{0}$ and $Q_{n}^{\alpha}(\widehat{\boldsymbol{\theta}})$
$<Q_{n}^{\alpha}(\boldsymbol{\gamma}_{2})$ if $\boldsymbol{\gamma}_{2}%
\neq\widehat{\boldsymbol{\theta}},$ since $\widehat{\boldsymbol{\theta}}$ is
the strict minimizer of $Q_{n}^{\alpha}(\boldsymbol{\theta})$ in
$\mathcal{N}_{0},$ and it will be enough to prove that $Q_{n}^{\alpha
}(\boldsymbol{\gamma}_{2})<Q_{n}^{\alpha}(\boldsymbol{\gamma}_{1})$ for any
$\boldsymbol{\gamma}_{1}\in\mathcal{N}_{1}\setminus\mathcal{N}_{0}$. On the basis of
the mean-value theorem,
\begin{equation}
Q_{n}^{\alpha}(\boldsymbol{\gamma}_{2})-Q_{n}^{\alpha}(\boldsymbol{\gamma}%
_{1})=\nabla Q_{n}^{\alpha}(\boldsymbol{\gamma}_{0})(\boldsymbol{\gamma}%
_{2}-\boldsymbol{\gamma}_{1}), \tag{A.2.2}%
\end{equation}
where $\boldsymbol{\gamma}_{0}$ lies on the line segment jointly
$\boldsymbol{\gamma}_{2}$ and $\boldsymbol{\gamma}_{1}.$ The components of the
vector $\boldsymbol{\gamma}_{1}-\boldsymbol{\gamma}_{2}$ coincide in
$\mathcal{B}\cap\mathcal{N}_{1}$ because $\boldsymbol{\gamma}_{2}$ is the
projection of $\boldsymbol{\gamma}_{1}$ onto $\mathcal{B}$, and
$\gamma_{2j}=0$ for $s<j<p+1$ because it belongs to $\mathcal{B}.$ 
Moreover, $\operatorname{sg}(\gamma_{0,j})$ $=\operatorname{sg}(\gamma_{1,j})$ if
$s<j<p+1.$ Therefore, we have $$Q_{n}^{\alpha}(\boldsymbol{\gamma}_{2}%
)-Q_{n}^{\alpha}(\boldsymbol{\gamma}_{1})=-\alpha\gamma_{0,p+1}^{-\frac
	{2\alpha+1}{\alpha+1}}\sum_{i=1}^{n}\phi_{1,\alpha}\left(  r_{i}%
(\boldsymbol{\gamma}_{0})\right)  \boldsymbol{x}_{2i}^{T}\boldsymbol{\gamma
}_{12}-\sum_{j=s+1}^{p}p_{\lambda}^{\prime}(|\gamma_{0,j}|)|\gamma_{1,j}|,$$
\ where $\boldsymbol{\gamma}_{12}$ are the non null components of
$(\boldsymbol{\gamma}_{1}-\boldsymbol{\gamma}_{2}).$ By $\boldsymbol{\gamma
}_{1}\in\mathcal{N}_{1}-\mathcal{N}_{0}$ we have $\boldsymbol{\gamma}_{12}%
\neq0.$

From concavity of $p_{\lambda}(s)$ , applying Condition (C1) of the main paper, we have that
$p_{\lambda}^{\prime}(s)$ is decreasing in $s\in\left[  0,\infty\right)  $.
Therefore by Assumption (19) of the statement of the Theorem and continuity of $p_{\lambda}^{\prime
}(s)$ , there exist $\delta>0$ such that $\forall\boldsymbol{\theta}%
\in\mathbb{B}(\widehat{\boldsymbol{\theta}},\delta)$ with $\mathbb{B}%
(\widehat{\boldsymbol{\theta}},\delta)=\{\boldsymbol{\theta}%
:||\boldsymbol{\theta}-\widehat{\boldsymbol{\theta}}||<\delta\}$ verifies
$\left\Vert \alpha\sigma^{-\frac{2\alpha+1}{\alpha+1}}\sum_{i=1}%
^{n}\boldsymbol{\phi}_{1,\alpha}(r_{i}(\boldsymbol{\theta}))\boldsymbol{x}%
_{2i}\right\Vert _{\infty}<p_{\lambda}^{\prime}(\delta).$

Reducing the ball if it is necessary, we assume that $\mathcal{N}_{1}%
\subset\mathbb{B}(\widehat{\boldsymbol{\theta}},\delta),$ and therefore
$|\gamma_{0,j}|<\delta, s<j<p+1$. Now, taking into account that $p_{\lambda
}^{\prime}$ is decreasing, we have $Q_{n}^{\alpha}(\boldsymbol{\gamma}%
_{2})-Q_{n}^{\alpha}(\boldsymbol{\gamma}_{1})<p_{\lambda}^{\prime}%
(\delta)||\boldsymbol{\gamma}_{12}||_{1}-p_{\lambda}^{\prime}(\delta
)||\boldsymbol{\gamma}_{12}||_{1}=0.$ This complete the proof.

\subsubsection{Proof Proposition 7}

Let $Z_1,..,Z_n$ be independent bounded random variables with $Z_i \in [a,b]$ for all $i$, where $-\infty < a \leq b < \infty$, the Hoeﬀding’s inequality establishes
$$\mathbb{P}\left(|S_n - \mathbb{E}(S_n) | \geq \varepsilon \right) \leq 2\exp\left(\frac{-2\varepsilon^2}{n(b-a)^2}\right) \hspace{0.3cm} \forall \varepsilon \geq 0.$$
We define, $$Z_i = \alpha \sigma^{-\frac{2\alpha+1}{\alpha+1}} \phi_{1,\alpha}(r_i(\boldsymbol{\theta}_0))$$
where $\phi_{1,\alpha}(u) = u \exp\left(\frac{-\alpha}{2}u^2\right)$. It can be shown that the function $\phi_{1,\alpha}(u)$ is bounded, 
$$-\sqrt{\frac{1}{\alpha}} \exp\left(-0.5\right) \leq \phi_{1,\alpha}(u) \leq \sqrt{\frac{1}{\alpha}} \exp\left(-0.5\right)$$
and so are the variables $Z_i$. 

On the other hand,  $$\mathbb{E}(S_n) = \sum_{i=1}^{n} \mathbb{E}_{\boldsymbol{X}}  \left[  \mathbb{E}_{Y|\boldsymbol{X}}  \left[\alpha \sigma^{-\frac{2\alpha+1}{\alpha+1}} \phi_{1,\alpha}(r_i(\boldsymbol{\theta}_0)) \right] \right] = 0.$$
Now, for any $\boldsymbol{a}=(a_1,..,a_n) \in \mathbb{R}^n$, $(a_1Z_i,..,a_nZ_n)$ are $n$ independent bounded random variables.  Applying Hoeﬀding’s inequality, we have
$$\mathbb{P}\left(\bigg| \sum_{i=1}^{n} \alpha \sigma^{-\frac{2\alpha+1}{\alpha+1}} \phi_{1,\alpha}(r_i(\boldsymbol{\theta}_0)) a_i \bigg| \geq \varepsilon \right) \leq 2\exp\left(-2c_1\varepsilon^2\right) \hspace{0.3cm} \forall \varepsilon \geq 0$$
with
$c_1 = \frac{\alpha \exp(1)}{4n||\boldsymbol{a}||_2^2}$, or equivalently, using that $\left(\frac{a_1}{||\boldsymbol{a}||_2}Z_i,..,\frac{a_n}{||\boldsymbol{a}||_2}Z_n\right)$ have the same bounds,
$$\mathbb{P}\left(\bigg| \sum_{i=1}^{n} \alpha \sigma^{-\frac{2\alpha+1}{\alpha+1}} \phi_{1,\alpha}(r_i(\boldsymbol{\theta}_0)) a_i \bigg| \geq \varepsilon ||\boldsymbol{a}||_2  \right) \leq 2\exp\left(-2c_1\varepsilon^2\right) \hspace{0.3cm} \forall \varepsilon \geq 0$$
with
$c_1 = \frac{\alpha \exp(1)}{4n}$.

\subsubsection{Proof Theorem 8}

Let $\boldsymbol{\theta}_{0}^{T}=\left(  \boldsymbol{\beta}_{0},\sigma
_{0}\right)  $ the true value of the parameter and $\boldsymbol{\xi
}=(\boldsymbol{\xi}_{\mathcal{S}}^{T},\boldsymbol{\xi}_{\mathcal{N}}^{T}%
,\xi_{p+1})^{T}=\sum_{i=1}^{n}\boldsymbol{\Psi}_{\alpha}\left(  (y_{i}%
,\boldsymbol{x}_{i}),\boldsymbol{\theta}_{0}\right)  $, where $\boldsymbol{\xi
}_{\mathcal{S}}=\left(  \xi_{1},..,\xi_{s}\right)  ^{t}$ and $\boldsymbol{\xi
}_{\mathcal{N}}=\left(  \xi_{s+1},..,\xi_{p}\right)  ^{T}$ and we also
consider \ the events,%
\[
\zeta_{1}=\left\{  ||\boldsymbol{\xi}_{\mathcal{S}}||_{\infty}\leq\sqrt
{c_{1}^{-1}n\log n}\right\}  ;\text{ }\zeta_{2}=\left\{  ||\boldsymbol{\xi
}_{\mathcal{N}}||_{\infty}\leq u_{n}\sqrt{n}\right\}  \text{ and }\zeta
_{3}=\left\{  |\xi_{p+1}|\leq\sqrt{c_{1}^{-1}n\log n}\right\}  ,
\]
where $u_{n}=c_{1}^{-1/2}n^{1/2-\tau^{\ast}}(\log n)^{1/2}$ is a divergence
sequence, $\tau^{\ast}$ is considered in Assumption (A4) and $c_{1}$ in
Proposition 7 of the main paper, respectively. Applying Bonferroni`s inequality and
Proposition 7 of the main paper with $\boldsymbol{a=}\left(  1,...,1\right)  ^{T},$
we have
\begin{align*}
\Pr\left(  \zeta_{1}\cap\zeta_{2}\cap\zeta_{3}\right)   &  \geq1-\Pr\left(
\zeta_{1}^{C}\right)  -\Pr\left(  \zeta_{2}^{C}\right)  -\Pr\left(  \zeta
_{3}^{C}\right)\\  & \geq1-%
{\textstyle\sum_{j\in\mathcal{S}\cup{p+1}}}
\mathbb{P}\left(  |\xi_{j}|>\sqrt{c_{1}^{-1}n\log n}\right)  -%
{\textstyle\sum_{j\in\mathcal{S}^{c}}}
\mathbb{P}\left(  |\xi_{j}|>u_{n}\sqrt{n}\right) \\
&  =1-2\left[  (s+1)n^{-1}+(p-s)\exp(-c_{1}u_{n}^{2})\right]  =1-2\left[
(s+1)n^{-1}+(p-s)\exp(-n^{1-2\tau^{\ast}}\log n)\right]  .
\end{align*}
In our case $\varepsilon$ appearing in Proposition 7 is given by
$\varepsilon=u_{n}$ or $\sqrt{c_{1}^{-1}\log n}$ and it is necessary to see
that $0<\varepsilon<\frac{||\boldsymbol{a}||_{2}}{||\boldsymbol{a}||_{\infty}%
}=\sqrt{n}.$ It is clear that $\sqrt{c_{1}^{-1}\log n}<\sqrt{n}$ and
$u_{n}=c_{1}^{-1/2}n^{1/2-\tau^{\ast}}(\log n)^{1/2}=n^{1/2}\frac{1}%
{c_{1}^{1/2}}\frac{(\log n)^{1/2}}{n^{\tau^{\ast}}}<n^{1/2}\frac{1}%
{c_{1}^{1/2}}\frac{1}{\max_{1\leq j\leq p}||\boldsymbol{x}^{(j)}||_{\infty}%
}<n^{1/2}.$

Under the event $\boldsymbol{\zeta}=\zeta_{1}\cap\zeta_{2}\cap\zeta_{3}$ we
shall show that there exists a solution $\widehat{\boldsymbol{\theta}}%
^{T}=(\widehat{\boldsymbol{\beta}},\widehat{\sigma})$ to (18) and
(20) of the main paper. First we establish that for sufficiently large $n,$
(18) and (20) have a solution inside the hypercube in
$\mathbb{R}^{s}\times\mathbb{R}^{+}$
\[
\mathcal{N}=\left\{  (\boldsymbol{\delta},\sigma)\in\mathbb{R}^{s}%
\times\mathbb{R}^{+}:||\boldsymbol{\delta}-\boldsymbol{\beta}_{S_{0}%
}||_{\infty}=n^{-\tau}\log n,|\sigma-\sigma_{0}|=n^{-\tau}\log n\right\}  .
\]
Let $\boldsymbol{\delta=}\left(  \delta_{1},...,\delta_{s}\right)$ and
$\sigma$ $\in\mathcal{N}$. Since $n^{-\tau}\log n\leq d_{n}=\min
_{j\in\mathcal{S}}|\beta_{0j}|/2$,
\begin{equation}
\min_{1\leq j\leq s}|\delta_{j}|\geq\min_{j\in\mathcal{S}}|\beta_{0j}%
|-d_{n}=d_{n},\hspace{0.3cm}j=1,..,s, \label{A2.3}%
\end{equation}
$\operatorname{sg}(\delta_{j})=\operatorname{sg}(\beta_{j0}),$ $j=1,..,s,$ and
$\operatorname{sg}(\sigma)=\operatorname{sg}(\sigma_{0}).$ The last inequality
follows, by definition of $\mathcal{N\ }$, because $|\delta_{j}%
|\geq|\beta_{0j}|-n^{-\tau}\log n\geq|\beta_{0j}|-d_{n},\hspace{0.3cm}%
j=1,..,s.$

Let $\boldsymbol{\eta}=n\boldsymbol{\tilde{p}}_{\lambda}^{\ast}%
(\boldsymbol{\delta},\boldsymbol{0}_{p-s})$. Using that $p_{\lambda}^{\prime}$
is decreasing and inequality (\ref{A2.3}), we have $$||\boldsymbol{\eta
}||_{\infty}=np_{\lambda}^{\prime}(\min_{j=1,..s}|\delta_{j}|)\leq
np_{\lambda}^{\prime}(d_{n})$$ which jointly with the definition of $\zeta_{1}$
entails,
\begin{equation}
||\boldsymbol{\xi}_{\mathcal{S}}+\boldsymbol{\eta}||_{\infty}\leq\sqrt
{c_{1}^{-1}n\log n}+np_{\lambda}^{\prime}(d_{n}). \label{cotaep+n}%
\end{equation}
We define the two following functions for all $\boldsymbol{\delta}%
\in\mathbb{R}^{s}$ and $\sigma\in\mathbb{R}^{+},$ $\boldsymbol{\gamma
}(\boldsymbol{\delta},\sigma)=\left(  \gamma_{1}(\boldsymbol{\delta}%
,\sigma),..,\gamma_{p}(\boldsymbol{\delta},\sigma),\gamma_{p+1}%
(\boldsymbol{\delta},\sigma)\right)  ^{T}=\sum_{i=1}^{n}\boldsymbol{\Psi
}_{\alpha}\left(  (y_{i},\boldsymbol{x}_{i}),(\boldsymbol{\delta}_{\ast}%
)^{T}\right)  $ and $\boldsymbol{\Phi}(\boldsymbol{\delta},\sigma
)=\boldsymbol{\gamma}_{\mathcal{S}}^{\ast}(\boldsymbol{\delta},\sigma
)-\boldsymbol{\gamma}_{\mathcal{S}}^{\ast}(\boldsymbol{\beta}_{\mathcal{S}%
	0},\sigma_{0})+\boldsymbol{\xi}_{\mathcal{S}}^{\ast}+\boldsymbol{\eta}^{\ast
},$ where $\boldsymbol{a}_{\mathcal{S}}^{\ast}=(a_{1},..,a_{s},a_{p+1})^{T}$
for any $(p+1)-$dimensional vector and and $\boldsymbol{\eta}^{\ast
}=(\boldsymbol{\eta}^{T},0)^{T}.$ The Equations (18) and (20) of the main paper
are equivalent to $\boldsymbol{\Phi}(\boldsymbol{\delta},\sigma
)=\boldsymbol{0}_{s+1}$ and then we need to prove that it has a solution inside the
hypercube $\mathcal{N}.$

The function $\boldsymbol{\gamma}(\boldsymbol{\delta},\sigma)$ is twice
differentiable in $\mathcal{N}_{\text{ }}$and a second order Taylor expansion
gives
\begin{equation}
\boldsymbol{\gamma}_{\mathcal{S}}^{\ast}(\boldsymbol{\delta},\sigma
)=\boldsymbol{\gamma}_{\mathcal{S}}^{\ast}(\boldsymbol{\beta}_{\mathcal{S}%
	0},\sigma_{0})+\left(  \mathbb{X}_{\mathcal{S}}^{\ast T}\boldsymbol{\Sigma
}_{\alpha}(\boldsymbol{\beta}_{\mathcal{S}0},\sigma_{0})\mathbb{X}%
_{\mathcal{S}}^{\ast}\right)  \left[  (\boldsymbol{\delta},\sigma
)-(\boldsymbol{\beta}_{\mathcal{S}0},\sigma_{0})\right]  +\boldsymbol{r},
\label{L1}%
\end{equation}
with $\boldsymbol{r}=(r_{1},..,r_{s+1})^{T}$ and $r_{j}=\frac{1}{2}\left[
(\boldsymbol{\delta},\sigma)-(\boldsymbol{\beta}_{\mathcal{S}0},\sigma
_{0})\right]  ^{T}\nabla^{2}\gamma_{j}(\boldsymbol{\delta}^{\ast},\sigma
^{\ast})\left[  (\boldsymbol{\delta},\sigma)-(\boldsymbol{\beta}%
_{\mathcal{S}0},\sigma_{0})\right]  ,$ with $(\boldsymbol{\delta}^{\ast
},\sigma^{\ast})$ some vector lying on the line segment joining
$(\boldsymbol{\delta},\sigma)$ and $(\boldsymbol{\beta}_{\mathcal{S}0}%
,\sigma_{0}).$ We are going to get a bound for $||\boldsymbol{r}||_{\infty}$,
\begin{equation}
||\boldsymbol{r}||_{\infty}\leq\frac{1}{2}\left(  \max_{(\boldsymbol{\delta
	},\sigma)\in\mathcal{N}_{0}}\max_{1\leq j\leq p+1}\Lambda_{\max}(\nabla
^{2}\gamma_{j}(\boldsymbol{\delta}^{\ast},\sigma^{\ast}))\right)  \left(
(s+1)||(\boldsymbol{\delta},\sigma)-(\boldsymbol{\beta}_{\mathcal{S}0}%
,\sigma_{0})||_{2}^{2}\right).  \label{L2}%
\end{equation}
By Equation (24) in Assumption (A2) of the main paper, $\max_{(\boldsymbol{\delta},\sigma
	)\in\mathcal{N}_{0}}\max_{1\leq j\leq p+1}\left\{  \Lambda_{\max}\left(
\nabla^{2}\gamma_{j}(\boldsymbol{\delta},\sigma)\right)  \right\}  =O(n).$ At
the same time $||(\boldsymbol{\delta},\sigma)-(\boldsymbol{\beta}%
_{\mathcal{S}0},\sigma_{0})||_{2}^{2}=%
{\textstyle\sum_{j=1}^{s}}
\left(  \delta_{j}-\beta_{S0j}\right)  ^{2}+\left(  \sigma-\sigma_{0}\right)
^{2},$ but $||(\boldsymbol{\delta}-\boldsymbol{\beta}_{\mathcal{S}0}%
)||_{\infty}=\max_{j}\left\vert \delta_{j}-\beta_{S0j}\right\vert =n^{-\tau
}\log n$ and $\left(  \delta_{j}-\beta_{S0j}\right)  ^{2}=O(n^{-2\tau}(\log
n)^{2}).$ On the other hand $\left(  \sigma-\sigma_{0}\right)  ^{2}%
=O(n^{-2\tau}(\log n)^{2}).$ Finally,%
\begin{equation}
||\boldsymbol{r}||_{\infty}\leq O\left(  (s+1)n^{1-2\tau}(\log n)^{2}\right)
. \label{cotar}%
\end{equation}
Now, let $\boldsymbol{\Phi}^{\ast}(\boldsymbol{\delta},\sigma):=\left(
\mathbb{X}_{\mathcal{S}}^{\ast T}\boldsymbol{\Sigma}_{\alpha}%
(\boldsymbol{\theta}_{0})\mathbb{X}_{\mathcal{S}}^{\ast}\right)
^{-1}\boldsymbol{\Phi}(\boldsymbol{\delta},\sigma).$ Applying definition of
$\boldsymbol{\Phi}(\boldsymbol{\delta},\sigma)$ and (\ref{L1}) we have
\begin{align}
\boldsymbol{\Phi}^{\ast}(\boldsymbol{\delta},\sigma)  &  =\left(
\mathbb{X}_{\mathcal{S}}^{\ast T}\boldsymbol{\Sigma}_{\alpha}%
(\boldsymbol{\theta}_{0})\mathbb{X}_{\mathcal{S}}^{\ast}\right)  ^{-1}\left(
\boldsymbol{\gamma}_{\mathcal{S}}^{\ast}(\boldsymbol{\delta},\sigma
)-\boldsymbol{\gamma}_{\mathcal{S}}^{\ast}(\boldsymbol{\beta}_{\mathcal{S}%
	0},\sigma_{0})+\boldsymbol{\xi}_{\mathcal{S}}^{\ast}+\boldsymbol{\eta}^{\ast
}\right) \nonumber\\
&  =\left(  \mathbb{X}_{\mathcal{S}}^{\ast T}\boldsymbol{\Sigma}_{\alpha
}(\boldsymbol{\theta}_{0})\mathbb{X}_{\mathcal{S}}^{\ast}\right)  ^{-1}\left(
\left(  \mathbb{X}_{\mathcal{S}}^{\ast T}\boldsymbol{\Sigma}_{\alpha
}(\boldsymbol{\beta}_{\mathcal{S}0},\sigma_{0})\mathbb{X}_{\mathcal{S}}^{\ast
}\right)  \left[  (\boldsymbol{\delta},\sigma)-(\boldsymbol{\beta
}_{\mathcal{S}0},\sigma_{0})\right]  +\boldsymbol{r}+\boldsymbol{\xi
}_{\mathcal{S}}^{\ast}+\boldsymbol{\eta}^{\ast}\right) \nonumber\\
&  =\left[  (\boldsymbol{\delta},\sigma)-(\boldsymbol{\beta}_{\mathcal{S}%
	0},\sigma_{0})\right]  +\left(  \mathbb{X}_{\mathcal{S}}^{\ast T}%
\boldsymbol{\Sigma}_{\alpha}(\boldsymbol{\theta}_{0})\mathbb{X}_{\mathcal{S}%
}^{\ast}\right)  ^{-1}\left(  \boldsymbol{r}+\boldsymbol{\xi}_{\mathcal{S}%
}^{\ast}+\boldsymbol{\eta}^{\ast}\right)  =\left[  (\boldsymbol{\delta}%
,\sigma)-(\boldsymbol{\beta}_{0},\sigma_{0})\right]  +\boldsymbol{u,}
\label{phi}%
\end{align}

where $\boldsymbol{u}:=\left(  \mathbb{X}_{\mathcal{S}}^{\ast T}%
\boldsymbol{\Sigma}_{\alpha}(\boldsymbol{\theta}_{0})\mathbb{X}_{\mathcal{S}%
}^{\ast}\right)  ^{-1}\left[  \boldsymbol{\xi}_{\mathcal{S}}^{\ast
}+\boldsymbol{\eta}^{\ast}+\boldsymbol{r}\right]  $.

It follows from Assumption (A2) of the main paper, inequalities (\ref{cotaep+n}), (\ref{cotar})
and the condition on $b_{s}$ given in (9) of Assumption (A3) of the main paper that
\begin{equation}
\begin{aligned}
||\boldsymbol{u}||_{\infty}\leq &||\left(  \mathbb{X}_{\mathcal{S}}^{\ast
	T}\boldsymbol{\Sigma}_{\alpha}(\boldsymbol{\theta}_{0})\mathbb{X}%
_{\mathcal{S}}^{\ast}\right) ^{-1}||_{\infty}\left\{  |\boldsymbol{\xi
}_{\mathcal{S}}^{\ast}+\boldsymbol{\eta}^{\ast}||_{\infty}+||\boldsymbol{r}%
||_{\infty}\right\}\\
=&o\left(  b_{s}\sqrt{n^{-1}\log n}+b_{s}p_{\lambda
}^{\prime}(d_{n})+b_{s}sn^{-2\tau}(\log n)^{2}\right)  \\
=&o\left(  n^{-\tau}\log
n\right). \label{cotau}%
\end{aligned}
\end{equation}
Taking a vector $\boldsymbol{k}=\left[  (\boldsymbol{\delta},\sigma
)-(\boldsymbol{\beta}_{0},\sigma_{0})\right]  \in\mathbb{R}^{s+1}$ and
$(\boldsymbol{\delta},\sigma)\in\mathcal{N}$, we have by (\ref{cotau}), that for
all $j=1,..,s+1,$
\[%
\begin{split}
\boldsymbol{\Phi}^{\ast}(\boldsymbol{\delta},\sigma)_{j}  &  \geq n^{\tau
}\sqrt{\log n}-||\boldsymbol{u}||_{\infty}\geq0,\text{ if }k_{j}=n^{\tau}%
\sqrt{\log n}\\
\boldsymbol{\Phi}^{\ast}(\boldsymbol{\delta},\sigma)_{j}  &  \leq-n^{\tau
}\sqrt{\log n}+||\boldsymbol{u}||_{\infty}\leq0,\text{ if }k_{j}=-n^{\tau
}\sqrt{\log n}%
\end{split}
\]
for sufficiently large $n$. By the continuity of $\boldsymbol{\Phi}^{\ast
}(\boldsymbol{\delta},\sigma)$ and applying Miranda's existence Theorem, the
equation $\boldsymbol{\Phi}^{\ast}(\boldsymbol{\delta},\sigma)=\boldsymbol{0}%
_{s+1}^{T}$ has a solution, $(\widehat{\boldsymbol{\beta}}_{1},\boldsymbol{0}%
_{p-s},\widehat{\sigma})^{T},$ in the interior of $\mathcal{N}$ and therefore
$(\widehat{\boldsymbol{\beta}}_{1},\boldsymbol{0}_{p-s},\widehat{\sigma}%
)^{T}$ is a solution for $\boldsymbol{\Phi}(\boldsymbol{\delta}%
,\sigma)=\boldsymbol{0}_{s+1}$ too. Therefore, there exists
$(\widehat{\boldsymbol{\beta}}_{1},\boldsymbol{0}_{p-s},\widehat{\sigma})$
verifying (18) and (20) of Theorem 6 of the main paper.

Now we prove the verification of (19) and (21) of Theorem 6 of the main paper. Condition (19) is verified in
$\mathcal{N}_{0}$ by assumption (A3) of the main paper, therefore it is necessary to establish inequality
(21). Let
\[
\boldsymbol{z}:=\frac{1}{n\lambda}\alpha\sigma^{-\frac{2\alpha+1}{\alpha+1}%
}\sum_{i=1}^{n}\phi_{1,\alpha}\left(  r_{i}(\widehat{\boldsymbol{\theta}%
})\right)  \boldsymbol{x}_{2,i}=\frac{1}{n\lambda}\left[  \boldsymbol{\xi
}_{\mathcal{N}}+\boldsymbol{\gamma}_{\mathcal{N}}(\widehat{\boldsymbol{\beta}%
}_{\mathcal{S}},\widehat{\sigma})-\boldsymbol{\gamma}_{\mathcal{N}%
}(\boldsymbol{\beta}_{\mathcal{S}0},\sigma_{0})\right]  .
\]
On the event $\boldsymbol{\zeta}_{2}$ and by Assumption (A2) of the main paper, $\lambda
\geq(\log n)^{2}/n^{\tau^{\ast}}$. Thus, we have
\[
\left\Vert n^{-}\lambda^{-1}\boldsymbol{\xi}_{\mathcal{N}}\right\Vert
_{\infty}\leq o\left(  n^{-1/2}\lambda^{-1}u_{n}\right)  =o\left(
c_{1}^{-1/2}n^{1/2-\tau^{\ast}}(\log n)^{1/2}n^{-1/2}\lambda^{-1}\right)
=o\left(  (\log n)^{-3/2}\right)  \leq o(1).
\]
A second order Taylor expansion of $\boldsymbol{\gamma}_{\mathcal{N}}$ around
$(\boldsymbol{\beta}_{\mathcal{S}0},\sigma_{0})$, gives
\[
\boldsymbol{\gamma}_{\mathcal{N}}(\widehat{\boldsymbol{\beta}}_{\mathcal{S}%
},\widehat{\sigma})=\boldsymbol{\gamma}_{\mathcal{N}}(\boldsymbol{\beta
}_{\mathcal{S}0},\sigma_{0})+\left(  \mathbb{X}_{\mathcal{N}}^{\ast
	T}\boldsymbol{\Sigma}_{\alpha}(\boldsymbol{\beta}_{0},\sigma_{0}%
)\mathbb{X}_{\mathcal{S}}^{\ast}\right)  \left[  (\widehat{\boldsymbol{\beta}%
}_{\mathcal{S}},\widehat{\sigma})-(\boldsymbol{\beta}_{\mathcal{S}0}%
,\sigma_{0})\right]  +\boldsymbol{\omega},
\]
with $\boldsymbol{\omega}=(\omega_{s+1},..,\omega_{p})^{T}$ and $\omega
_{j}=\frac{1}{2}\left[  (\widehat{\boldsymbol{\beta}}_{\mathcal{S}%
},\widehat{\sigma})-(\boldsymbol{\beta}_{\mathcal{S}0},\sigma_{0})\right]
^{T}\nabla^{2}\gamma_{j}(\boldsymbol{\delta}^{\ast\ast},\sigma^{\ast\ast
})\left[  (\widehat{\boldsymbol{\beta}}_{\mathcal{S}},\widehat{\sigma
})-(\boldsymbol{\beta}_{\mathcal{S}0},\sigma_{0})\right]  ,$ being
$(\boldsymbol{\delta}^{\ast\ast},\sigma^{\ast\ast})$ some vector lying on the
line segment connecting $(\widehat{\boldsymbol{\beta}}_{\mathcal{S}%
},\widehat{\sigma})$ and $(\boldsymbol{\beta}_{\mathcal{S}0},\sigma_{0}).$ By
Equation (24) in Assumption (A2) of the main paper and taking into account that
$(\widehat{\boldsymbol{\beta}}_{\mathcal{S}},\widehat{\sigma})\in\mathcal{N}$,
we could argue similarly to (\ref{L2}) to obtain $||\boldsymbol{\omega}||_{\infty}\leq
O\left(  sn^{1-2\tau}(\log n)^{2}\right)  .$ Since,
$(\widehat{\boldsymbol{\beta}}_{\mathcal{S}},\widehat{\sigma})$ satisfies the
equation $\boldsymbol{\Phi}^{\ast}\left(  \boldsymbol{\delta},\sigma\right)
=\boldsymbol{0}_{s+1}$, we have $(\widehat{\boldsymbol{\beta}}_{\mathcal{S}%
},\widehat{\sigma})-(\boldsymbol{\beta}_{0},\sigma_{0})=-\left(
\mathbb{X}_{\mathcal{S}}^{\ast T}\boldsymbol{\Sigma}_{\alpha}%
(\boldsymbol{\theta}_{0})\mathbb{X}_{\mathcal{S}}^{\ast}\right)  ^{-1}\left(
\boldsymbol{\xi}_{\mathcal{S}}^{\ast}+\boldsymbol{\eta}^{\ast}+\boldsymbol{r}%
\right)  $ and it is possible to get a bound for the norm of $\boldsymbol{z}$ by%
\begin{align*}
\left\Vert \boldsymbol{z}\right\Vert _{\infty}  &  \leq o(1)+\frac{1}%
{n\lambda}||\boldsymbol{\gamma}_{\mathcal{N}}(\widehat{\boldsymbol{\beta}%
}_{\mathcal{S}},\widehat{\sigma})-\boldsymbol{\gamma}_{\mathcal{N}%
}(\boldsymbol{\beta}_{\mathcal{S}0},\sigma_{0})||_{\infty}\\
&  \leq o(1)+\frac{1}{n\lambda}\left\Vert \left(  \mathbb{X}_{\mathcal{N}%
}^{\ast T}\boldsymbol{\Sigma}_{\alpha}(\boldsymbol{\beta}_{0},\sigma
_{0})\mathbb{X}_{\mathcal{S}}^{\ast}\right)  \left(  \mathbb{X}_{\mathcal{S}%
}^{\ast T}\boldsymbol{\Sigma}_{\alpha}(\boldsymbol{\beta}_{0},\sigma
_{0})\mathbb{X}_{\mathcal{S}}^{\ast}\right)  ^{-1}\right\Vert _{\infty}\left(
||\boldsymbol{\xi}_{\mathcal{S}}^{\ast}+\boldsymbol{\eta}^{\ast}||_{\infty
}+||\boldsymbol{r}||_{\infty}\right)  +\frac{1}{n\lambda}||\boldsymbol{\omega
}||_{\infty}\\
&  \leq o(1)+\frac{1}{n\lambda}\left\Vert \left(  \mathbb{X}_{\mathcal{N}%
}^{\ast T}\boldsymbol{\Sigma}_{\alpha}(\boldsymbol{\beta}_{0},\sigma
_{0})\mathbb{X}_{\mathcal{S}}^{\ast}\right)  \left(  \mathbb{X}_{\mathcal{S}%
}^{\ast T}\boldsymbol{\Sigma}_{\alpha}(\boldsymbol{\beta}_{0},\sigma
_{0})\mathbb{X}_{\mathcal{S}}^{\ast}\right)  ^{-1}\right\Vert _{\infty
}O\left(  \sqrt{c_{1}^{-1}n\log n}+np_{\lambda}^{\prime}(d_{n})+sn^{1-2\tau
}(\log n)^{2}\right) \\
&  +\frac{1}{n\lambda}O(sn^{1-2\tau}(\log n)^{2})\\
&  \leq o(1)+\frac{1}{n\lambda}\left\Vert \left(  \mathbb{X}_{\mathcal{N}%
}^{\ast T}\boldsymbol{\Sigma}_{\alpha}(\boldsymbol{\beta}_{0},\sigma
_{0})\mathbb{X}_{\mathcal{S}}^{\ast}\right)  \left(  \mathbb{X}_{\mathcal{S}%
}^{\ast T}\boldsymbol{\Sigma}_{\alpha}(\boldsymbol{\beta}_{0},\sigma
_{0})\mathbb{X}_{\mathcal{S}}^{\ast}\right)  ^{-1}\right\Vert _{\infty
}np_{\lambda}^{\prime}(d_{n})\\
&  +\frac{1}{n\lambda}O\left(  n^{\tau_{1}}\sqrt{n\log n}+sn^{1-2\tau+\tau
	_{1}}(\log n)^{2}\right)  +\frac{1}{n\lambda}O\left(  sn^{1-2\tau}(\log
n)^{2}\right) \\
&  \leq o(1)+\frac{1}{\lambda}\frac{Cp_{\lambda}^{\prime}(0+)}{p_{\lambda
	}^{\prime}(d_{n})}p_{\lambda}^{\prime}(d_{n})+\frac{1}{n\lambda}O\left(
n^{\tau_{1}}\sqrt{n\log n}+sn^{1-2\tau+\tau_{1}}(\log n)^{2}+sn^{1-2\tau}(\log
n)^{2}\right) \\
&  \leq o(1)+\rho(p_{\lambda})\leq\rho(p_{\lambda})
\end{align*}
for sufficiently large $n.$ Therefore we have condition (19) of the Theorem 6 of the main paper and
$(\widehat{\boldsymbol{\beta}}_{\mathcal{S}},\widehat{\sigma})$ is a strict
minimizer of $Q_{n}^{\alpha}(\boldsymbol{\theta})$ on $\boldsymbol{\zeta}$
with probability at least $1-2\left[  (s+1)n^{-1}+(p-s)\exp(-n^{1-2\tau^{\ast
}}\log n)\right]  ,$ with the last $p-s$ components of
$\widehat{\boldsymbol{\beta}}_{\mathcal{S}}$ non null and
$(\widehat{\boldsymbol{\beta}}_{\mathcal{S}},\widehat{\sigma})$ is in the
interior of $\mathcal{N}.$

\subsubsection{Proof Theorem 9}

First we study the consistency in the $\left(  s+1\right)  $-dimensional
subspace $\mathcal{B}=\{(\boldsymbol{\beta},\sigma)\in\mathbb{R}^{p}%
\times\mathbb{R}^{+}:\boldsymbol{\beta}_{\mathcal{N}}=\boldsymbol{0}\}$. The
first step will be to see that $Q_{n}^{\alpha}(\boldsymbol{\theta})$
constrained to $\mathcal{B}$ has a strict local minimizer. The constrained
objective function is given by
\[
Q_{n,\mathcal{B}}^{\alpha}(\boldsymbol{\delta},\sigma)=L_{n,\mathcal{B}%
}^{\alpha}(\boldsymbol{\delta},\sigma)+%
{\textstyle\sum_{j=1}^{s}}
p_{\lambda}(|\delta_{j}|),
\]
with $\boldsymbol{\delta}=(\delta_{1},..,\delta_{s})^{T}$ and
$L_{n,\mathcal{B}}^{\alpha}(\boldsymbol{\delta},\sigma)$ obtained form Equation
(8) of the main paper, replacing $\boldsymbol{\beta}$ by $\left(  \delta_{1}%
,..,\delta_{s},0,.^{(p-s},0\right)$ and $\boldsymbol{x}_{i}$ by
$\boldsymbol{x}_{i,s}=\left(  x_{1,s},...,x_{i,s}\right)  ^{T}.$ Now we will
prove that there exists a strict local minimizer $(\widehat{\boldsymbol{\beta
}}_{1},\widehat{\sigma})^{T}$ of $Q_{n,\mathcal{B}}^{\alpha}%
(\boldsymbol{\delta},\sigma)$ verifying $\left\Vert \widehat{\boldsymbol{\beta
}}_{1}-\boldsymbol{\beta}_{\mathcal{S}0}\right\Vert =O_{p}(\sqrt{s/n})$ and
$\left\Vert \widehat{\sigma}^{\alpha}-\sigma_{0}\right\Vert =O_{p}(n^{-1/2}).$
For $r\in\left(  0,\infty\right),$ we define the closet set $\mathcal{N}%
_{r}=\left\{  (\boldsymbol{\delta},\sigma)\in\mathbb{R}^{s}\times
\mathbb{R}^{+}:||\boldsymbol{\delta}-\boldsymbol{\beta}_{\mathcal{S}0}%
||_{2}\leq\sqrt{\frac{s}{n}}r,|\sigma-\sigma_{0}|\leq\frac{r}{\sqrt{n}%
}\right\}  $ and the event$,$%
\[
\zeta_{n}=\left\{  Q_{n,\mathcal{B}}^{\alpha}(\boldsymbol{\beta}%
_{\mathcal{S}0},\sigma_{0})<\min_{(\boldsymbol{\delta},\sigma)\in
	\partial\mathcal{N}_{r}}Q_{n,\mathcal{B}}^{\alpha}(\boldsymbol{\delta}%
,\sigma)\right\}
\]
where $\partial\mathcal{N}_{r}$ denotes the boundary of $\partial
\mathcal{N}_{r}.$ It is clear that on $\zeta_{n}$ there exists a local
minimizer $(\widehat{\boldsymbol{\beta}}_{1},\widehat{\sigma})$ of
$Q_{n,\mathcal{B}}^{\alpha}(\boldsymbol{\delta},\sigma)$ in $\mathcal{N}_{r}.$
Therefore we only need to show that $\Pr(\zeta_{n})\rightarrow1$ as
$n\rightarrow\infty$ when $r$ is large. We need to analyze the function
$Q_{n,\mathcal{B}}^{\alpha}(\boldsymbol{\delta},\sigma)$ on the boundary
$\partial\mathcal{N}_{r}.$ Let $n$ be sufficiently large such that $\sqrt
{n/n}r\leq d_{n}.$ This is possible because by assumption (A3)$^{\ast}$ of the main paper we
have $d_{n}\gg\sqrt{s/n}.$ In the same way that in the proof of Theorem
8, for $\boldsymbol{\delta}\in\mathcal{N}_{r}$ entails
$\operatorname{sg}(\boldsymbol{\delta})=\operatorname{sg}(\boldsymbol{\beta
}_{\mathcal{S}0}),$ $||\boldsymbol{\delta}-\boldsymbol{\beta}_{\mathcal{S}%
	0}||_{\infty}\leq d_{n},|\sigma-\sigma_{0}|\leq d_{n}$, $\min_{j}%
|\delta_{j}|\geq d_{n}$ and $\operatorname{sg}(\boldsymbol{\delta
})=\operatorname{sg}(\boldsymbol{\beta}_{\mathcal{S}0}).$ A second order
Taylor expansion of $Q_{n,\mathcal{B}}^{\alpha}(\boldsymbol{\delta},\sigma)$
gives
\begin{equation}\label{z}
Q_{n}^{\alpha},_{\mathcal{B}}(\boldsymbol{\delta},\sigma)=Q_{n}^{\alpha
},_{\mathcal{B}}(\boldsymbol{\beta}_{\mathcal{S}0},\sigma_{0}%
)+[(\boldsymbol{\delta},\sigma)-(\boldsymbol{\beta}_{\mathcal{S}0},\sigma
_{0})]^{T}\boldsymbol{d}+\frac{1}{2}[(\boldsymbol{\delta},\sigma
)-(\boldsymbol{\beta}_{\mathcal{S}0},\sigma_{0})]^{T}\boldsymbol{D}%
[(\boldsymbol{\delta},\sigma)-(\boldsymbol{\beta}_{\mathcal{S}0},\sigma_{0})]
\end{equation}
where $\boldsymbol{d}=(\boldsymbol{d}_{1}^{T},d_{2})^{T}$ $\ $with
$\boldsymbol{d}_{1}=\alpha\sigma_{0}^{-\frac{2\alpha+1}{\alpha+1}}\sum
_{i=1}^{n}\phi_{1,\alpha}(\boldsymbol{\beta}_{0},\sigma_{0})\boldsymbol{x}%
_{i,\mathcal{S}}+\boldsymbol{\tilde{p}}_{\lambda}^{\ast}(\boldsymbol{\beta
}_{0}),$ $d_{2}=\alpha\sigma_{0}^{-\frac{2\alpha+1}{\alpha+1}}\sum_{i=1}%
^{n}\phi_{2,\alpha}(\boldsymbol{\beta}_{0},\sigma_{0})$ and $\boldsymbol{D}%
=\frac{1}{n}\left(  \mathbb{X}_{\mathcal{S}}^{\ast T}\boldsymbol{\Sigma
}_{\alpha}(\tilde{\boldsymbol{\delta}},\tilde{\sigma})\mathbb{X}_{\mathcal{S}%
}^{\ast}+\boldsymbol{\tilde{p}}_{\lambda}^{\ast\ast}(\tilde{\boldsymbol{\delta
}},\boldsymbol{0})\right)  $. The vector $(\tilde{\boldsymbol{\delta}}%
,\tilde{\sigma})$ lies in the line segment joining $(\boldsymbol{\delta
},\sigma)$ and $(\boldsymbol{\beta}_{\mathcal{S}0},\sigma_{0}),$ and then
$(\tilde{\boldsymbol{\delta}},\tilde{\sigma})\in\mathcal{N}_{0}.$ More
generally, when the second derivative of the penalty function does not exist, it is not difficult to see that the second part of the
matrix $\boldsymbol{D}$ can be replaced by a diagonal matrix with maximum
absolute element bounded by $\max_{(\boldsymbol{\delta},\sigma)\in
	\mathcal{N}_{0}}\xi(p_{\lambda},\boldsymbol{\delta}).$

By (A2)$^{\ast}$ of the main paper we have,
\[
\Lambda_{\min}(\boldsymbol{D})=\Lambda_{\min}\left(  \frac{1}{n}\left(
\mathbb{X}_{\mathcal{S}}^{\ast T}\boldsymbol{\Sigma}_{\alpha}(\tilde
{\boldsymbol{\delta}},\tilde{\sigma})\mathbb{X}_{\mathcal{S}}^{\ast
}+\boldsymbol{\tilde{p}}_{\lambda}^{\ast\ast}(\tilde{\boldsymbol{\delta}%
},\boldsymbol{0})\right)  \right)  \geq c+\Lambda_{\min}\left(
\boldsymbol{\tilde{p}}_{\lambda}^{\ast\ast}(\tilde{\boldsymbol{\delta}%
},\boldsymbol{0})\right)  \geq c+\max\left(  -p_{\lambda}^{\prime\prime
}(|\boldsymbol{\delta}|\right)  \geq\frac{c}{2}.
\]
But $[(\boldsymbol{\delta},\sigma)-(\boldsymbol{\beta}_{\mathcal{S}0}%
,\sigma_{0})]^{T}\boldsymbol{d}+\frac{1}{2}[(\boldsymbol{\delta}%
,\sigma)-(\boldsymbol{\beta}_{\mathcal{S}0},\sigma_{0})]^{T}\boldsymbol{D}%
[(\boldsymbol{\delta},\sigma)-(\boldsymbol{\beta}_{\mathcal{S}0},\sigma_{0})]$
can be written as $$\left(  \boldsymbol{\delta}-\boldsymbol{\beta}%
_{\mathcal{S}0}\right)  \boldsymbol{d}_{1}+\frac{1}{2}(\boldsymbol{\delta
}-\boldsymbol{\beta}_{\mathcal{S}0})^{T}\boldsymbol{D}(\boldsymbol{\delta
}-\boldsymbol{\beta}_{\mathcal{S}0})+\left(  \sigma-\sigma_{0}\right)
d_{2}+\frac{1}{2}\left(  \sigma-\sigma_{0}\right)  ^{T}\boldsymbol{D}\left(
\sigma-\sigma_{0}\right)  =\boldsymbol{Y}+\boldsymbol{X}+\boldsymbol{Z}%
+\boldsymbol{V}.$$ Now, we apply that \ \ $\left\Vert \boldsymbol{Y}%
+\boldsymbol{X}+\boldsymbol{Z}+\boldsymbol{V}\right\Vert \geq\left\Vert
\boldsymbol{X}\right\Vert -\left\Vert \boldsymbol{Y}\right\Vert -\left\Vert
\boldsymbol{Z}\right\Vert +\left\Vert V\text{ }\right\Vert $ and we get%
\begin{align*}
\left\Vert \boldsymbol{Y}+\boldsymbol{X}+\boldsymbol{Z}+\boldsymbol{V}%
\right\Vert  &  \geq\frac{1}{2}\left\Vert (\boldsymbol{\delta}%
-\boldsymbol{\beta}_{\mathcal{S}0})^{T}\boldsymbol{D}(\boldsymbol{\delta
}-\boldsymbol{\beta}_{\mathcal{S}0})\right\Vert -\left\Vert \left(
\boldsymbol{\delta}-\boldsymbol{\beta}_{\mathcal{S}0}\right)  \boldsymbol{d}%
_{1}\right\Vert -\left\Vert \left(  \sigma-\sigma_{0}\right)  d_{2}\right\Vert
+\frac{1}{2}\left\Vert \left(  \sigma-\sigma_{0}\right)  ^{T}\boldsymbol{D}%
\left(  \sigma-\sigma_{0}\right)  \right\Vert \\
&  =\frac{1}{2}\frac{c}{2}\left(  \sqrt{\frac{s}{n}}r\right)  ^{2}-\sqrt
{\frac{s}{n}}r\left\Vert \boldsymbol{d}_{1}\right\Vert -n^{-1/2}r\left\vert
d_{2}\right\vert +\frac{1}{2}\frac{c}{2}n^{-1/2}r=\sqrt{\frac{s}{n}}r\left(
-||\boldsymbol{d}_{1}||_{2}+\frac{c}{4}\sqrt{\frac{s}{n}}r\right) \\
&  +\frac{r}{\sqrt{n}}\left(  -|d_{2}|+\frac{c}{4}\frac{r}{\sqrt{n}}\right)  .
\end{align*}
Therefore, $\min_{(\boldsymbol{\delta},\sigma)\in\partial\mathcal{N}_{r}%
}Q_{n,\mathcal{B}}^{\alpha}(\boldsymbol{\delta},\sigma)-Q_{n}^{\alpha
},_{\mathcal{B}}(\boldsymbol{\beta}_{\mathcal{S}0},\sigma_{0})\geq\sqrt
{\frac{s}{n}}r\left(  -||\boldsymbol{d}_{1}||_{2}+\frac{c}{4}\sqrt{\frac{s}%
	{n}}r\right)  +\frac{r}{\sqrt{n}}\left(  -|d_{2}|+\frac{c}{4}\frac{r}{\sqrt
	{n}}\right).$ 

We consider the events, $A=\left\{  ||\boldsymbol{d}%
_{1}||_{2}^{2}<\left(  \frac{c}{4}\sqrt{\frac{s}{n}}r\right)  ^{2}\right\}  $
and $B=\left\{  |\boldsymbol{d}_{2}|^{2}<\left(  \frac{c}{4}\frac{r}{\sqrt{n}%
}\right)  ^{2}\right\}.$ It is clear that $A\cap B\subset\zeta_{n}.$
Then,
\[
\mathbb{P}(\zeta_{n})\geq\mathbb{P}(A\cap B)\geq1-\mathbb{P}(A)-\mathbb{P}%
(B)\geq1-\mathbb{P}\left(  ||\boldsymbol{d}_{1}||_{2}^{2}\geq\frac{c^{2}%
	s}{16n}r^{2}\right)  -\mathbb{P}\left(  |d_{2}|^{2}<\frac{c^{2}r^{2}}%
{16n}\right)  \geq1-\frac{16n}{c^{2}sr^{2}}\mathbb{E}\left[  ||\boldsymbol{d}%
_{1}||_{2}^{2}\right]  -\frac{16n}{c^{2}r^{2}}\mathbb{E}\left[  |d_{2}%
|^{2}\right]  .
\]
The last inequality follows by Markov inequality.

Using triangular inequality, Assumptions (A2)$^{\ast}$ and (A3)$^{\ast}$ of the main paper as
well as that the function $p_{\lambda}^{\prime}$ is an increasing function we
have,
\begin{align*}
\mathbb{E}\left[  ||\boldsymbol{d}_{1}||_{2}^{2}\right]   &  \leq
\mathbb{E}\left[  \left\Vert \alpha\sigma_{0}^{-\frac{2\alpha+1}{\alpha+1}}%
{\textstyle\sum_{i=1}^{n}}
\phi_{1,\alpha}(\boldsymbol{\beta}_{0},\sigma_{0})\boldsymbol{x}%
_{i,\mathcal{S}}\right\Vert _{2}^{2}+\left\Vert \boldsymbol{\tilde{p}%
}_{\lambda}^{\ast}(\boldsymbol{\beta}_{0})\right\Vert _{2}^{2}\right] \\
&  \leq\mathbb{E}\left[  \left\Vert \alpha\sigma_{0}^{-\frac{2\alpha+1}%
	{\alpha+1}}%
{\textstyle\sum_{i=1}^{n}}
\phi_{1,\alpha}(\boldsymbol{\beta}_{0},\sigma_{0})\boldsymbol{x}%
_{i,\mathcal{S}}\right\Vert _{2}^{2}\right]  +\mathbb{E}\left[  \left\Vert
\boldsymbol{\tilde{p}}_{\lambda}^{\ast}(\boldsymbol{\beta}_{0})\right\Vert
_{2}^{2}\right]  \leq O\left(  \frac{s}{n}\right)  +sp_{\lambda}^{\prime
}(d_{n})=O\left(  sn^{-1}\right)  .
\end{align*}
In a similar way, it is possible to see that $\mathbb{E}\left[  |d_{2}%
|^{2}\right]  =O\left(  n^{-1}\right)  .$ Therefore $\mathbb{P}(\zeta_{n}%
)\geq1-O(r^{-2})-O(r^{-2})=1-O(r^{-2})$ and we have established the
convergence in probability \ of $\zeta_{n}.$ Then $\left\Vert
(\widehat{\boldsymbol{\beta}}_{1},\widehat{\sigma})-\boldsymbol{\beta}%
_{1}\right\Vert =O_{p}(\sqrt{s/n})$ and $\left\Vert \widehat{\sigma}^{\alpha
}-\sigma_{0}\right\Vert =O_{p}(n^{-1/2}).$ 

Now, we are going to establish the
sparsity. We are going to see that $\widehat{\boldsymbol{\beta}}%
^{T}:=(\widehat{\boldsymbol{\beta}}_{1},\boldsymbol{0},\widehat{\sigma})$
$\ $\ is a minimizer of $Q_{n}^{\alpha}(\boldsymbol{\theta})$. From the proof
of Theorem 8 it is only necessary to establish the inequality
(19) of the main paper. We consider the vector $\boldsymbol{\xi}=\sum_{i=1}%
^{n}\boldsymbol{\Psi}_{\alpha}$ and the event $\zeta_{2}=\{||\boldsymbol{\xi
}_{\mathcal{N}}||_{\infty}\leq u_{n}\sqrt{n}\},$ where $u_{n}=c_{1}%
^{-1/2}n^{\tau^{\ast}/2}\sqrt{\log n}.$ In the same way that in Theorem
8, $\Pr(\zeta_{2})\leq1-2(p-s)\exp(-c_{1}u_{n}^{2})\leq
1-2(p-s)\exp(-n^{1-2\tau^{\ast}}\log n)$ that tends to $1$ when $n\rightarrow
\infty$, because $\log p=O(n^{\alpha}).$ A second order Taylor expansion like
in (\ref{z}), jointly with the assumptions corresponding the the design matrix
given in (A2)$^{\ast}$ of the main paper and the bound given for the regularization parameter,
$\lambda,$ we have,
\begin{align*}
\left\Vert \frac{1}{n\lambda}\alpha\sigma^{-\frac{2\alpha+1}{\alpha+1}}%
{\textstyle\sum_{i=1}^{n}}
\phi_{1,\alpha}\left(  r_{i}(\widehat{\boldsymbol{\theta}})\right)
\boldsymbol{x}_{2,i}\right\Vert \infty &  \leq o(1)\\
&  +\frac{1}{n\lambda}\left\Vert \left(  \mathbb{X}_{\mathcal{N}}^{\ast
	T}\boldsymbol{\Sigma}_{\alpha}(\boldsymbol{\beta}_{0},\sigma_{0}%
)\mathbb{X}_{\mathcal{S}}^{\ast}\right)  \left[  (\boldsymbol{\delta}%
,\sigma)-(\boldsymbol{\beta}_{\mathcal{S}0,\sigma_{0}})\right]  \right\Vert
_{\infty}+\frac{1}{n\lambda}||\boldsymbol{\omega}||_{\infty}\\
&  \leq o(1)+\frac{O(n)}{n\lambda}||\left[  (\boldsymbol{\delta}%
,\sigma)-(\boldsymbol{\beta}_{\mathcal{S}0,\sigma_{0}})\right]  ||_{2}%
+\frac{O(n)}{n\lambda}||\left[  (\boldsymbol{\delta},\sigma
)-(\boldsymbol{\beta}_{\mathcal{S}0,\sigma_{0}})\right]  ||_{2}^{2}\\
&  \leq o(1)+O\left(  \lambda^{-1}\sqrt{s/n}\right)  =o(1)
\end{align*}
which shows that inequality.

\subsubsection{Proof Theorem 10}

On the event $\zeta_{n}$ defined in the proof of Theorem 9 it has
been shown that $\widehat{\boldsymbol{\theta}}=(\widehat{\boldsymbol{\beta}%
},\widehat{\sigma})\in\mathcal{N}_{r}\subset\mathcal{N}_{0}$ with
$\widehat{\boldsymbol{\beta}}^{T}=(\widehat{\boldsymbol{\beta}}%
,\widehat{\boldsymbol{\beta}}_{\mathcal{N}})$ $\in\mathcal{N}_{r}%
\subset\mathcal{N}_{0}$ and $\widehat{\boldsymbol{\beta}}_{\mathcal{N}%
}=\boldsymbol{0}$ is a strict minimizer of $Q_{n,\mathcal{B}}^{\alpha
}(\boldsymbol{\delta},\sigma)$ and $~\mathbb{P}(\zeta_{n})\rightarrow1.$
Therefore only it is necessary to establish the asymptotic distribution of
$\widehat{\boldsymbol{\theta}}=(\widehat{\boldsymbol{\beta}},\widehat{\sigma
}).$ We have, $\boldsymbol{0}=\nabla Q_{n,\mathcal{B}}^{\alpha}%
(\widehat{\boldsymbol{\beta}}_{\mathcal{S}}^{\alpha},\widehat{\sigma})=\nabla
L_{n,\mathcal{B}}^{\alpha}(\widehat{\boldsymbol{\beta}}_{\mathcal{S}%
},\widehat{\sigma})+\tilde{\boldsymbol{p}}_{\lambda}^{\ast}%
(\widehat{\boldsymbol{\beta}}_{\mathcal{S}})$ and a second order Taylor
expansion of $\nabla L_{n}^{\alpha},_{\mathcal{B}}$ around $(\boldsymbol{\beta
}_{\mathcal{S}0},\sigma_{0})\ $gives
\begin{align*}
\nabla L_{n}^{\alpha},_{\mathcal{B}}(\widehat{\boldsymbol{\beta}}%
_{\mathcal{S}},\widehat{\sigma})=  &  \nabla L_{n,\mathcal{B}}^{\alpha
}(\boldsymbol{\beta}_{\mathcal{S}0},\sigma_{0})+\frac{1}{n}\left(
\mathbb{X}_{\mathcal{S}}^{\ast T}\boldsymbol{\Sigma}_{\alpha}%
(\boldsymbol{\theta}_{0})\mathbb{X}_{\mathcal{S}}^{\ast}\right)  \left[
(\widehat{\boldsymbol{\beta}}_{\mathcal{S}},\widehat{\sigma}%
)-(\boldsymbol{\beta}_{\mathcal{S}0},\sigma_{0})\right] \\
&  +\frac{1}{2}\left[  (\widehat{\boldsymbol{\beta}}_{\mathcal{S}%
},\widehat{\sigma})-(\boldsymbol{\beta}_{\mathcal{S}0},\sigma_{0})\right]
^{T}\nabla^{3}L_{n,\mathcal{B}}^{\alpha}(\boldsymbol{\delta}^{\ast}%
,\sigma^{\ast})\left[  (\widehat{\boldsymbol{\beta}}_{\mathcal{S}%
},\widehat{\sigma})-(\boldsymbol{\beta}_{\mathcal{S}0},\sigma_{0})\right]  ,
\end{align*}
being $(\boldsymbol{\delta}^{\ast},\sigma^{\ast})$ some vector lying on the
line segment jointly $(\widehat{\boldsymbol{\beta}}_{\mathcal{S}%
},\widehat{\sigma})$ and $(\boldsymbol{\beta}_{\mathcal{S}0},\sigma_{0}).$ By
Equation (28) in (A2)$^\ast$ of the main paper,
\[
\nabla L_{n,\mathcal{B}}^{\alpha}(\widehat{\boldsymbol{\beta}}_{\mathcal{S}%
},\widehat{\sigma})=\nabla L_{n}^{\alpha}|_{\mathcal{B}}(\boldsymbol{\beta
}_{\mathcal{S}0},\sigma_{0})+\frac{1}{n}\left(  \mathbb{X}_{\mathcal{S}}^{\ast
	T}\boldsymbol{\Sigma}_{\alpha}(\boldsymbol{\theta}_{0})\mathbb{X}%
_{\mathcal{S}}^{\ast}\right)  \left[  (\widehat{\boldsymbol{\beta}%
}_{\mathcal{S}},\widehat{\sigma})-(\boldsymbol{\beta}_{\mathcal{S}0}%
,\sigma_{0})\right]  +O(1)\sqrt{s}||(\widehat{\boldsymbol{\beta}}%
_{\mathcal{S}},\widehat{\sigma})-(\boldsymbol{\beta}_{\mathcal{S}0},\sigma
_{0})||_{2}^{2},
\]
and%
\begin{align}
\boldsymbol{0}  &  =\nabla L_{n,\mathcal{B}}^{\alpha}(\boldsymbol{\beta
}_{\mathcal{S}0},\sigma_{0})+\frac{1}{n}\left(  \mathbb{X}_{\mathcal{S}}^{\ast
	T}\boldsymbol{\Sigma}_{\alpha}(\boldsymbol{\theta}_{0})\mathbb{X}%
_{\mathcal{S}}^{\ast}\right)  \left[  (\widehat{\boldsymbol{\beta}%
}_{\mathcal{S}},\widehat{\sigma})-(\boldsymbol{\beta}_{\mathcal{S}0}%
,\sigma_{0})\right]  +O(\sqrt{s})||(\widehat{\boldsymbol{\beta}}_{\mathcal{S}%
},\widehat{\sigma})-(\boldsymbol{\beta}_{\mathcal{S}0},\sigma_{0})||_{2}%
^{2},+\tilde{\boldsymbol{p}}_{\lambda}^{\ast}(\widehat{\boldsymbol{\beta}%
}_{\mathcal{S}})\nonumber\\
&  =\nabla L_{n,\mathcal{B}}^{\alpha}(\boldsymbol{\beta}_{\mathcal{S}0}%
,\sigma_{0})+\frac{1}{n}\left(  \mathbb{X}_{\mathcal{S}}^{\ast T}%
\boldsymbol{\Sigma}_{\alpha}(\boldsymbol{\theta}_{0})\mathbb{X}_{\mathcal{S}%
}^{\ast}\right)  \left[  (\widehat{\boldsymbol{\beta}}_{\mathcal{S}%
},\widehat{\sigma})-(\boldsymbol{\beta}_{\mathcal{S}0},\sigma_{0})\right]
+O_{P}\left(  s^{3/2}n^{-1}\right)  +\tilde{\boldsymbol{p}}_{\lambda}^{\ast
}(\widehat{\boldsymbol{\beta}}_{\mathcal{S}}). \label{Last1}%
\end{align}

By condition (A5) of the main paper, $p_{\lambda}^{\prime}(d_{n})=O(\sqrt{ns}^{-1})$, as
$(\widehat{\boldsymbol{\beta}}_{\mathcal{S}},\widehat{\sigma})\in
\mathcal{S}_{0}$ and by the monocity of $p_{\lambda}^{\prime}$ we have%
\begin{equation}
||\tilde{\boldsymbol{p}}_{\lambda}^{\ast}(\widehat{\boldsymbol{\beta}%
}_{\mathcal{S}})||_{2}\leq\sqrt{s}p_{\lambda}^{\prime}(d_{n})=O_{P}\left(
n^{-1/2}\right)  . \label{Last2}%
\end{equation}
Combining (\ref{Last1}) and (\ref{Last2}), and using $s=O(n^{1/3})$ we have
\begin{equation}
\left(  \mathbb{X}_{\mathcal{S}}^{\ast T}\boldsymbol{\Sigma}_{\alpha
}(\boldsymbol{\theta}_{0})\mathbb{X}_{\mathcal{S}}^{\ast}\right)  \left[
(\widehat{\boldsymbol{\beta}}_{\mathcal{S}},\widehat{\sigma}%
)-(\boldsymbol{\beta}_{\mathcal{S}0},\sigma_{0})\right]  =-\nabla
L_{n}^{\alpha},_{\mathcal{B}}(\boldsymbol{\beta}_{\mathcal{S}0},\sigma
_{0})-O_{P}\left(  \sqrt{n}\right)  . \label{AAA}%
\end{equation}

\bigskip\ Now we define the matrices $\boldsymbol{S}_{1,n}=\left(
\mathbb{X}_{\mathcal{S}}^{\ast T}\boldsymbol{\Sigma}_{\alpha}%
(\boldsymbol{\theta}_{0})\mathbb{X}_{\mathcal{S}}^{\ast}\right)  $ and
$\boldsymbol{S}_{2,n}=\left(  \mathbb{X}_{\mathcal{S}}^{\ast T}%
\boldsymbol{\Sigma}_{\alpha}^{\ast}(\boldsymbol{\theta}_{0})\mathbb{X}%
_{\mathcal{S}}^{\ast}\right)  $. Multiplying the two members of (\ref{AAA}) by
$\boldsymbol{S}_{2,n}^{-\frac{1}{2}}$ and using condition (A5) of the main paper, we have
\[
\boldsymbol{S}_{2,n}^{-\frac{1}{2}}\boldsymbol{S}_{1,n}\left[
(\widehat{\boldsymbol{\beta}}_{\mathcal{S}},\widehat{\sigma}%
)-(\boldsymbol{\beta}_{\mathcal{S}0},\sigma_{0})\right]  =-\boldsymbol{S}%
_{2,n}^{-\frac{1}{2}}\nabla L_{n,\mathcal{B}}^{\alpha}(\boldsymbol{\beta
}_{\mathcal{S}0},\sigma_{0})-O_{P}\left(  1\right)  ,
\]
and multiplying the two members by the matrix $\boldsymbol{A}_{n}$ given in
the statement
\[
\boldsymbol{A}_{n}\boldsymbol{S}_{2,n}^{-\frac{1}{2}}\boldsymbol{S}%
_{1,n}\left[  (\widehat{\boldsymbol{\beta}}_{\mathcal{S}},\widehat{\sigma
})-(\boldsymbol{\beta}_{\mathcal{S}0},\sigma_{0})\right]  =-\boldsymbol{A}%
_{n}\boldsymbol{S}_{2,n}^{-\frac{1}{2}}\nabla L_{n,\mathcal{B}}^{\alpha
}(\boldsymbol{\beta}_{\mathcal{S}0},\sigma_{0})-O_{P}\left(  1\right)  .
\]
Thus, by Slutsky's lemma it is enough to prove that $\boldsymbol{u}%
_{n}:=-\boldsymbol{A}_{n}\boldsymbol{S}_{2,n}^{-\frac{1}{2}}\nabla
L_{n,\mathcal{B}}^{\alpha}(\boldsymbol{\beta}_{\mathcal{S}0},\sigma_{0})$
converge in law to a normal random variable, $\mathcal{N}(\boldsymbol{0}%
,\boldsymbol{G}),$ to get the result.

For any unit vector $\boldsymbol{a}\in\mathbb{R}^{q},$ we consider the
asymptotic distribution of the linear combination $\boldsymbol{a}%
^{T}\boldsymbol{u}_{n}=-\boldsymbol{a}^{T}\boldsymbol{A}_{n}\boldsymbol{S}%
_{2,n}^{-\frac{1}{2}}\nabla L_{n}^{\alpha}|_{\mathcal{B}}(\boldsymbol{\beta
}_{\mathcal{S}0},\sigma_{0})=\sum_{i=1}^{n}\xi_{i}$ where
\[
\xi_{i}=-\boldsymbol{a}^{T}\boldsymbol{A}_{n}\boldsymbol{S}_{2,n}^{-\frac
	{1}{2}}\left[
\begin{matrix}
\phi_{1,\alpha}(r_{i}(\boldsymbol{\theta}_{0}))\boldsymbol{x}_{i}\\
\phi_{2,\alpha}(r_{i}(\boldsymbol{\theta}_{0}))
\end{matrix}
\right]  .
\]
The random variables $r_{i}(\boldsymbol{\theta}_{0})=\left(  \frac
{y_{i}-\boldsymbol{x}_{i}^{T}\boldsymbol{\beta}_{0}}{\sigma_{0}}\right)  $ are
independent, therefore the random variables $\xi_{i}$ are  independent too with
mean $\boldsymbol{0}$ and 
\begin{align*}%
{\textstyle\sum_{i=1}^{n}}
\operatorname{var}(\xi_{i})  &  =\operatorname{var}\left(  -\boldsymbol{a}%
^{T}\boldsymbol{A}_{n}\boldsymbol{S}_{2,n}^{-\frac{1}{2}}%
{\textstyle\sum_{i=1}^{n}}
\left[
\begin{matrix}
\phi_{1,\alpha}(r_{i}(\boldsymbol{\theta}_{0}))\boldsymbol{x}_{i}\\
\phi_{2,\alpha}(r_{i}(\boldsymbol{\theta}_{0}))
\end{matrix}
\right]  \right)\\  &=-\boldsymbol{a}^{T}\boldsymbol{A}_{n}\boldsymbol{S}%
_{2,n}^{-\frac{1}{2}}\operatorname{var}\left(
{\textstyle\sum_{i=1}^{n}}
\left[
\begin{matrix}
\phi_{1,\alpha}(r_{i}(\boldsymbol{\theta}_{0}))\boldsymbol{x}_{i}\\
\phi_{2,\alpha}(r_{i}(\boldsymbol{\theta}_{0}))
\end{matrix}
\right]  \right)  \boldsymbol{S}_{2,n}^{-\frac{1}{2}}\boldsymbol{A}_{n}%
^{T}\boldsymbol{a}\\
&  =\boldsymbol{a}^{T}\boldsymbol{A}_{n}\boldsymbol{S}_{2,n}^{-\frac{1}{2}%
}\boldsymbol{S}_{2,n}\boldsymbol{S}_{2,n}^{-\frac{1}{2}}\boldsymbol{A}_{n}%
^{T}\boldsymbol{a}.
\end{align*}
By hypothesis $\sum_{i=1}^{n}\operatorname{var}(\xi_{i})\underset{n\rightarrow
	\infty}{\rightarrow}\boldsymbol{a}^{T}\boldsymbol{G}\boldsymbol{a}.$ Finally,
by Assumption (A5) of the main paper and using the Cauchy-Schwarz inequality, we have
\begin{align*}%
{\textstyle\sum_{i=1}^{n}}
\mathbb{E}\left[  |\xi_{i}|^{3}\right]   &  \leq%
{\textstyle\sum_{i=1}^{n}}
\left\vert \boldsymbol{a}^{T}\boldsymbol{A}_{n}\boldsymbol{S}_{2,n}^{-\frac
	{1}{2}}\boldsymbol{x}_{\mathcal{S}i}^{\ast}\right\vert ^{3}\max_{k=1,2}%
|\phi_{k,\alpha}(r_{i}(\boldsymbol{\theta}_{0}))|=O(1)%
{\textstyle\sum_{i=1}^{n}}
\left\vert \boldsymbol{a}^{T}\boldsymbol{A}_{n}\boldsymbol{S}_{2,n}^{-\frac
	{1}{2}}\boldsymbol{x}_{\mathcal{S}i}^{\ast}\right\vert ^{3}\\
&  \leq O(1)%
{\textstyle\sum_{i=1}^{n}}
\left\Vert \boldsymbol{a}^{T}\boldsymbol{A}_{n}\right\Vert _{2}^{3}\left\Vert
\boldsymbol{S}_{2,n}^{-\frac{1}{2}}\boldsymbol{x}_{\mathcal{S}i}^{\ast
}\right\Vert _{2}^{3}\leq O(1)\left\Vert \boldsymbol{a}^{T}\boldsymbol{A}%
_{n}\right\Vert _{2}^{3}%
{\textstyle\sum_{i=1}^{n}}
\left\Vert \boldsymbol{x}_{\mathcal{S}i}^{\ast}\boldsymbol{S}_{2,n}%
^{-1}\boldsymbol{x}_{\mathcal{S}i}^{\ast}\right\Vert _{2}^{3/2}=o\left(
1\right)  .
\end{align*}
where $\boldsymbol{x}_{\mathcal{S}i}^{\ast}=(\boldsymbol{x}_{\mathcal{S}i}%
^{T},1)^{T}$ . Applying Lyaupunov's theorem we have, $\boldsymbol{a}%
^{T}\boldsymbol{u}_{n}\underset{n\rightarrow\infty}{\overset{\mathcal{L}%
	}{\rightarrow}}\mathcal{N}(\boldsymbol{0},\boldsymbol{a}^{T}\boldsymbol{G}%
\boldsymbol{a}).$ Thus, this asymptotic normality holds for any vector
$\boldsymbol{a}\in\mathbb{R}^{q},$ 
\[
\boldsymbol{u}_{n}\underset{n\rightarrow\infty}{\overset{\mathcal{L}%
	}{\rightarrow}}\mathcal{N}(\boldsymbol{0},\boldsymbol{G}).
\]

\subsection{Additional Numerical Results}
Tables \ref{p100e0signal1}–\ref{p200e1signal0x1} present the simulation results, under the set-up discussed in Section 6 of the main paper, for $p = 100$ and  $p=200$ covariates. For each number of covariates, we consider pure data, 
$10\%$ $Y$ outliers and $\boldsymbol{X}$ outliers respectively. 
\begin{table}[H]
	\centering
	\caption{Performance measures obtained by different methods for $p=100$, strong signal and no outliers}
	\begin{tabular}{l|rrr|rrr|r}
		\hline
		Method	& MS($\widehat{\boldsymbol{\beta}}$) & TP($\widehat{\boldsymbol{\beta}}$) & TN($\widehat{\boldsymbol{\beta}}$) & MSES($\widehat{\boldsymbol{\beta}}$) &  MSES($\widehat{\boldsymbol{\beta}}$)  &  EE($\widehat{\sigma}$) & APrB($\widehat{\boldsymbol{\beta}}$)  \\ 
		& & & &  $(10^{-2})$ & $(10^{-5})$ &  $(10^{-2})$ &  $(10^{-2})$ \\
		\hline
		LS-LASSO & 6.45 & 1.00 & 0.98 & 1.91 & 3.10 & 32.79 & 4.93 \\ 
		LS-SCAD & 4.90 & 0.98 & 1.00 & 15.95 & 0.00 & 70.66 & 6.59\\ 
		LS-MCP & 4.90 & 0.98 & 1.00 & 11.25 & 0.00 & 55.29 & 5.87\\ 
		LAD-Lasso & 6.14 & 1.00 & 0.99 & 2.96 & 6.79 & 32.78 & 5.06  \\ 
		RLARS & 8.81 & 1.00 & 0.96 & 0.94 & 30.80 & 7.46 & 4.29 \\ 
		sLTS & 5.76 & 1.00 & 0.99 & 6.59 & 4.23 & 25.17 & 6.54  \\ 
		RANSAC & 9.13 & 1.00 & 0.96 & 3.05 & 28.48 & 9.93 & 4.97 \\ 
		\hline
		DPD-lasso $\alpha = $ 0.1 & 14.66 & 1.00 & 0.90 & 4.15 & 16.53 & 18.62 & 5.61 \\ 
		DPD-lasso $\alpha = $ 0.3 & 9.08 & 1.00 & 0.96 & 5.98 & 7.35 & 24.33 & 6.57 \\ 
		DPD-lasso $\alpha = $ 0.5 & 6.29 & 1.00 & 0.99 & 7.11 & 4.37 & 26.61 & 6.66 \\ 
		DPD-lasso $\alpha = $ 0.7 & 5.36 & 1.00 & 1.00 & 8.39 & 3.98 & 29.28 & 6.90 \\ 
		DPD-lasso $\alpha = $ 1 & 7.73 & 1.00 & 0.97 & 33.15 & 2173.95 & 27.90 & 8.95 \\ 
		\hline
		LDPD-lasso $\alpha = $ 0.1 & 5.29 & 1.00 & 1.00 & 7.19 & 1.55 & 29.56 & 6.90 \\ 
		LDPD-lasso $\alpha = $ 0.2 & 5.28 & 1.00 & 1.00 & 7.35 & 1.72 & 29.73 & 6.88  \\ 
		LDPD-lasso $\alpha = $ 0.3 & 5.28 & 1.00 & 1.00 & 7.81 & 2.07 & 30.36 & 6.91  \\ 
		LDPD-lasso $\alpha = $ 0.4 & 5.26 & 1.00 & 1.00 & 10.31 & 2.52 & 34.86 & 7.40  \\ 
		LDPD-lasso $\alpha = $ 0.5 & 5.24 & 1.00 & 1.00 & 18.53 & 3.50 & 47.82 & 8.89 \\ 
		\hline
		DPD-ncv $\alpha = $ 0.1 & 5.00 & 1.00 & 1.00 & 0.44 & 0.00 & 4.44 & 4.03  \\ 
		DPD-ncv $\alpha = $ 0.3 & 5.00 & 1.00 & 1.00 & 0.68 & 0.00 & 8.34 & 4.08  \\ 
		DPD-ncv $\alpha = $ 0.5 & 5.00 & 1.00 & 1.00 & 1.01 & 0.00 & 11.64 & 4.13  \\ 
		DPD-ncv $\alpha = $ 0.7 & 4.99 & 1.00 & 1.00 & 1.50 & 0.00 & 14.12 & 4.15 \\ 
		DPD-ncv $\alpha = $ 1 & 4.97 & 0.99 & 1.00 & 2.00 & 0.00 & 17.02 & 4.15 \\ 
		\hline
		MNPRPE-ncv $\alpha = $ 0.1 & 5.02 & 1.00 & 1.00 & 0.34 & 0.07 & 3.10 & 3.99  \\ 
		MNPRPE-ncv $\alpha = $ 0.2 & 5.01 & 1.00 & 1.00 & 0.34 & 0.04 & 3.21 & 3.98  \\ 
		MNPRPE-ncv $\alpha = $ 0.3 & 5.01 & 1.00 & 1.00 & 0.36 & 0.03 & 3.42 & 3.97 \\ 
		MNPRPE-ncv $\alpha = $ 0.4 & 5.01 & 1.00 & 1.00 & 0.37 & 0.02 & 3.64 & 3.96 \\ 
		MNPRPE-ncv $\alpha = $ 0.5 & 5.00 & 1.00 & 1.00 & 0.38 & 0.00 & 3.87 & 3.95 \\ 
		\hline
	\end{tabular}
	\label{p100e0signal1}
\end{table}
\begin{table}[H]
	\centering
	\caption{Performance measures obtained by different methods for $p=100$, weak signal and no outliers}
	\begin{tabular}{l|rrr|rrr|r}
		\hline
		Method	& MS($\widehat{\boldsymbol{\beta}}$) & TP($\widehat{\boldsymbol{\beta}}$) & TN($\widehat{\boldsymbol{\beta}}$) & MSES($\widehat{\boldsymbol{\beta}}$) &  MSES($\widehat{\boldsymbol{\beta}}$)  &  EE($\widehat{\sigma}$) & APrB($\widehat{\boldsymbol{\beta}}$)  \\ 
		& & & &  $(10^{-2})$ & $(10^{-5})$ &  $(10^{-2})$ &  $(10^{-2})$ \\
		\hline
		LS-LASSO & 7.81 & 1.00 & 0.97 & 1.56 & 4.98 & 28.65 & 4.66  \\ 
		LS-SCAD & 6.73 & 1.00 & 0.98 & 0.35 & 1.85 & 19.91 & 3.98  \\ 
		LS-MCP & 6.05 & 1.00 & 0.99 & 0.34 & 3.53 & 19.92 & 3.94 \\ 
		LAD-Lasso & 6.13 & 1.00 & 0.99 & 2.99 & 7.29 & 32.66 & 5.10 \\ 
		RLARS & 9.42 & 1.00 & 0.95 & 0.48 & 46.47 & 8.56 & 4.05 \\ 
		sLTS & 19.36 & 1.00 & 0.85 & 3.52 & 48.90 & 11.70 & 5.09 \\ 
		RANSAC & 12.33 & 1.00 & 0.92 & 2.87 & 69.31 & 19.22 & 5.05 \\ 
		\hline
		DPD-lasso $\alpha = $ 0.1 & 14.99 & 1.00 & 0.89 & 4.03 & 16.59 & 18.27 & 5.68 \\ 
		DPD-lasso $\alpha = $ 0.3 & 8.74 & 1.00 & 0.96 & 6.09 & 6.80 & 24.69 & 6.52 \\ 
		DPD-lasso  $\alpha = $ 0.5 & 6.08 & 1.00 & 0.99 & 7.17 & 3.96 & 26.72 & 6.68 \\ 
		DPD-lasso $\alpha = $ 0.7 & 5.36 & 1.00 & 1.00 & 8.48 & 3.94 & 29.50 & 6.92 \\ 
		DPD-lasso  $\alpha = $ 1 & 5.66 & 1.00 & 0.99 & 9.43 & 8.26 & 27.68 & 7.10 \\ 
		\hline
		LDPD-lasso $\alpha = $ 0.1 & 5.29 & 1.00 & 1.00 & 7.19 & 1.53 & 29.56 & 6.90  \\ 
		LDPD-lasso $\alpha = $ 0.2 & 5.28 & 1.00 & 1.00 & 7.35 & 1.71 & 29.73 & 6.88  \\ 
		LDPD-lasso $\alpha = $ 0.3 & 5.28 & 1.00 & 1.00 & 7.81 & 2.05 & 30.36 & 6.91  \\ 
		LDPD-lasso $\alpha = $ 0.4 & 5.23 & 0.99 & 1.00 & 9.88 & 2.50 & 34.48 & 7.35  \\ 
		LDPD-lasso $\alpha = $ 0.5 & 5.14 & 0.98 & 1.00 & 14.25 & 2.88 & 44.30 & 8.21  \\ 
		\hline
		DPD-ncv $\alpha = $ 0.1 & 5.07 & 1.00 & 1.00 & 0.76 & 0.19 & 4.78 & 4.11 \\ 
		DPD-ncv $\alpha = $ 0.3 & 5.05 & 1.00 & 1.00 & 0.88 & 0.21 & 9.13 & 4.09 \\ 
		DPD-ncv $\alpha = $ 0.5 & 5.02 & 0.99 & 1.00 & 1.43 & 0.26 & 12.79 & 4.23   \\ 
		DPD-ncv $\alpha = $ 0.7 & 4.99 & 0.98 & 1.00 & 1.71 & 0.29 & 15.67 & 4.19  \\ 
		DPD-ncv $\alpha = $ 1 & 4.91 & 0.97 & 1.00 & 2.92 & 0.27 & 18.94 & 4.46 \\ 
		\hline
		MNPRPE-ncv $\alpha = $ 0.1 & 5.08 & 1.00 & 1.00 & 0.53 & 0.17 & 3.24 & 4.01 \\ 
		MNPRPE-ncv $\alpha = $ 0.2 & 5.03 & 1.00 & 1.00 & 0.53 & 0.17 & 3.38 & 4.00 \\ 
		MNPRPE-ncv $\alpha = $ 0.3 & 5.01 & 1.00 & 1.00 & 0.52 & 0.18 & 3.58 & 3.98 \\ 
		MNPRPE-ncv $\alpha = $ 0.4 & 5.01 & 1.00 & 1.00 & 0.55 & 0.19 & 3.79 & 3.98  \\ 
		MNPRPE-ncv $\alpha = $ 0.5 & 5.01 & 1.00 & 1.00 & 0.56 & 0.20 & 4.02 & 3.96  \\ 
		\hline
	\end{tabular}
	
	\label{p100e0signal0}
\end{table}

\begin{table}[H]
	\centering
	\caption{Performance measures obtained by different methods for $p=100$, strong signal and $Y-$outliers}
	\begin{tabular}{l|rrr|rrr|r}
		\hline
		Method	& MS($\widehat{\boldsymbol{\beta}}$) & TP($\widehat{\boldsymbol{\beta}}$) & TN($\widehat{\boldsymbol{\beta}}$) & MSES($\widehat{\boldsymbol{\beta}}$) &  MSES($\widehat{\boldsymbol{\beta}}$)  &  EE($\widehat{\sigma}$) & APrB($\widehat{\boldsymbol{\beta}}$)  \\ 
		& & & &  $(10^{-2})$ & $(10^{-5})$ &  $(10^{-2})$ &  $(10^{-2})$ \\
		\hline
		LS-LASSO & 6.44 & 1.00 & 0.98 & 1.91 & 3.10 & 32.79 & 4.93  \\ 
		LS-SCAD & 4.90 & 0.98 & 1.00 & 15.95 & 0.00 & 70.66 & 6.59 \\ 
		LS-MCP & 4.90 & 0.98 & 1.00 & 11.25 & 0.00 & 55.29 & 5.87 \\ 
		LAD-Lasso & 6.13 & 1.00 & 0.99 & 2.95 & 6.97 & 32.77 & 5.09  \\ 
		RLARS & 8.81 & 1.00 & 0.96 & 0.94 & 30.80 & 7.46 & 4.29  \\ 
		sLTS & 5.76 & 1.00 & 0.99 & 6.59 & 4.23 & 25.17 & 6.54 \\ 
		RANSAC & 9.44 & 1.00 & 0.95 & 3.73 & 32.80 & 11.77 & 5.39 \\ 
		\hline
		DPD-lasso $\alpha = $ 0.1 & 16.18 & 1.00 & 0.88 & 3.65 & 20.16 & 18.25 & 5.52 \\ 
		DPD-lasso $\alpha = $ 0.3 & 9.29 & 1.00 & 0.95 & 5.53 & 8.42 & 23.70 & 6.14 \\ 
		DPD-lasso $\alpha = $ 0.5 & 6.20 & 1.00 & 0.99 & 6.68 & 3.92 & 27.26 & 6.53 \\ 
		DPD-lasso $\alpha = $ 0.7 & 5.50 & 1.00 & 0.99 & 7.51 & 3.93 & 29.44 & 6.64 \\ 
		DPD-lasso $\alpha = $ 1 & 5.80 & 1.00 & 0.99 & 7.80 & 6.88 & 28.06 & 6.49 \\ 
		\hline
		LDPD-lasso $\alpha = $ 0.1 & 5.28 & 1.00 & 1.00 & 10.32 & 2.71 & 35.42 & 6.95 \\ 
		LDPD-lasso $\alpha = $ 0.2 & 5.30 & 1.00 & 1.00 & 7.41 & 1.79 & 29.62 & 6.86  \\ 
		LDPD-lasso $\alpha = $ 0.3 & 5.35 & 1.00 & 1.00 & 7.97 & 2.10 & 30.44 & 6.93 \\ 
		LDPD-lasso $\alpha = $ 0.4 & 5.34 & 1.00 & 1.00 & 10.80 & 2.59 & 35.02 & 7.48  \\ 
		LDPD-lasso $\alpha = $ 0.5 & 5.28 & 0.99 & 1.00 & 37.95 & 3.79 & 54.34 & 9.90 \\ 
		\hline
		DPD-ncv $\alpha = $ 0.1 & 4.99 & 1.00 & 1.00 & 2.24 & 0.00 & 7.30 & 4.72  \\ 
		DPD-ncv  $\alpha = $ 0.3 & 4.99 & 1.00 & 1.00 & 2.87 & 0.00 & 9.45 & 4.77 \\ 
		DPD-ncv $\alpha = $ 0.5 & 4.99 & 1.00 & 1.00 & 3.42 & 0.00 & 11.65 & 4.82  \\ 
		DPD-ncv $\alpha = $ 0.7 & 4.98 & 1.00 & 1.00 & 4.18 & 0.00 & 13.36 & 4.99  \\ 
		DPD-ncv  $\alpha = $ 1 & 4.96 & 0.99 & 1.00 & 4.98 & 0.00 & 15.33 & 5.04  \\ 
		\hline
		MNPRPE-ncv $\alpha = $ 0.1 & 5.03 & 1.00 & 1.00 & 0.37 & 0.17 & 3.17 & 3.99 \\ 
		MNPRPE-ncv $\alpha = $ 0.2 & 5.01 & 1.00 & 1.00 & 0.38 & 0.06 & 3.24 & 3.99 \\ 
		MNPRPE-ncv $\alpha = $ 0.3 & 5.01 & 1.00 & 1.00 & 0.39 & 0.04 & 3.44 & 3.98 \\ 
		MNPRPE-ncv $\alpha = $ 0.4 & 5.01 & 1.00 & 1.00 & 0.41 & 0.03 & 3.73 & 3.96  \\ 
		MNPRPE-ncv $\alpha = $ 0.5 & 5.00 & 1.00 & 1.00 & 0.58 & 0.00 & 4.22 & 4.00  \\ 
		\hline
	\end{tabular}
	
	\label{p100e1signal1}
\end{table}
\begin{table}[H]
	\centering
	\caption{Performance measures obtained by different methods for $p=100$, weak signal and $Y-$outliers}
	\begin{tabular}{l|rrr|rrr|r}
		\hline
		Method	& MS($\widehat{\boldsymbol{\beta}}$) & TP($\widehat{\boldsymbol{\beta}}$) & TN($\widehat{\boldsymbol{\beta}}$) & MSES($\widehat{\boldsymbol{\beta}}$) &  MSES($\widehat{\boldsymbol{\beta}}$)  &  EE($\widehat{\sigma}$) & APrB($\widehat{\boldsymbol{\beta}}$)  \\ 
		& & & &  $(10^{-2})$ & $(10^{-5})$ &  $(10^{-2})$ &  $(10^{-2})$ \\
		\hline
		LS-LASSO & 7.78 & 1.00 & 0.97 & 1.56 & 4.94 & 28.75 & 4.67  \\ 
		LS-SCAD & 6.64 & 1.00 & 0.98 & 0.35 & 1.74 & 20.00 & 3.98  \\ 
		LS-MCP & 5.96 & 1.00 & 0.99 & 0.34 & 3.32 & 20.10 & 3.94  \\ 
		LAD-Lasso & 6.12 & 1.00 & 0.99 & 3.00 & 7.42 & 32.76 & 5.13  \\ 
		RLARS & 9.42 & 1.00 & 0.95 & 0.48 & 46.47 & 8.56 & 4.05\\ 
		sLTS & 19.36 & 1.00 & 0.85 & 3.52 & 48.90 & 11.70 & 5.09\\ 
		RANSAC & 12.09 & 1.00 & 0.93 & 2.71 & 63.20 & 18.40 & 4.88 \\ 
		\hline
		DPD-lasso $\alpha = $ 0.1 & 16.41 & 1.00 & 0.88 & 3.53 & 21.25 & 18.49 & 5.65 \\ 
		DPD-lasso $\alpha = $ 0.3 & 8.67 & 1.00 & 0.96 & 5.70 & 7.66 & 24.21 & 6.13 \\ 
		DPD-lasso $\alpha = $ 0.5 & 6.02 & 1.00 & 0.99 & 6.74 & 3.50 & 27.54 & 6.53 \\ 
		DPD-lasso $\alpha = $ 0.7 & 5.50 & 1.00 & 0.99 & 7.55 & 3.89 & 29.55 & 6.65 \\ 
		DPD-lasso $\alpha = $ 1 & 5.80 & 1.00 & 0.99 & 7.82 & 6.83 & 28.15 & 6.50 \\ 
		\hline
		LDPD-lasso $\alpha = $ 0.1 & 5.30 & 1.00 & 1.00 & 7.30 & 1.62 & 29.61 & 6.92 \\ 
		LDPD-lasso $\alpha = $ 0.2 & 5.30 & 1.00 & 1.00 & 7.44 & 1.77 & 29.69 & 6.87 \\ 
		LDPD-lasso $\alpha = $ 0.3 & 5.34 & 1.00 & 1.00 & 7.97 & 2.07 & 30.44 & 6.93  \\ 
		LDPD-lasso $\alpha = $ 0.4 & 5.32 & 1.00 & 1.00 & 10.05 & 2.56 & 34.40 & 7.38  \\ 
		LDPD-lasso $\alpha = $ 0.5 & 5.15 & 0.97 & 1.00 & 15.20 & 2.97 & 44.88 & 8.37  \\ 
		\hline
		DPD-ncv $\alpha = $ 0.1 & 4.99 & 0.99 & 1.00 & 2.01 & 0.04 & 3.94 & 4.30 \\ 
		DPD-ncv $\alpha = $ 0.3 & 4.97 & 0.99 & 1.00 & 2.18 & 0.06 & 5.19 & 4.32  \\ 
		DPD-ncv $\alpha = $ 0.5 & 4.90 & 0.97 & 1.00 & 2.32 & 0.08 & 7.09 & 4.35  \\ 
		DPD-ncv $\alpha = $ 0.7 & 4.87 & 0.97 & 1.00 & 2.40 & 0.06 & 9.09 & 4.37  \\ 
		DPD-ncv $\alpha = $ 1 & 4.78 & 0.95 & 1.00 & 2.76 & 0.08 & 11.79 & 4.44  \\ 
		\hline
		MNPRPE-ncv $\alpha = $ 0.1 & 5.16 & 1.00 & 1.00 & 0.74 & 0.22 & 3.52 & 4.12  \\ 
		MNPRPE-ncv $\alpha = $ 0.2 & 5.04 & 0.99 & 1.00 & 0.90 & 0.09 & 3.80 & 4.12 \\ 
		MNPRPE-ncv $\alpha = $ 0.3 & 4.99 & 0.99 & 1.00 & 0.99 & 0.05 & 3.99 & 4.11  \\ 
		MNPRPE-ncv $\alpha = $ 0.4 & 4.95 & 0.99 & 1.00 & 1.20 & 0.03 & 4.41 & 4.15 \\ 
		MNPRPE-ncv $\alpha = $ 0.5 & 4.91 & 0.98 & 1.00 & 1.45 & 0.02 & 4.86 & 4.16  \\ 
		\hline
	\end{tabular}
	
	\label{p100e1signal0}
\end{table}

\begin{table}[H]
	\centering
	\caption{Performance measures obtained by different methods for $p=100$, strong signal and $\boldsymbol{X}-$outliers}
	\begin{tabular}{l|rrr|rrr|r}
		\hline
		\multicolumn{8}{c}{Strong signal} \\
		\hline
		Method	& MS($\widehat{\boldsymbol{\beta}}$) & TP($\widehat{\boldsymbol{\beta}}$) & TN($\widehat{\boldsymbol{\beta}}$) & MSES($\widehat{\boldsymbol{\beta}}$) &  MSES($\widehat{\boldsymbol{\beta}}$)  &  EE($\widehat{\sigma}$) & APrB($\widehat{\boldsymbol{\beta}}$)  \\ 
		& & & &  $(10^{-2})$ & $(10^{-5})$ &  $(10^{-2})$ &  $(10^{-2})$ \\
		\hline
		LS-LASSO & 6.44 & 1.00 & 0.98 & 1.91 & 3.10 & 32.79 & 4.93 \\ 
		LS-SCAD & 4.90 & 0.98 & 1.00 & 15.95 & 0.00 & 70.66 & 6.59  \\ 
		LS-MCP & 4.90 & 0.98 & 1.00 & 11.25 & 0.00 & 55.29 & 5.87  \\ 
		LAD-Lasso & 6.13 & 1.00 & 0.99 & 2.95 & 6.97 & 32.77 & 5.09\\ 
		RLARS & 8.81 & 1.00 & 0.96 & 0.94 & 30.80 & 7.46 & 4.29 \\ 
		sLTS & 5.76 & 1.00 & 0.99 & 6.59 & 4.23 & 25.17 & 6.54  \\ 
		RANSAC & 9.25 & 1.00 & 0.96 & 3.30 & 39.16 & 11.09 & 5.01 \\ 
		\hline
		DPD-lasso $\alpha = $ 0.1 & 5.29 & 1.00 & 1.00 & 7.36 & 1.53 & 30.22 & 6.92 \\ 
		DPD-lasso $\alpha = $ 0.3 & 5.30 & 1.00 & 1.00 & 7.97 & 2.07 & 31.06 & 6.99 \\ 
		DPD-lasso $\alpha = $ 0.5 & 5.24 & 1.00 & 1.00 & 18.66 & 3.47 & 48.24 & 8.89 \\ 
		DPD-lasso $\alpha = $ 0.7 & 5.36 & 1.00 & 1.00 & 8.39 & 3.98 & 29.30 & 6.91 \\ 
		DPD-lasso $\alpha = $ 1 & 5.67 & 1.00 & 0.99 & 9.36 & 7.39 & 27.71 & 7.05 \\  
		\hline
		LDPD-lasso  $\alpha = $ 0.1 & 5.29 & 1.00 & 1.00 & 7.19 & 1.55 & 29.56 & 6.90 \\ 
		LDPD-lasso  $\alpha = $ 0.2 & 5.28 & 1.00 & 1.00 & 7.35 & 1.72 & 29.73 & 6.88  \\ 
		LDPD-lasso  $\alpha = $ 0.3 & 5.28 & 1.00 & 1.00 & 7.81 & 2.07 & 30.36 & 6.91  \\ 
		LDPD-lasso  $\alpha = $ 0.4 & 5.26 & 1.00 & 1.00 & 10.31 & 2.52 & 34.86 & 7.40 \\ 
		LDPD-lasso  $\alpha = $ 0.5 & 5.24 & 1.00 & 1.00 & 18.53 & 3.50 & 47.82 & 8.89 \\ 
		\hline
		DPD-ncv $\alpha = $ 0.1 & 5.02 & 1.00 & 1.00 & 0.34 & 0.15 & 4.53 & 3.98  \\ 
		DPD-ncv $\alpha = $ 0.3 & 5.02 & 1.00 & 1.00 & 0.38 & 0.15 & 8.22 & 3.98  \\ 
		DPD-ncv $\alpha = $ 0.5 & 5.09 & 1.00 & 1.00 & 0.44 & 0.74 & 11.48 & 3.99 \\ 
		DPD-ncv $\alpha = $ 0.7 & 5.08 & 1.00 & 1.00 & 0.57 & 0.62 & 14.00 & 3.96  \\ 
		DPD-ncv $\alpha = $ 1 & 5.02 & 1.00 & 1.00 & 0.73 & 0.11 & 16.93 & 3.92 \\ 
		\hline
		MNPRPE-ncv $\alpha = $ 0.1 & 5.02 & 1.00 & 1.00 & 0.34 & 0.07 & 3.10 & 3.99  \\ 
		MNPRPE-ncv $\alpha = $ 0.2 & 5.01 & 1.00 & 1.00 & 0.34 & 0.04 & 3.21 & 3.98  \\ 
		MNPRPE-ncv $\alpha = $ 0.3 & 5.01 & 1.00 & 1.00 & 0.36 & 0.03 & 3.41 & 3.97  \\ 
		MNPRPE-ncv $\alpha = $ 0.4 & 5.01 & 1.00 & 1.00 & 0.37 & 0.02 & 3.63 & 3.95 \\ 
		MNPRPE-ncv $\alpha = $ 0.5 & 5.00 & 1.00 & 1.00 & 0.38 & 0.00 & 3.86 & 3.95 \\ 
		\hline
		
	\end{tabular}
	\label{p100e1signal1x1}
\end{table}

\begin{table}[H]
	\centering
	\caption{Performance measures obtained by different methods for $p=100$, weak signal and $\boldsymbol{X}-$outliers}
	\begin{tabular}{l|rrr|rrr|r}
		\hline
		Method	& MS($\widehat{\boldsymbol{\beta}}$) & TP($\widehat{\boldsymbol{\beta}}$) & TN($\widehat{\boldsymbol{\beta}}$) & MSES($\widehat{\boldsymbol{\beta}}$) &  MSES($\widehat{\boldsymbol{\beta}}$)  &  EE($\widehat{\sigma}$) & APrB($\widehat{\boldsymbol{\beta}}$)  \\ 
		& & & &  $(10^{-2})$ & $(10^{-5})$ &  $(10^{-2})$ &  $(10^{-2})$ \\
		\hline
		LS-LASSO & 7.78 & 1.00 & 0.97 & 1.56 & 4.94 & 28.75 & 4.67 \\ 
		LS-SCAD & 6.64 & 1.00 & 0.98 & 0.35 & 1.74 & 20.00 & 3.98  \\ 
		LS-MCP & 5.96 & 1.00 & 0.99 & 0.34 & 3.32 & 20.10 & 3.94  \\ 
		LAD-Lasso & 6.12 & 1.00 & 0.99 & 3.00 & 7.42 & 32.76 & 5.13  \\ 
		RLARS & 9.42 & 1.00 & 0.95 & 0.48 & 46.47 & 8.56 & 4.05 \\ 
		sLTS & 19.36 & 1.00 & 0.85 & 3.52 & 48.90 & 11.70 & 5.09  \\ 
		RANSAC & 12.72 & 1.00 & 0.92 & 2.88 & 70.86 & 20.46 & 4.99  \\ 
		\hline
		DPD-lasso $\alpha = $ 0.1 & 5.29 & 1.00 & 1.00 & 7.36 & 1.52 & 30.22 & 6.92 \\ 
		DPD-lasso $\alpha = $ 0.3 & 5.30 & 1.00 & 1.00 & 7.97 & 2.04 & 31.07 & 6.99 \\ 
		DPD-lasso $\alpha = $ 0.5 & 5.18 & 0.99 & 1.00 & 12.52 & 2.96 & 41.69 & 7.99 \\ 
		DPD-lasso $\alpha = $ 0.7 & 5.36 & 1.00 & 1.00 & 8.48 & 3.94 & 29.50 & 6.92 \\ 
		DPD-lasso $\alpha = $ 1 & 5.66 & 1.00 & 0.99 & 9.43 & 8.26 & 27.68 & 7.10 \\ 
		\hline
		LDPD-lasso  $\alpha = $ 0.1 & 5.29 & 1.00 & 1.00 & 7.19 & 1.53 & 29.56 & 6.90  \\ 
		LDPD-lasso  $\alpha = $ 0.2 & 5.28 & 1.00 & 1.00 & 7.35 & 1.71 & 29.73 & 6.88  \\ 
		LDPD-lasso  $\alpha = $ 0.3 & 5.28 & 1.00 & 1.00 & 7.81 & 2.05 & 30.36 & 6.91  \\ 
		LDPD-lasso  $\alpha = $ 0.4 & 5.23 & 0.99 & 1.00 & 9.88 & 2.50 & 34.48 & 7.35  \\ 
		LDPD-lasso  $\alpha = $ 0.5 & 5.14 & 0.98 & 1.00 & 14.25 & 2.88 & 44.30 & 8.21  \\ 
		\hline
		DPD-ncv $\alpha = $ 0.1 & 5.14 & 1.00 & 1.00 & 0.44 & 0.76 & 5.00 & 3.98 \\
		DPD-ncv $\alpha = $ 0.3 & 5.12 & 1.00 & 1.00 & 0.58 & 1.04 & 9.53 & 3.96 \\ 
		DPD-ncv $\alpha = $ 0.5 & 5.09 & 0.99 & 1.00 & 0.82 & 1.36 & 13.11 & 4.05  \\ 
		DPD-ncv $\alpha = $ 0.7 & 5.07 & 0.99 & 1.00 & 1.12 & 1.84 & 15.91 & 4.00  \\ 
		DPD-ncv $\alpha = $ 1 & 5.14 & 0.98 & 1.00 & 1.37 & 2.80 & 19.23 & 4.05  \\ 
		\hline
		MNPRPE-ncv $\alpha = $ 0.1 & 5.08 & 1.00 & 1.00 & 0.53 & 0.17 & 3.24 & 4.02\\ 
		MNPRPE-ncv $\alpha = $ 0.2 & 5.03 & 1.00 & 1.00 & 0.52 & 0.17 & 3.39 & 4.00  \\ 
		MNPRPE-ncv $\alpha = $ 0.3 & 5.01 & 1.00 & 1.00 & 0.52 & 0.18 & 3.58 & 3.98 \\ 
		MNPRPE-ncv $\alpha = $ 0.4 & 5.01 & 1.00 & 1.00 & 0.54 & 0.19 & 3.78 & 3.97\\ 
		MNPRPE-ncv $\alpha = $ 0.5 & 5.01 & 1.00 & 1.00 & 0.56 & 0.20 & 4.01 & 3.96 \\ 
		\hline
	\end{tabular}
	\label{p100e1signal0x1}
\end{table}

\begin{table}[H]
	\centering
	\caption{Performance measures obtained by different methods for $p=200$, strong signal and no outliers}
	\begin{tabular}{l|rrr|rrr|r}
		\hline
		Method	& MS($\widehat{\boldsymbol{\beta}}$) & TP($\widehat{\boldsymbol{\beta}}$) & TN($\widehat{\boldsymbol{\beta}}$) & MSES($\widehat{\boldsymbol{\beta}}$) &  MSES($\widehat{\boldsymbol{\beta}}$)  &  EE($\widehat{\sigma}$) & APrB($\widehat{\boldsymbol{\beta}}$)  \\ 
		& & & &  $(10^{-2})$ & $(10^{-5})$ &  $(10^{-2})$ &  $(10^{-2})$ \\
		\hline
		LS-LASSO & 6.90 & 1.00 & 0.99 & 2.13 & 1.67 & 34.00 & 4.72  \\ 
		LS-SCAD & 4.94 & 0.99 & 1.00 & 16.06 & 0.00 & 68.46 & 7.53  \\ 
		LS-MCP & 4.94 & 0.99 & 1.00 & 11.14 & 0.00 & 52.61 & 6.69 \\ 
		LAD-Lasso & 6.29 & 1.00 & 0.99 & 3.67 & 2.28 & 37.67 & 5.35 \\ 
		RLARS & 7.86 & 0.99 & 0.99 & 1.72 & 10.18 & 7.61 & 4.53  \\ 
		sLTS & 6.06 & 1.00 & 0.99 & 6.61 & 2.19 & 24.14 & 6.23  \\ 
		RANSAC & 9.94 & 1.00 & 0.97 & 4.45 & 21.32 & 10.05 & 5.35 \\
		\hline 
		DPD-lasso $\alpha = $ 0.1 & 11.99 & 1.00 & 0.96 & 4.15 & 4.55 & 18.44 & 5.29 \\ 
		DPD-lasso $\alpha = $ 0.3 & 15.65 & 1.00 & 0.95 & 17.90 & 294.60 & 20.86 & 6.62 \\ 
		DPD-lasso  $\alpha = $ 0.5 & 10.89 & 0.99 & 0.97 & 66.82 & 543.25 & 22.16 & 9.18 \\ 
		DPD-lasso  $\alpha = $ 0.7 & 5.73 & 1.00 & 1.00 & 7.46 & 1.45 & 26.35 & 6.46 \\ 
		DPD-lasso  $\alpha = $ 1 & 10.17 & 0.98 & 0.97 & 54.19 & 186.23 & 24.47 & 8.44 \\
		\hline
		LDPD-lasso $\alpha = $ 0.1 & 5.35 & 1.00 & 1.00 & 6.81 & 0.75 & 27.71 & 6.51  \\ 
		LDPD-lasso $\alpha = $ 0.2 & 5.36 & 1.00 & 1.00 & 6.75 & 0.79 & 27.22 & 6.43  \\ 
		LDPD-lasso $\alpha = $ 0.3 & 5.34 & 1.00 & 1.00 & 7.42 & 0.77 & 28.89 & 6.56  \\ 
		LDPD-lasso $\alpha = $ 0.4 & 5.31 & 1.00 & 1.00 & 10.18 & 0.84 & 34.52 & 7.25  \\ 
		LDPD-lasso $\alpha = $ 0.5 & 5.21 & 0.99 & 1.00 & 19.87 & 1.06 & 50.48 & 8.96 \\ 
		\hline
		DPD-ncv $\alpha = $ 0.1 & 5.00 & 1.00 & 1.00 & 0.38 & 0.00 & 4.67 & 3.98 \\ 
		DPD-ncv $\alpha = $ 0.3 & 5.00 & 1.00 & 1.00 & 0.53 & 0.00 & 8.62 & 3.99 \\ 
		DPD-ncv $\alpha = $ 0.5 & 5.00 & 1.00 & 1.00 & 0.80 & 0.00 & 11.95 & 4.06  \\ 
		DPD-ncv $\alpha = $ 0.7 & 5.00 & 1.00 & 1.00 & 0.97 & 0.00 & 14.53 & 4.13 \\ 
		DPD-ncv $\alpha = $ 1 & 4.98 & 1.00 & 1.00 & 1.42 & 0.00 & 17.56 & 4.23  \\ 
		\hline
		MNPRPE-ncv $\alpha = $ 0.1 & 5.04 & 1.00 & 1.00 & 0.34 & 0.06 & 3.18 & 3.97 \\ 
		MNPRPE-ncv $\alpha = $ 0.2 & 5.03 & 1.00 & 1.00 & 0.35 & 0.03 & 3.33 & 3.98  \\ 
		MNPRPE-ncv $\alpha = $ 0.3 & 5.01 & 1.00 & 1.00 & 0.36 & 0.00 & 3.47 & 3.98 \\ 
		MNPRPE-ncv $\alpha = $ 0.4 & 5.00 & 1.00 & 1.00 & 0.38 & 0.00 & 3.66 & 3.98\\ 
		MNPRPE-ncv $\alpha = $ 0.5 & 5.00 & 1.00 & 1.00 & 0.40 & 0.00 & 3.91 & 3.97 \\ 
		\hline
	\end{tabular}
	\label{p200e0signal1}
\end{table}
\begin{table}[H]
	\centering
	\caption{Performance measures obtained by different methods for $p=200$, weak signal and no outliers}
	\begin{tabular}{l|rrr|rrr|r}
		\hline
		Method	& MS($\widehat{\boldsymbol{\beta}}$) & TP($\widehat{\boldsymbol{\beta}}$) & TN($\widehat{\boldsymbol{\beta}}$) & MSES($\widehat{\boldsymbol{\beta}}$) &  MSES($\widehat{\boldsymbol{\beta}}$)  &  EE($\widehat{\sigma}$) & APrB($\widehat{\boldsymbol{\beta}}$)  \\ 
		& & & &  $(10^{-2})$ & $(10^{-5})$ &  $(10^{-2})$ &  $(10^{-2})$ \\
		\hline
		LS-LASSO & 8.53 & 1.00 & 0.98 & 1.85 & 2.85 & 30.20 & 4.59 \\ 
		LS-SCAD & 7.79 & 1.00 & 0.99 & 0.33 & 1.14 & 18.36 & 4.04\\ 
		LS-MCP & 6.22 & 1.00 & 0.99 & 0.33 & 1.66 & 19.39 & 4.10  \\ 
		LAD-Lasso & 6.28 & 1.00 & 0.99 & 3.56 & 2.34 & 37.63 & 5.29  \\ 
		RLARS & 11.06 & 1.00 & 0.97 & 0.47 & 28.64 & 9.93 & 4.24  \\ 
		sLTS & 27.42 & 1.00 & 0.89 & 5.05 & 31.63 & 16.90 & 5.17  \\ 
		RANSAC & 13.67 & 1.00 & 0.96 & 3.16 & 40.91 & 20.92 & 4.75  \\ 
		\hline
		DPD-lasso $\alpha = $ 0.1 & 11.17 & 1.00 & 0.97 & 4.45 & 4.14 & 19.24 & 5.51 \\ 
		DPD-lasso $\alpha = $ 0.3 & 12.87 & 1.00 & 0.96 & 4.38 & 5.40 & 19.52 & 5.58 \\ 
		DPD-lasso $\alpha = $ 0.5 & 10.21 & 1.00 & 0.97 & 5.44 & 4.68 & 22.36 & 5.80 \\ 
		DPD-lasso $\alpha = $ 0.7 & 5.52 & 1.00 & 1.00 & 7.60 & 1.32 & 26.98 & 6.56 \\ 
		DPD-lasso  $\alpha = $ 1 & 8.47 & 0.96 & 0.98 & 15.67 & 22.67 & 23.43 & 8.33 \\ 
		\hline
		LDPD-lasso $\alpha = $ 0.1 & 5.34 & 1.00 & 1.00 & 6.81 & 0.75 & 27.71 & 6.51  \\ 
		LDPD-lasso $\alpha = $ 0.2 & 5.34 & 1.00 & 1.00 & 6.76 & 0.78 & 27.27 & 6.43  \\ 
		LDPD-lasso $\alpha = $ 0.3 & 5.33 & 1.00 & 1.00 & 7.44 & 0.76 & 28.97 & 6.56  \\ 
		LDPD-lasso $\alpha = $ 0.4 & 5.29 & 1.00 & 1.00 & 10.10 & 0.83 & 34.45 & 7.23  \\ 
		LDPD-lasso $\alpha = $ 0.5 & 5.18 & 0.98 & 1.00 & 17.53 & 1.03 & 48.27 & 8.68  \\ 
		\hline
		DPD-ncv $\alpha = $ 0.1 & 5.11 & 1.00 & 1.00 & 0.83 & 0.14 & 4.79 & 3.88  \\ 
		DPD-ncv $\alpha = $ 0.3 & 5.12 & 1.00 & 1.00 & 0.88 & 0.20 & 9.57 & 3.89  \\ 
		DPD-ncv $\alpha = $ 0.5 & 5.07 & 0.99 & 1.00 & 1.02 & 0.19 & 13.22 & 3.88  \\ 
		DPD-ncv $\alpha = $ 0.7 & 5.02 & 0.99 & 1.00 & 1.37 & 0.17 & 16.03 & 4.02 \\ 
		DPD-ncv $\alpha = $ 1 & 4.92 & 0.97 & 1.00 & 2.72 & 0.09 & 19.29 & 4.35  \\ 
		\hline
		MNPRPE-ncv $\alpha = $ 0.1 & 5.21 & 1.00 & 1.00 & 0.54 & 0.21 & 3.22 & 3.92 \\ 
		MNPRPE-ncv $\alpha = $ 0.2 & 5.13 & 1.00 & 1.00 & 0.56 & 0.18 & 3.47 & 3.91 \\ 
		MNPRPE-ncv $\alpha = $ 0.3 & 5.09 & 1.00 & 1.00 & 0.54 & 0.16 & 3.66 & 3.93 \\ 
		MNPRPE-ncv $\alpha = $ 0.4 & 5.07 & 1.00 & 1.00 & 0.56 & 0.16 & 3.79 & 3.92 \\ 
		MNPRPE-ncv $\alpha = $ 0.5 & 5.05 & 1.00 & 1.00 & 0.59 & 0.12 & 3.96 & 3.95\\ 
		\hline
	\end{tabular}
	\label{p200e0signal0}
\end{table}

\begin{table}[H]
	\centering
	\caption{Performance measures obtained by different methods for $p=200$, strong signal and $Y-$outliers}
	\begin{tabular}{l|rrr|rrr|r}
		\hline
		Method	& MS($\widehat{\boldsymbol{\beta}}$) & TP($\widehat{\boldsymbol{\beta}}$) & TN($\widehat{\boldsymbol{\beta}}$) & MSES($\widehat{\boldsymbol{\beta}}$) &  MSES($\widehat{\boldsymbol{\beta}}$)  &  EE($\widehat{\sigma}$) & APrB($\widehat{\boldsymbol{\beta}}$)  \\ 
		& & & &  $(10^{-2})$ & $(10^{-5})$ &  $(10^{-2})$ &  $(10^{-2})$ \\
		\hline
		LS-LASSO & 6.91 & 1.00 & 0.99 & 2.12 & 1.68 & 33.97 & 4.72  \\ 
		LS-SCAD & 4.94 & 0.99 & 1.00 & 16.06 & 0.00 & 68.46 & 7.53  \\ 
		LS-MCP & 4.94 & 0.99 & 1.00 & 11.14 & 0.00 & 52.61 & 6.69  \\ 
		LAD-Lasso & 6.29 & 1.00 & 0.99 & 3.68 & 2.27 & 37.69 & 5.35  \\ 
		RLARS & 7.86 & 0.99 & 0.99 & 1.72 & 10.18 & 7.61 & 4.53  \\ 
		sLTS & 6.06 & 1.00 & 0.99 & 6.61 & 2.19 & 24.14 & 6.23  \\ 
		RANSAC & 9.61 & 1.00 & 0.98 & 4.27 & 21.24 & 13.46 & 5.74  \\ 
		\hline
		DPD-lasso $\alpha = $ 0.1 & 12.81 & 1.00 & 0.96 & 4.05 & 5.58 & 17.46 & 5.64 \\ 
		DPD-lasso $\alpha = $ 0.3 & 14.81 & 1.00 & 0.95 & 3.79 & 8.05 & 18.64 & 5.00 \\ 
		DPD-lasso $\alpha = $ 0.5 & 12.89 & 1.00 & 0.96 & 71.87 & 866.25 & 21.71 & 8.40 \\ 
		DPD-lasso $\alpha = $ 0.7 & 5.88 & 1.00 & 1.00 & 10.28 & 9.21 & 27.17 & 6.72 \\ 
		DPD-lasso $\alpha = $ 1 & 9.28 & 0.98 & 0.98 & 64.87 & 382.50 & 23.48 & 9.65 \\ 
		\hline
		LDPD-lasso $\alpha = $ 0.1 & 5.30 & 1.00 & 1.00 & 11.87 & 1.13 & 33.22 & 7.18  \\ 
		LDPD-lasso $\alpha = $ 0.2 & 5.33 & 1.00 & 1.00 & 7.13 & 1.19 & 27.76 & 6.70  \\ 
		LDPD-lasso $\alpha = $ 0.3 & 5.34 & 1.00 & 1.00 & 7.71 & 1.23 & 28.92 & 6.82  \\ 
		LDPD-lasso $\alpha = $ 0.4 & 5.27 & 1.00 & 1.00 & 10.83 & 1.36 & 35.09 & 7.61  \\ 
		LDPD-lasso $\alpha = $ 0.5 & 5.15 & 0.97 & 1.00 & 51.08 & 1.58 & 62.24 & 10.31  \\ 
		\hline
		DPD-ncv $\alpha = $ 0.1 & 5.01 & 1.00 & 1.00 & 0.89 & 0.04 & 6.49 & 4.27  \\ 
		DPD-ncv $\alpha = $ 0.3 & 5.00 & 1.00 & 1.00 & 1.02 & 0.00 & 6.03 & 4.24  \\ 
		DPD-ncv $\alpha = $ 0.5 & 5.00 & 1.00 & 1.00 & 1.21 & 0.00 & 7.42 & 4.29  \\ 
		DPD-ncv $\alpha = $ 0.7 & 5.00 & 1.00 & 1.00 & 1.55 & 0.00 & 8.58 & 4.44 \\ 
		DPD-ncv $\alpha = $ 1 & 5.00 & 1.00 & 1.00 & 2.44 & 0.00 & 9.80 & 4.72  \\ 
		\hline
		MNPRPE-ncv $\alpha = $ 0.1 & 5.03 & 1.00 & 1.00 & 0.36 & 0.05 & 3.24 & 4.00 \\ 
		MNPRPE-ncv $\alpha = $ 0.2 & 5.02 & 1.00 & 1.00 & 0.37 & 0.04 & 3.36 & 3.99 \\ 
		MNPRPE-ncv $\alpha = $ 0.3 & 5.00 & 1.00 & 1.00 & 0.38 & 0.00 & 3.54 & 3.99 \\ 
		MNPRPE-ncv $\alpha = $ 0.4 & 5.00 & 1.00 & 1.00 & 0.40 & 0.00 & 3.74 & 3.99 \\ 
		MNPRPE-ncv $\alpha = $ 0.5 & 5.00 & 1.00 & 1.00 & 0.42 & 0.00 & 3.98 & 3.99  \\ 
		\hline
	\end{tabular}
	\label{p200e1signal1}
\end{table}
\begin{table}[H]
	\centering
	\caption{Performance measures obtained by different methods for $p=200$, weak signal and $Y-$outliers}
	\begin{tabular}{l|rrr|rrr|r}
		\hline
		Method	& MS($\widehat{\boldsymbol{\beta}}$) & TP($\widehat{\boldsymbol{\beta}}$) & TN($\widehat{\boldsymbol{\beta}}$) & MSES($\widehat{\boldsymbol{\beta}}$) &  MSES($\widehat{\boldsymbol{\beta}}$)  &  EE($\widehat{\sigma}$) & APrB($\widehat{\boldsymbol{\beta}}$)  \\ 
		& & & &  $(10^{-2})$ & $(10^{-5})$ &  $(10^{-2})$ &  $(10^{-2})$ \\
		\hline
		LS-LASSO & 8.52 & 1.00 & 0.98 & 1.84 & 2.74 & 30.25 & 4.57  \\ 
		LS-SCAD & 7.70 & 1.00 & 0.99 & 0.34 & 1.11 & 18.49 & 4.02  \\ 
		LS-MCP & 6.24 & 1.00 & 0.99 & 0.32 & 1.66 & 19.38 & 4.10 \\ 
		LAD-Lasso & 6.27 & 1.00 & 0.99 & 3.58 & 2.31 & 37.67 & 5.31  \\ 
		RLARS & 11.06 & 1.00 & 0.97 & 0.47 & 28.64 & 9.93 & 4.24  \\ 
		sLTS & 27.42 & 1.00 & 0.89 & 5.05 & 31.63 & 16.90 & 5.17 \\ 
		RANSAC & 12.85 & 1.00 & 0.96 & 3.13 & 37.87 & 19.90 & 5.34  \\ 
		\hline
		DPD-lasso $\alpha = $ 0.1 & 12.21 & 1.00 & 0.96 & 4.29 & 5.14 & 18.19 & 5.80 \\ 
		DPD-lasso $\alpha = $ 0.3 & 13.98 & 1.00 & 0.95 & 3.99 & 7.26 & 19.36 & 5.11 \\ 
		DPD-lasso $\alpha = $ 0.5 & 9.26 & 1.00 & 0.98 & 5.37 & 4.55 & 21.65 & 6.03 \\ 
		DPD-lasso $\alpha = $ 0.7 & 5.57 & 1.00 & 1.00 & 6.79 & 1.58 & 27.61 & 6.41 \\ 
		DPD-lasso $\alpha = $ 1 & 8.85 & 0.97 & 0.98 & 12.36 & 19.93 & 22.87 & 7.42 \\ 
		\hline
		LDPD-lasso $\alpha = $ 0.1 & 5.32 & 1.00 & 1.00 & 6.95 & 1.14 & 27.41 & 6.67  \\ 
		LDPD-lasso $\alpha = $ 0.2 & 5.32 & 1.00 & 1.00 & 7.13 & 1.19 & 27.76 & 6.71  \\ 
		LDPD-lasso $\alpha = $ 0.3 & 5.34 & 1.00 & 1.00 & 7.71 & 1.24 & 28.92 & 6.82  \\ 
		LDPD-lasso $\alpha = $ 0.4 & 5.27 & 1.00 & 1.00 & 10.93 & 1.35 & 35.20 & 7.63  \\ 
		LDPD-lasso $\alpha = $ 0.5 & 5.15 & 0.97 & 1.00 & 17.72 & 1.60 & 47.70 & 8.80  \\ 
		\hline
		DPD-ncv $\alpha = $ 0.1 & 5.25 & 1.00 & 1.00 & 0.69 & 0.35 & 4.22 & 3.98  \\ 
		DPD-ncv $\alpha = $ 0.3 & 5.16 & 1.00 & 1.00 & 0.71 & 0.37 & 6.62 & 3.99  \\ 
		DPD-ncv $\alpha = $ 0.5 & 5.09 & 0.99 & 1.00 & 0.83 & 0.32 & 8.85 & 4.01  \\ 
		DPD-ncv $\alpha = $ 0.7 & 5.04 & 0.99 & 1.00 & 0.96 & 0.31 & 10.82 & 4.00 \\ 
		DPD-ncv $\alpha = $ 1 & 4.99 & 0.99 & 1.00 & 1.14 & 0.15 & 13.36 & 4.00  \\ 
		\hline
		MNPRPE-ncv $\alpha = $ 0.1 & 5.17 & 0.99 & 1.00 & 0.87 & 0.14 & 3.34 & 3.91 \\ 
		MNPRPE-ncv $\alpha = $ 0.2 & 5.12 & 0.99 & 1.00 & 0.85 & 0.10 & 3.50 & 3.91 \\ 
		MNPRPE-ncv $\alpha = $ 0.3 & 5.05 & 0.99 & 1.00 & 1.09 & 0.06 & 4.11 & 4.00 \\ 
		MNPRPE-ncv $\alpha = $ 0.4 & 5.02 & 0.99 & 1.00 & 1.28 & 0.05 & 4.48 & 4.04 \\ 
		MNPRPE-ncv $\alpha = $ 0.5 & 4.92 & 0.98 & 1.00 & 1.46 & 0.00 & 4.87 & 4.10 \\ 
		\hline
	\end{tabular}
	\label{p200e1signal0}
\end{table}

\begin{table}[H]
	\centering
	\caption{Performance measures obtained by different methods for $p=200$, strong signal and $\boldsymbol{X}-$outliers}
	\begin{tabular}{l|rrr|rrr|r}
		\hline
		Method	& MS($\widehat{\boldsymbol{\beta}}$) & TP($\widehat{\boldsymbol{\beta}}$) & TN($\widehat{\boldsymbol{\beta}}$) & MSES($\widehat{\boldsymbol{\beta}}$) &  MSES($\widehat{\boldsymbol{\beta}}$)  &  EE($\widehat{\sigma}$) & APrB($\widehat{\boldsymbol{\beta}}$)  \\ 
		& & & &  $(10^{-2})$ & $(10^{-5})$ &  $(10^{-2})$ &  $(10^{-2})$ \\
		\hline
		LS-LASSO & 6.91 & 1.00 & 0.99 & 2.12 & 1.68 & 33.97 & 4.72  \\ 
		LS-SCAD & 4.94 & 0.99 & 1.00 & 16.06 & 0.00 & 68.46 & 7.53  \\ 
		LS-MCP & 4.94 & 0.99 & 1.00 & 11.14 & 0.00 & 52.61 & 6.69  \\ 
		LAD-Lasso & 6.29 & 1.00 & 0.99 & 3.68 & 2.27 & 37.69 & 5.35  \\ 
		RLARS & 7.86 & 0.99 & 0.99 & 1.72 & 10.18 & 7.61 & 4.53  \\ 
		sLTS & 6.06 & 1.00 & 0.99 & 6.61 & 2.19 & 24.14 & 6.23 \\ 
		RANSAC & 10.07 & 1.00 & 0.97 & 4.38 & 18.89 & 12.13 & 5.35\\ 
		\hline
		DPD-lasso  $\alpha = $ 0.1 & 5.31 & 1.00 & 1.00 & 6.93 & 0.74 & 28.16 & 6.50 \\ 
		DPD-lasso  $\alpha = $ 0.3 & 5.35 & 1.00 & 1.00 & 7.58 & 0.77 & 29.49 & 6.61 \\ 
		DPD-lasso  $\alpha = $ 0.5 & 5.21 & 0.99 & 1.00 & 19.94 & 1.05 & 51.10 & 8.99 \\ 
		DPD-lasso  $\alpha = $ 0.7 & 5.60 & 1.00 & 1.00 & 7.52 & 1.43 & 26.53 & 6.48 \\ 
		DPD-lasso  $\alpha = $ 1 & 8.01 & 0.99 & 0.98 & 21.71 & 58.84 & 24.48 & 7.52 \\ 
		\hline
		LDPD-lasso  $\alpha = $ 0.1 & 5.35 & 1.00 & 1.00 & 6.81 & 0.75 & 27.71 & 6.51 \\ 
		LDPD-lasso  $\alpha = $ 0.2 & 5.36 & 1.00 & 1.00 & 6.75 & 0.79 & 27.22 & 6.43  \\ 
		LDPD-lasso  $\alpha = $ 0.3 & 5.34 & 1.00 & 1.00 & 7.42 & 0.77 & 28.89 & 6.56 \\ 
		LDPD-lasso  $\alpha = $ 0.4 & 5.31 & 1.00 & 1.00 & 10.18 & 0.84 & 34.52 & 7.25  \\ 
		LDPD-lasso  $\alpha = $ 0.5 & 5.21 & 0.99 & 1.00 & 19.87 & 1.06 & 50.48 & 8.96  \\ 
		\hline
		DPD-ncv $\alpha = $ 0.1 & 5.01 & 1.00 & 1.00 & 0.35 & 0.02 & 4.63 & 3.93   \\ 
		DPD-ncv $\alpha = $ 0.3 & 5.00 & 1.00 & 1.00 & 0.41 & 0.00 & 8.44 & 3.94 \\ 
		DPD-ncv $\alpha = $ 0.5 & 5.00 & 1.00 & 1.00 & 0.51 & 0.00 & 11.67 & 3.98 \\ 
		DPD-ncv $\alpha = $ 0.7 & 5.00 & 1.00 & 1.00 & 0.61 & 0.00 & 14.26 & 4.03 \\ 
		DPD-ncv $\alpha = $ 1 & 5.00 & 1.00 & 1.00 & 0.83 & 0.00 & 17.31 & 4.11\\ 
		\hline
		MNPRPE-ncv $\alpha = $ 0.1 & 5.04 & 1.00 & 1.00 & 0.34 & 0.06 & 3.18 & 3.97  \\ 
		MNPRPE-ncv $\alpha = $ 0.2 & 5.03 & 1.00 & 1.00 & 0.35 & 0.03 & 3.33 & 3.98  \\ 
		MNPRPE-ncv $\alpha = $ 0.3 & 5.01 & 1.00 & 1.00 & 0.36 & 0.00 & 3.47 & 3.98  \\ 
		MNPRPE-ncv $\alpha = $ 0.4 & 5.00 & 1.00 & 1.00 & 0.38 & 0.00 & 3.66 & 3.98 \\ 
		MNPRPE-ncv $\alpha = $ 0.5 & 5.00 & 1.00 & 1.00 & 0.40 & 0.00 & 3.91 & 3.97 \\ 
		\hline
	\end{tabular}
	\label{p200e1signal1x1}
\end{table}

\begin{table}[H]
	\centering
	\caption{Performance measures obtained by different methods for $p=200$, weak signal and $\boldsymbol{X}-$outliers}
	\begin{tabular}{l|rrr|rrr|r}
		\hline
		Method	& MS($\widehat{\boldsymbol{\beta}}$) & TP($\widehat{\boldsymbol{\beta}}$) & TN($\widehat{\boldsymbol{\beta}}$) & MSES($\widehat{\boldsymbol{\beta}}$) &  MSES($\widehat{\boldsymbol{\beta}}$)  &  EE($\widehat{\sigma}$) & APrB($\widehat{\boldsymbol{\beta}}$)  \\ 
		& & & &  $(10^{-2})$ & $(10^{-5})$ &  $(10^{-2})$ &  $(10^{-2})$ \\
		\hline
		LS-LASSO & 8.52 & 1.00 & 0.98 & 1.84 & 2.74 & 30.25 & 4.57 \\ 
		LS-SCAD & 7.70 & 1.00 & 0.99 & 0.34 & 1.11 & 18.49 & 4.02  \\ 
		LS-MCP & 6.24 & 1.00 & 0.99 & 0.32 & 1.66 & 19.38 & 4.10 \\ 
		LAD-Lasso & 6.27 & 1.00 & 0.99 & 3.58 & 2.31 & 37.67 & 5.31  \\ 
		RLARS & 11.06 & 1.00 & 0.97 & 0.47 & 28.64 & 9.93 & 4.24 \\ 
		sLTS & 27.42 & 1.00 & 0.89 & 5.05 & 31.63 & 16.90 & 5.17 \\ 
		RANSAC & 13.68 & 1.00 & 0.96 & 3.16 & 39.58 & 21.22 & 5.03 \\ 
		\hline
		DPD-lasso $\alpha = $ 0.1 & 5.31 & 1.00 & 1.00 & 6.93 & 0.74 & 28.16 & 6.50 \\ 
		DPD-lasso $\alpha = $ 0.3 & 5.34 & 1.00 & 1.00 & 7.62 & 0.77 & 29.62 & 6.63 \\ 
		DPD-lasso $\alpha = $ 0.5 & 5.29 & 1.00 & 1.00 & 12.14 & 1.05 & 40.87 & 7.63 \\ 
		DPD-lasso  $\alpha = $ 0.7 & 5.52 & 1.00 & 1.00 & 7.60 & 1.32 & 26.98 & 6.56 \\ 
		DPD-lasso  $\alpha = $ 1 & 7.33 & 0.98 & 0.99 & 12.20 & 12.10 & 23.97 & 7.62 \\ 
		\hline
		LDPD-lasso  $\alpha = $ 0.1 & 5.34 & 1.00 & 1.00 & 6.81 & 0.75 & 27.71 & 6.51 \\ 
		LDPD-lasso  $\alpha = $ 0.2 & 5.34 & 1.00 & 1.00 & 6.76 & 0.78 & 27.27 & 6.43  \\ 
		LDPD-lasso  $\alpha = $ 0.3 & 5.33 & 1.00 & 1.00 & 7.44 & 0.76 & 28.97 & 6.56  \\ 
		LDPD-lasso  $\alpha = $ 0.4 & 5.29 & 1.00 & 1.00 & 10.10 & 0.83 & 34.45 & 7.23 \\ 
		LDPD-lasso  $\alpha = $ 0.5 & 5.18 & 0.98 & 1.00 & 17.53 & 1.03 & 48.27 & 8.68  \\ 
		\hline
		DPD-ncv $\alpha = $ 0.1 & 5.20 & 1.00 & 1.00 & 0.59 & 0.31 & 5.22 & 3.92  \\ 
		DPD-ncv $\alpha = $ 0.3 & 5.21 & 1.00 & 1.00 & 0.49 & 0.53 & 10.09 & 3.92\\ 
		DPD-ncv $\alpha = $ 0.5 & 5.17 & 1.00 & 1.00 & 0.58 & 0.62 & 13.66 & 3.93  \\ 
		DPD-ncv $\alpha = $ 0.7 & 5.15 & 0.99 & 1.00 & 0.78 & 0.88 & 16.43 & 4.02  \\ 
		DPD-ncv $\alpha = $ 1 & 5.10 & 0.99 & 1.00 & 1.23 & 0.93 & 19.60 & 4.22  \\ 
		\hline
		MNPRPE-ncv $\alpha = $ 0.1 & 5.21 & 1.00 & 1.00 & 0.54 & 0.21 & 3.22 & 3.92  \\ 
		MNPRPE-ncv $\alpha = $ 0.2 & 5.13 & 1.00 & 1.00 & 0.56 & 0.18 & 3.47 & 3.90  \\ 
		MNPRPE-ncv $\alpha = $ 0.3 & 5.09 & 1.00 & 1.00 & 0.54 & 0.17 & 3.66 & 3.93  \\ 
		MNPRPE-ncv $\alpha = $ 0.4 & 5.07 & 1.00 & 1.00 & 0.56 & 0.16 & 3.80 & 3.92 \\ 
		MNPRPE-ncv $\alpha = $ 0.5 & 5.05 & 1.00 & 1.00 & 0.59 & 0.12 & 3.97 & 3.95 \\ 
		\hline
	\end{tabular}
	\label{p200e1signal0x1}
\end{table}

\subsection{Example R Code for computation of the MNPRPE}

The present R code is provided to help the reader to implement the MNPRPE. This code has been used to obtain the results in the simulation study in Section 6 and to fit the model in the numerical example of glioblastoma gene expression analysis studied in Section 8 of the main paper. The code is inspired from the \textit{Robust and Sparse Regression via Gamma-Divergence} (\textit{gamreg}) package, created by Takayuki Kawashima (2017), and it uses \textit{ncvreg} (Breheny and Huang (2011)) and \textit{rqPen} (Sherwood and Maidman (2020)) packages.

\begin{lstlisting}[language=R]
#alpha: tuning parameter in the Renyi function
#penalty: 1 for SCAD, 2 for MCP
#lambda : Regularization parameter in the penalty function

library(ncvreg)
library(rqPen)
pr_ncv<- function(X, Y, beta,beta0, sigma, lambda,a lpha, inter,penalty){

#if every coeff on init beta is zero, stop
if(all(beta==0)){
stop("null beta init")
}

N = dim(X)[1]
p = dim(X)[2]

#create intercept term
tmp = rep(1, N)
tmp1 = inter*beta0*tmp
tmp2 = X%*%beta #X matrix doesnt contain ones column 
tmp3 = drop(tmp1 + tmp2) #y estimate

for (m in 1:5000){
#temporary copies 
beta0_tmp = beta0*inter
beta_tmp = beta
sigma_tmp = sigma
#weight
mu = exp(-(alpha/2)*((Y-tmp3)/sigma)^2)
mu = drop(mu/sum(mu))
#update beta0 
beta0 = drop(t(mu)%*%(Y - tmp2)*inter)
tmp1 = beta0*tmp
#weigthed matrices
Y_w = diag(sqrt(mu))%*%((Y-tmp1)/sigma) #Y contains incercept (mean)
X_w =  diag(sqrt(mu))%*%(X/sigma) #without intercept

#solve using coordinate descent now	
if (penalty==1){
estimate = ncvreg(X_w, Y_w, family="gaussian",
penalty="SCAD", lambda = lambda)
}
else{estimate = ncvreg(X_w, Y_w, family="gaussian",
penalty="MCP", lambda = lambda)}	

beta = estimate$beta[1:p+1]

#update tmp3 with new beta
tmp2 = X%*%beta 
tmp3 = drop(tmp1 + tmp2) 

#update sigma now
sigma = drop(sqrt((1+alpha)*t(mu)%*%((Y-tmp3)^2)))

#stopping criteria
#I put N*penalization to ensure enough penalization
if (penalty==1){
if(all(beta==0) |
abs((-alpha/(1+alpha))*log(sigma_tmp) -(-alpha/(1+alpha))*log(sigma)  
+ (-1/alpha)*log(sum(exp(-alpha*(Y-beta0_tmp*tmp-X%*%beta_tmp)^2/(2*sigma_tmp^2))*(2*pi*sigma_tmp^2)^(-alpha/2))) 
+ N*sum(scad(beta_tmp,lambda)) - N*sum(scad(beta,lambda))
-(-1/alpha)*log(sum( exp(-alpha*(Y-tmp3)^2/(2*sigma^2))*(2*pi*sigma^2)^(-alpha/2))))<= 1e-9 ){
break
} }
else{
if(all(beta==0) |
abs((-alpha/(1+alpha))*log(sigma_tmp) -(-alpha/(1+alpha))*log(sigma)  
+ (-1/alpha)*log(sum(exp(-alpha*(Y-beta0_tmp*tmp-X%*%beta_tmp)^2/(2*sigma_tmp^2))*(2*pi*sigma_tmp^2)^(-alpha/2))) 
+ N*sum(mcp(beta_tmp,lambda)) - N*sum(mcp(beta,lambda))
-(-1/alpha)*log(sum( exp(-alpha*(Y-tmp3)^2/(2*sigma^2))*(2*pi*sigma^2)^(-alpha/2))))<= 1e-9 ){
break
} 
}
} 
return(list("beta0"=beta0,"beta" = beta, "sigma"= sigma))
}
\end{lstlisting}

\end{document}